\documentclass{amsart}

\usepackage[left=1.25in, right=1.25in, top=1.5in, bottom=1.5in]{geometry}
\usepackage{amsmath, amssymb, amsthm, tikz-cd}
\usepackage{mathrsfs}

\usetikzlibrary{fit}
\usepackage{mathtools}
\usepackage{placeins}
\usepackage{longtable}

\usepackage[colorlinks=true, linkcolor=black, citecolor=black, urlcolor=black]{hyperref}
\usepackage{cite}

\newcommand{\DomD}{\mathrm{Dom}_\Delta}
\newcommand{\Dom}{\mathrm{Dom}}
\newcommand{\Dtor}{D^b_c(R)_{\mathrm{tor}}}
\newcommand{\rkFG}{\operatorname{rk}_{\mathsf{FG}}}
\newcommand{\degFG}{\operatorname{deg}_{\mathsf{FG}}}
\newcommand{\muFG}{\mu_{\mathsf{FG}}}
\DeclareMathOperator{\fib}{fib}
\newcommand{\Domss}{\mathrm{Dom}^{ss}_\Delta}
\newcommand{\FilD}{\mathrm{Fil}_\Delta}
\newcommand{\rkD}{\operatorname{rk}_\Delta}
\newcommand{\degD}{\operatorname{deg}_\Delta}
\newcommand{\muD}{\mu_\Delta}
\newcommand{\HN}{\mathrm{HN}}
\newcommand{\fgpair}[2]{%
  \mathrel{\mathop{\rightleftarrows}^{\scriptstyle #1}_{\scriptstyle #2}}}
\DeclareMathOperator{\Coh}{Coh}
\DeclareMathOperator{\Perf}{Perf}

\newtheorem{theorem}{Theorem}[section]
\newtheorem{proposition}[theorem]{Proposition}
\newtheorem{lemma}[theorem]{Lemma}
\newtheorem{corollary}[theorem]{Corollary}

\makeatletter
\let\c@equation\c@theorem

\makeatother

\theoremstyle{definition}
\newtheorem{definition}[theorem]{Definition}
\newtheorem{example}[theorem]{Example}

\theoremstyle{remark}
\newtheorem{remark}[theorem]{Remark}

\begin{document}

\title{Harder--Narasimhan filtration of torsion $F$-gauges}
\author{Yuanning Zhang}
\address{Department of Mathematics, Northwestern University}
\email{yuanningzhang2026@u.northwestern.edu}
\date{}

\begin{abstract}
We establish a Harder--Narasimhan theory for torsion coherent
$F$-gauges over a perfect field $k$ of positive characteristic and
completely classify its stable objects.  Our construction uses
Ekedahl's derived equivalence with coherent complexes over the
Raynaud ring, where the theory refines his half-integral type
filtration on diagonal dominoes.  At each finite rational slope
$d/q$ in lowest terms, there is a unique stable diagonal domino
up to Breuil--Kisin twist, constructed explicitly from the
corresponding Christoffel word.  The semistable category of this
slope is equivalent to a product of $q$ copies of the zero-slope
category.  The classification of stable objects corrects
Ekedahl's classification of weakly simple diagonal dominoes of
type $\tfrac12$.  The theory upgrades to a locally
finite Bridgeland stability condition on
$\Perf_{\mathrm{tor}}(k^{\mathrm{Syn}})$.
\end{abstract}

\maketitle

\tableofcontents

\section{Introduction}
\label{sec:intro}

Let $k$ be a perfect field of characteristic $p$, and let
$W:=W(k)$ be its ring of Witt vectors with Frobenius $\sigma$.  This paper develops a
Harder--Narasimhan theory for the torsion coherent $F$-gauges over $k$, and
relates it to a filtration Ekedahl constructed forty years ago \cite{ekedahl3}.
The strategy of this paper is to take a detour: we first study the
Harder--Narasimhan formalism on another category
$\Delta_{\mathrm{tor}}$, and then return to $F$-gauges by a
torsion-pair tilt in the sense of \cite{HRS96}.

Let $R$ be the Raynaud ring of Illusie--Raynaud \cite{IR83}, the graded ring
generated by $F$, $V$ and the de Rham--Witt differential $d$, so that the de
Rham--Witt complex of a smooth proper $X/k$ is a single object
$R\Gamma(X,W\Omega_X^\bullet)$ of the bounded coherent derived category
$D^b_c(R)$, whose totalization $\mathbf s$ computes crystalline cohomology.
Dominoes are the parts of Hodge--Witt cohomology that fail to be finitely
generated over $W$.  By the Illusie--Raynaud d\'evissage they carry exactly the
nonzero differentials of the slope spectral sequence, while the finitely
generated part survives to $E_\infty$ \cite[\S I.2]{IR83}.

In \cite{ekedahl3}, Ekedahl first proved a derived equivalence
$S\colon D^b_c(R)\xrightarrow{\ \sim\ }D^b_c(F\text{-gauges})$, and the latter
category is $\Perf(k^{\mathrm{Syn}})$ in the modern language of
\cite{BhattFgauges}.  This equivalence does not restrict to an
equivalence of abelian hearts: the Postnikov $t$-structure on
$D^b_c(R)$ has heart the coherent $R$-modules, and the standard
$t$-structure on $\Perf(k^{\mathrm{Syn}})$ has heart
$\Coh(k^{\mathrm{Syn}})$.  Ekedahl defines a third
\emph{diagonal $t$-structure} on $D^b_c(R)$, whose heart $\Delta$
consists of what he calls \emph{diagonal complexes}.  The diagonal
complexes killed by a power of $p$ form a Serre subcategory
$\Delta_{\mathrm{tor}}$, which contains both the finite-length
Dieudonn\'e modules and the dominoes.  He defines the
\emph{diagonal dominoes} $\DomD$ as the subcategory of
$\Delta_{\mathrm{tor}}$ consisting of the objects with no nonzero
finite-length Dieudonn\'e subobject, and shows that this
definition generalizes dominoes.  He then defines a
$\tfrac12\mathbb Z$-indexed filtration $\FilD$ on $\DomD$ that
extends the $\mathbb Z$-indexed type filtration on dominoes.

The starting point of this paper is a conceptual interpretation
of $\FilD$.  Two of its features remain mysterious in Ekedahl's
treatment: (1) the objects of pure integer type form an abelian
category, while the objects of type in $\tfrac12+\mathbb Z$ do
not; (2) Ekedahl classified the weakly simple objects of type
$\tfrac12$ into two families, which he called $N_{m,n}$ and $N_m$
\cite[Theorem~III.2.7]{ekedahl3}, by a case analysis that leaves
the shape of the list unexplained.  His proof suggested
to us that a notion of slope on $\DomD$ accounts for both features,
and this is the case.

For $M\in\Delta_{\mathrm{tor}}$ let $\rkD(M)$ be the sum of Ekedahl's
domino numbers $T^{i,-i}(M)$ and let $\degD(M)=\chi(\mathbf s(M))$,
where $\mathbf s\colon D(R)\to D(W)$ is the totalization functor
and $\chi$ is the Euler characteristic of a complex of $W$-modules
with finite-length cohomology.  The slope is $\muD=\degD/\rkD$,
with the value $+\infty$ at rank zero.
Both invariants are additive.  The rank vanishes exactly on the
finite direct sums of Breuil--Kisin twists of finite-length
Dieudonn\'e modules, so
classical Dieudonn\'e theory occupies the locus $\muD=+\infty$, and the
part of finite rational slope corresponds to the diagonal dominoes.

\medskip
\noindent\textbf{Theorem A} (Harder--Narasimhan filtration on
$\Delta_{\mathrm{tor}}$: Theorems~\ref{thm:extended-hn}
and~\ref{thm:refinement})\textbf{.}
\emph{Every nonzero $M\in\Delta_{\mathrm{tor}}$ admits a unique
Harder--Narasimhan filtration for $\muD$.  The slope $+\infty$ occurs
exactly when the maximal finite-torsion subobject $M_{\mathrm{ft}}$
is nonzero.  On a
diagonal domino, $\FilD^rM=F^{\ge r}_{\HN}M$ and
$\FilD^{r+1/2}M=F^{>r}_{\HN}M$ for every $r\in\mathbb Z$.}

\medskip
For $r\in\mathbb Z$, the objects
of pure type $r$ form the abelian category of semistable objects of
slope $r$, while the objects of pure type $r+\tfrac12$ form the
category of objects whose HN slopes lie in the open interval
$(r,r+1)$.  Ekedahl's classification of weakly simple objects then
corresponds to the stable objects of slopes $1/q$ and $(q-1)/q$ up to
Breuil--Kisin twist (Theorem~D below).  Restricting to a single rational slope gives an abelian category
(Theorem~\ref{thm:fixed-slope}); its structure is described in
Theorem~C below.

To translate this result to the $F$-gauge side, we naturally extend
the HN formalism on $\Delta_{\mathrm{tor}}$ to a Bridgeland
stability condition $\sigma_\Delta=(Z_\Delta,\mathcal P_\Delta)$
on $\Dtor$, where $Z_\Delta=-\degD+i\,\rkD$ is the central charge
and $\mathcal P_\Delta$ is the slicing by its phase, to be recalled
in Section~\ref{subsec:transport}.  We then use Ekedahl's derived
equivalence
$S\colon D^b_c(R)\xrightarrow{\sim}\Perf(k^{\mathrm{Syn}})$
together with his identification of the $F$-gauge heart
$\Coh(k^{\mathrm{Syn}})$ on the left hand side, denoted $\mathcal G$
\cite[Ch.~II]{ekedahl3}.  Thus both the diagonal heart $\Delta$ and
$\mathcal G$ are abelian subcategories of $D^b_c(R)$.  The next
result says that they are related by a tilt at a torsion pair in
the sense of Happel--Reiten--Smal\o{} \cite{HRS96}.

\medskip
\noindent\textbf{Theorem B}
(Theorem~\ref{thm:gauge-heart-slope-zero})\textbf{.}
\emph{Under $S$, the torsion $F$-gauge heart
$\Coh_{\mathrm{tor}}(k^{\mathrm{Syn}})$ is the HRS tilt
$\bigl\langle\Delta_{\mathrm{tor}}^{\HN\le0}[1],
\Delta_{\mathrm{tor}}^{\HN>0}\bigr\rangle=\mathcal P_\Delta((1/2,3/2])$,
the rotation of $\Delta_{\mathrm{tor}}$ through half a phase.  Its stable
objects are the positive-slope $\Delta$-stable objects together with the
shifts by $[1]$ of the nonpositive-slope ones.}

\medskip
The rank, degree, and slope on the $F$-gauge side are intrinsic: a
coherent torsion $F$-gauge $G$ has
\[
 \rkFG(G)=\ell_W(G^{-\infty}),\qquad
 \degFG(G)=\sum_i\chi_W\bigl(\fib(t^{-\infty}\colon
 G^i\to G^{-\infty})\bigr),
\]
and $\muFG:=\degFG/\rkFG$.  Under the
rotation the slope inverts: on the intersection
$\mathcal P_\Delta((1/2,1])$ of the two hearts, we have
$\muFG(S(M))=-1/\muD(M)$
(Proposition~\ref{prop:gauge-euler-invariants}).  In particular the
finite-torsion objects of classical Dieudonn\'e theory, of slope
$+\infty$ on the diagonal side, become the $F$-gauges of
slope~$0$.

\medskip
We then study the stable objects of a given rational slope
$\lambda=d/q$, and find that they are exactly the Breuil--Kisin
twists of a single diagonal domino $U_{d/q}$, which we call the
Christoffel domino of slope $d/q$, whose name is inspired by the
combinatorics of Christoffel words.

\medskip
\noindent\textbf{Theorem C} (Stable objects and fixed-slope categories:
Theorems~\ref{thm:stable-classification},
\ref{thm:integer-slope-category}, and~\ref{thm:rational-slope-category})\textbf{.}
\emph{For $\lambda\in\mathbb Q$, write uniquely
$\lambda=r+d/q$ with $r\in\mathbb Z$, $0<d\le q$, and $\gcd(d,q)=1$.}
\begin{enumerate}
\item[(i)] \emph{Up to Breuil--Kisin twist, the unique stable object of slope $\lambda$ is
$U_{d/q}(r\mathbf1)$, where $U_{1/1}=U_1$.  The stable objects of slope $+\infty$
are the Breuil--Kisin twists of the simple finite-length Dieudonn\'e
modules.}
\item[(ii)] \emph{There are exact $\mathbb Z_p$-linear equivalences
\[
 \Domss(\lambda)
 \simeq\prod_{a=0}^{q-1}\Domss(\lambda)_a
 \simeq\Domss(0)^{\,q},
\]
where $\Domss(\lambda)_a$ consists of objects whose composition
factors are $U_{d/q}(r\mathbf1)\{a+bq\}$, $b\in\mathbb Z$.
The first decomposition is canonical, by the residue of the
twist modulo $q$.}
\end{enumerate}

Here $\Domss(0)$ is equivalent to the category of graded
$W[u,t]/(ut-p)$-modules of finite total $W$-length,
with $\deg u=1$ and $\deg t=-1$
(Theorem~\ref{thm:integer-slope-category}).  Thus the equivalence
type of a finite rational slope category depends only on the
denominator of the slope.

\medskip
These Christoffel dominoes are described explicitly as torsion
$F$-gauges in the proof of Proposition~\ref{prop:chain-presentation}.
The key step of its proof presents the $F$-gauge as a module over a
semilinear string algebra and combines the string--band
classification of Bennett-Tennenhaus and Crawley-Boevey \cite{BTCB}
with the boundary condition, which excludes the bands.

\medskip
\noindent\textbf{Theorem D}
(Theorem~\ref{thm:weak-simplicity})\textbf{.}
\emph{Up to Breuil--Kisin twist, the weakly simple diagonal dominoes
of pure type $\tfrac12$ are exactly the Christoffel dominoes
$U_{1/q}$ and $U_{(q-1)/q}$.}

\medskip
This corrects Ekedahl's classification
\cite[Theorem~III.2.7]{ekedahl3}: the printed list stops the second
family at slope $2/3$, and the step of the proof that discards the
higher slopes asserts an extension forbidden by additivity of the
degree (Remark~\ref{rem:weak-simplicity-gap}).

In the last section, with $k$ algebraically closed, we apply the
theory to the superspecial abelian varieties $E^g$.  The middle
diagonal domino of the threefold is the stable object $U_{1/2}$
(Corollary~\ref{thm:superspecial-threefold}).  In dimension four
the middle diagonal domino is no longer semistable: it is an
extension of $U_{1/3}$ by seventeen copies of $U_1\{-1\}$, with Postnikov
layers $U_1$, $U_1^{\oplus16}\oplus U_{0,2}$ and $U_{-1}$
(Theorem~\ref{thm:superspecial-fourfold}).  The layer $U_{-1}$
gives a nonzero higher differential in the slope spectral sequence
of $E^4$ (Corollary~\ref{cor:no-e2-degeneration}), a special case
of Ekedahl's theorem that the slope spectral sequence of a
supersingular abelian fourfold does not degenerate at $E_2$
\cite{ekedahl1}.

We contrast this picture with the Dieudonn\'e--Manin
classification \cite{Manin63}.  Both theories carry $\mathbb Q$-indexed slopes,
but an isocrystal is the direct sum of its isoclinic pieces.  The
splitting is geometrized by the Fargues--Fontaine curve: the
Harder--Narasimhan filtration of every vector bundle on it splits,
and every semistable bundle is a direct sum of stable ones
\cite{FarguesFontaine18}.  In our category
$\Delta_{\mathrm{tor}}$ the Harder--Narasimhan filtration does not
split in general.  Every stable diagonal domino $U$ is nevertheless
rigid and satisfies $\operatorname{Ext}^2(U,U)=k$
(Proposition~\ref{prop:ext-stable}, \eqref{eq:ext-table}, and
Lemma~\ref{lem:nygaard-shift}), so
the fixed-slope categories behave like sheaves on a surface rather
than on a curve.  The comparison is more than an analogy: the two
categories
$\Perf_{\mathrm{tor}}(k^{\mathrm{Syn}})$ and
$\Perf(k^{\mathrm{Syn}})[1/p]$ sit at the opposite ends of one
localization sequence for $\Perf(k^{\mathrm{Syn}})$
(Remark~\ref{rem:torsion-isocrystal}), and $\Coh(k^{\mathrm{Syn}})$
has global dimension $2$ (Remark~\ref{rem:ext-comparison}).

In \cite{Fargues10}, Fargues constructed a Harder--Narasimhan
filtration for finite flat group schemes over a complete rank-one
valuation ring of mixed characteristic $(0,p)$, with height the
$p$-logarithm of the order and degree the valuation of the module of
invariant differentials, and Levin and Wang-Erickson extended it
to torsion Kisin modules $\mathfrak M$ of bounded height attached
to a finite extension $K/\mathbb Q_p$ \cite{LevinWangErickson20}.  Their rank is the length of the
generic fibre, in parallel with $\rkFG(G)=\ell_W(G^{-\infty})$.
Their degree is the length of the cokernel of the linearized
Frobenius $\varphi^*\mathfrak M\to\mathfrak M$ divided by
$[K:\mathbb Q_p]$, the sum of the elementary divisors of the
Frobenius, and vanishes exactly on the \'etale objects, whereas
$\degFG(G)$ measures how far the lengths of the levels of $G$ drop
below the length of $G^{-\infty}$, and vanishes on every
finite-length Dieudonn\'e module.  So their theory separates the
finite flat group schemes, while ours places all of them at slope
$0$ and separates the diagonal dominoes, which have no counterpart
among Kisin modules.

\medskip
\noindent\textbf{Notation and conventions.}
For $M\in D^b_c(R)$ we write $M(i)$ for the
internal shift, with $M(i)^j=M^{i+j}$, $M[n]$ for the cohomological shift, and
$M\{i\}:=M(i)[-i]$ for the Breuil--Kisin twist.  Dieudonn\'e modules are always
$V$-complete, which for finite length means that $V$ is nilpotent.
Ekedahl's simple functor is $\mathbf s$ and his derived equivalence is $S$.
Restriction and extension of scalars along the Frobenius are
$\sigma_*$ and $\sigma^*$, mutually inverse since $k$ is perfect.
We use triangulated categories throughout.

\medskip
\noindent\textbf{Acknowledgements.} The author thanks his advisor, Ben Antieau,
for his guidance and for many helpful discussions.  This material is based
upon work supported by the National Science Foundation under Grant
No.~DMS-1928930 while the author was in residence at the Simons Laufer
Mathematical Sciences Institute in Berkeley, California, during the Fall
2026 semester, for the program \emph{Motivic Homotopy Theory: Connections
and Applications}.  The author is also supported by the Simons
Collaboration on Perfection in Algebra, Geometry, and Topology.

\medskip
\noindent\textbf{AI disclosure.} The mathematical content of this
paper is due to the author, apart from what follows.  The AI
assistants ChatGPT 5.6 Sol Extra High and Claude Fable 5 found a gap
in the author's first draft
of the proof of Proposition~\ref{prop:chain-presentation}, supplied a
new argument, and pointed out its connection with the
classification of string and band modules over semilinear clannish
algebras \cite{BTCB}.  The author then verified
that argument in detail and rewrote it.  These tools also drafted
expository passages and were used to proofread and edit the
manuscript.  All content they produced has been reviewed, edited, and
checked for accuracy by the author, who is responsible for the paper
as it stands.

\section{Preliminaries}
\label{sec:prelim}

This section provides a brief summary of Ekedahl's formalism of
diagonal complexes and $F$-gauges used below \cite{ekedahl3}.  The
topics of ordinary dominoes, coherent $R$-module d\'evissage, and
Mazur--Ogus varieties are treated at greater length in
\cite{ZhangDominoes}.

\subsection{Diagonal complexes and $F$-gauges}
\label{subsec:prelim-diagonal}

\begin{definition}
\label{def:raynaud-ring-ekedahl}
The Raynaud ring is the noncommutative graded $\mathbb Z_p$-algebra
\[
 R:=W_\sigma\{F,V,d\}/(FV=VF=p,\ d^2=0,\ FdV=d)=R^0\oplus R^1,
\]
graded by $\deg F=\deg V=0$ and $\deg d=1$, where the subscript
$\sigma$ denotes Frobenius semilinearity: $Fa=\sigma(a)F$,
$Va=\sigma^{-1}(a)V$, and $da=ad$ for $a\in W$.  Thus $dF=pFd$, $Vd=pdV$, and
$R^0=W_\sigma[F,V]/(FV-p)$.  The right ideals
$V^nR+dV^nR$ define its standard filtration.  Put
$R_n=R/(V^nR+dV^nR)$, a $(W_n[d]/d^2,R)$-bimodule, and write
$\widehat R:=\varprojlim_n R_n$ for the completed Raynaud ring.  Multiplication on $R$ extends continuously to
$\widehat R$.
\end{definition}

\begin{definition}[$R$-modules, completion, and coherence]
\label{def:raynaud-complex}
\label{def:raynaud-completion}
\label{def:coherent-raynaud-complex}
Write $\mathrm{Mod}(R)$ for graded left $R$-modules and $D(R)$ for its derived category.
The Postnikov $t$-structure on $D(R)$ has heart $\mathrm{Mod}(R)$.
For $M\in D(R)$ set
$\widehat M:=R\!\varprojlim_n(R_n\otimes_R^{\mathbf L}M)$.
It is \emph{$R$-complete} if $M\simeq\widehat M$, and \emph{coherent} if, in
addition, every $R_n\otimes_R^{\mathbf L}M$ has finitely generated $W_n$-cohomology.
Let $D^b_c(R)$ be the bounded coherent subcategory.  Coherent modules form an abelian
extension-closed category and $D^b_c(R)$ is triangulated
\cite[Cor.~2.4.8]{Illusie83}.
\end{definition}

\begin{definition}[The simple functor]
\label{def:simple-functor}
One can represent $M\in D(R)$ as a double complex with Raynaud
differential $d$ and cohomological differential $\partial$.  Ekedahl's
\emph{simple functor} $\mathbf s:D(R)\to D(W)$ totalizes the two
differentials \cite[O.3]{ekedahl3}.  For a smooth proper $X/k$, the
complex $R\Gamma(X,W\Omega_X^\bullet)$ lies in $D^b_c(R)$ and
$\mathbf s\,R\Gamma(X,W\Omega_X^\bullet)\simeq
R\Gamma_{\mathrm{crys}}(X/W)$.
\end{definition}

We use the bidegree notation $M^{i,j}$, with internal degree $i$
and cohomological degree $j$.  The internal shift is
$M(a)^{i,j}=M^{a+i,j}$ and the cohomological shift is
$M[n]^{i,j}=M^{i,j+n}$.  The \emph{Breuil--Kisin twist}
$M\{a\}:=M(a)[-a]$ preserves the total diagonal $i+j$.

\begin{definition}[The diagonal $t$-structure]
\label{def:heart-operators}
\label{def:diagonal-t-structure}
\label{def:diagonal-heart}
For an $R$-module $M$, set
\[
 \begin{aligned}
 Z^i&=\ker(d:M^i\to M^{i+1}),&
 B^i&=\operatorname{im}(d:M^{i-1}\to M^i),\\
 V^{-\infty}Z^i&=\bigcap_{n\ge0}\ker(dV^n),&
 F^\infty B^i&=\bigcup_{n\ge0}\operatorname{im}(F^nd).
 \end{aligned}
\]
Then $B^i\subset F^\infty B^i\subset V^{-\infty}Z^i\subset Z^i$, and $d$ factors
through $M^i/V^{-\infty}Z^i\to F^\infty B^{i+1}$.  For each $i$, the
diagonal truncations
\[
\begin{aligned}
 \widetilde\tau_{\le i}M&:=\bigl[\cdots\xrightarrow{d}M^{i-1}
 \xrightarrow{d}M^i\xrightarrow{d}F^\infty B^{i+1}\to0\bigr],\\
 \widetilde\tau_{\ge i+1}M&:=\bigl[0\to M^{i+1}/F^\infty B^{i+1}
 \xrightarrow{d}M^{i+2}\xrightarrow{d}\cdots\bigr]
\end{aligned}
\]
sit in a natural short exact sequence
$0\to\widetilde\tau_{\le i}M\to M\to\widetilde\tau_{\ge i+1}M\to0$.
Ekedahl's \emph{diagonal $t$-structure} on $D^b_c(R)$ has aisles
\[
 \widetilde D^{\le0}=\{E:\widetilde\tau_{\le-q}H^q(E)\simeq H^q(E)
 \text{ for all }q\},
 \qquad
 \widetilde D^{\ge0}=\{E:\widetilde\tau_{\ge-q}H^q(E)\simeq H^q(E)
 \text{ for all }q\}
\]
\cite[Theorem~I.1.1 and Definition~I.1.2]{ekedahl3}.  Its heart is
denoted by $\Delta$, its objects are called \emph{diagonal
complexes}, and its cohomology functors are
$\widetilde H^i(E):=(\widetilde\tau_{\le i}\widetilde\tau_{\ge i}E)[i]$.
\end{definition}

\begin{definition}[Dominoes]
\label{def:elementary-domino}
\label{rem:elementary-type}
For $j\in\mathbb Z$, with the convention $dV^n=F^{-n}d$ for $n<0$, put
$U_j:=\widehat R/\widehat R(F,dV^{j-1})$.  It is concentrated in internal degrees
$0,1$, with
\[
 (U_j)^0=\prod_{n\ge0}kV^n,\qquad (U_j)^1=\prod_{n\ge j}kdV^n.
\]
Here $V$ is injective and $F=0$ in degree $0$, while $F$ is surjective and
$V=0$ in degree $1$ \cite[p.~108]{IR83}.  Its Euler characteristic is
$\dim_k\ker d-\dim_k\operatorname{coker}d=j$, called the \emph{type}
by Ekedahl and the \emph{degree} in this paper.
\label{def:domino}
An $R$-module is a \emph{domino} if it has a finite filtration with
graded pieces isomorphic to $U_j$'s.
The number of factors is independent of the filtration, equals
$\dim_k(M^0/VM^0)=\dim_k(M^1[F])$, and is called the \emph{dimension}
of the domino.
\end{definition}

\begin{definition}
\label{def:s-torsion-vocab}
For every $M\in\Delta$, the cohomology of $\mathbf s(M)$ is
concentrated in degrees $0$ and $1$
\cite[Proposition~I.1.4(i)]{ekedahl3}.  An object $M\in\Delta$ is
\emph{$\mathbf s$-torsion} if
$H^*(\mathbf s(M))$ is torsion, \emph{$\mathbf s$-$0$-torsion}
(\emph{$\mathbf s$-$1$-torsion}) if moreover $H^1(\mathbf s(M))=0$
(resp.\ $H^0(\mathbf s(M))=0$), and \emph{$\mathbf s$-acyclic} if
$\mathbf s(M)=0$ \cite[Definition~I.3.1]{ekedahl3}.
\end{definition}

\begin{example}
\label{ex:Uj-s-torsion}
The complex $\mathbf s(U_j)$ is $k^j$ in degree $0$ for $j\ge0$, and $k^{-j}$ in
degree $1$ for $j\le0$.  Thus $U_j$ is $\mathbf s$-$0$-torsion for
$j\ge0$, $\mathbf s$-$1$-torsion for $j\le0$, and
$\mathbf s$-acyclic exactly for $j=0$.
\end{example}

\begin{definition}[Nygaard modifications]
\label{def:autoequivalence}
Write $\mathbf1$ for the constant function $1$ and let $\Phi$ be the
group of functions $\varphi\colon\mathbb Z\to\mathbb Z$ of the form
$\varphi=\pi+f$, where $\pi$ is periodic and $f$ has finite support.
This decomposition is unique, since a periodic function of finite
support is zero.  For periodic $\pi$, define $h(\pi)$ by
$h_0(\pi)=0$ and $h_{i-1}(\pi)-h_i(\pi)=\pi(i)$.
Put $s_i(\varphi):=h_i(\pi)-\sum_{j\le i}f(j)$, so that
$s_{i-1}(\varphi)-s_i(\varphi)=\varphi(i)$.

For an $R$-module $M$ and $\varphi\in\Phi$, the \emph{Nygaard modification}
$M(\varphi)$ has $M(\varphi)^i:=\sigma_*^{-s_i(\varphi)}M^i$.
The underlying additive maps of $F$, $V$ and morphisms are unchanged,
and the differential is
$d_\varphi=F^{\varphi(i)}d\colon M(\varphi)^{i-1}\to M(\varphi)^i$,
where $F^{-r}d:=dV^r$ for $r>0$.
The primitive identity makes $d_\varphi$ $W$-linear, and the
Raynaud relations are preserved.
The resulting functors are exact, and additivity of the primitive gives
$(M(\varphi_1))(\varphi_2)=M(\varphi_1+\varphi_2)$.
Thus they are autoequivalences with inverse $(-)(-\varphi)$,
as in \cite[0.2]{ekedahl3}.

Applied termwise, these functors preserve $D_c^b(R)$ by the elementary
d\'evissage: they carry degree shifts of elementary modules to
Frobenius twists of degree shifts of elementary modules
\cite[0.4]{ekedahl3}.
The relation $F^a dV^a=d$ for $a\ge0$ shows that $F^\infty B^i$
is unchanged as an additive subgroup.
Hence the functors commute with the diagonal truncations and
are exact for the diagonal $t$-structure
\cite[Proposition~I.1.6(i)]{ekedahl3}.

Write $\delta_i$ for the indicator of $i$.  The module $U_j$ is
concentrated in internal degrees $0$ and $1$, and
$U_j(r\delta_1)=U_j(r\mathbf1)\simeq U_{j+r}$
\cite[0.2 and Theorem~III.2.3]{ekedahl3}.
\end{definition}

\begin{definition}[Domino numbers]
\label{def:domino-numbers}
For $M\in D^b_c(R)$, write $H^j(M)^i$ for the degree-$i$ part of the
ordinary cohomology, the term $E_1^{i,j}$ of the slope spectral
sequence.  The differential $d:H^j(M)^i\to H^j(M)^{i+1}$ induces a map
$H^j(M)^i/V^{-\infty}Z\to F^\infty B^{i+1}$, which is a domino in
internal degrees $i$ and $i+1$, and $T^{i,j}(M)$ is its dimension
\cite[0.6]{ekedahl3}, called the \emph{domino number}.
\end{definition}

The other side of Ekedahl's equivalence is the category of
$F$-gauges, which we describe in the language of the stacky
approach to prismatic cohomology \cite{BhattFgauges}.

\begin{definition}[$k^{\Delta}$, $k^{\mathcal N}$, and $k^{\mathrm{Syn}}$]
\label{def:f-gauge}
\label{def:stacks}
We define three stacks over $\operatorname{Spf}\mathbb Z_p$ attached
to $\operatorname{Spec}k$ and describe their quasi-coherent sheaves,
following \cite[Sections~3.3, 3.4, 4.1 and~4.2]{BhattFgauges}.
\begin{enumerate}
\item The \emph{prismatization} is $k^{\Delta}:=\operatorname{Spf}W$.
  Its quasi-coherent sheaves are the derived $p$-complete
  $W$-modules.
\item Let $A:=W[u,t]/(ut-p)$ be the Rees algebra of the $p$-adic
  filtration of $W$, graded by $\deg u=1$ and $\deg t=-1$.  The
  \emph{Nygaard filtered prismatization} is
  $k^{\mathcal N}:=\operatorname{Spf}(A)/\mathbb G_m$, and its
  quasi-coherent sheaves are the derived $p$-complete graded
  $A$-modules.  A graded $A$-module is a \emph{gauge}: a graded
  $W$-module $M=(M^i)$ with $W$-linear maps $u\colon M^i\to M^{i+1}$
  and $t\colon M^{i+1}\to M^i$, $ut=tu=p$.  Put
  $M^\infty:=\varinjlim_{u}M^i$ and $M^{-\infty}:=\varinjlim_{t}M^{-i}$,
  with colimit maps $u^\infty\colon M^i\to M^\infty$ and
  $t^{-\infty}\colon M^i\to M^{-\infty}$.  The open substacks
  $\{u\ne0\}$ and $\{t\ne0\}$ are copies of $k^{\Delta}$, and
  the restrictions of $M$ to them are $M^\infty$ and $M^{-\infty}$.
  A gauge is \emph{coherent}
  if it is derived $p$-complete with finite-dimensional cohomology
  modulo $(u,t)$, and has \emph{level} $[m,n]$ if $t$ is invertible
  below $m$ and $u$ above $n$.  For a coherent gauge the level is
  the smallest interval that contains the \emph{Hodge--Tate
  weights}, the degrees in which the derived restriction of $M$ to
  the \emph{Hodge point} $k^{\mathrm{Hodge}}:=\{u=t=0\}=B\mathbb G_{m,k}$
  is nonzero.
\item The \emph{syntomification} $k^{\mathrm{Syn}}$ is the pushout
  \[
  \begin{tikzcd}[column sep=4em, row sep=2em]
   k^{\Delta}\sqcup k^{\Delta}
   \arrow[r, "{(j_{\mathrm{HT}},\,j_{\mathrm{dR}})}"]
   \arrow[d, "\mathrm{can}"']
   & k^{\mathcal N} \arrow[d, "j_{\mathcal N}"] \\
   k^{\Delta} \arrow[r, "j_\Delta"'] & k^{\mathrm{Syn}}
  \end{tikzcd}
  \]
  where $j_{\mathrm{HT}}$ and $j_{\mathrm{dR}}$ are the open
  immersions onto $\{u\ne0\}$ and $\{t\ne0\}$, the second over the
  inverse Frobenius of $W$, and $\mathrm{can}$ is the fold map.  Its
  quasi-coherent sheaves are the derived $p$-complete
  \emph{$F$-gauges}, gauges $M$ with a $W$-linear isomorphism
  $\tau\colon\sigma^*M^\infty\to M^{-\infty}$, so that
  $D_{\mathrm{qc}}(k^{\mathrm{Syn}})$ is the equalizer
  \[
  \begin{tikzcd}[column sep=5em]
   D_{\mathrm{qc}}(k^{\mathrm{Syn}}) \arrow[r, "j_{\mathcal N}^*"]
   & D_{\mathrm{qc}}(k^{\mathcal N})
   \arrow[r, shift left=.7ex, "M\mapsto M^{-\infty}"]
   \arrow[r, shift right=.7ex, "M\mapsto\sigma^*M^\infty"']
   & D_{\mathrm{qc}}(k^{\Delta})
  \end{tikzcd}
  \]
  The stack $k^{\mathrm{Syn}}$ is regular and noetherian, with
  bounded coherent category $\Perf(k^{\mathrm{Syn}})$ and standard
  heart $\Coh(k^{\mathrm{Syn}})$.  Without the requirement that
  $\tau$ be an isomorphism, $M$ is an \emph{$F$-module}
  \cite[Section~1.2]{FontaineJannsen}.  The open point $j_\Delta$
  satisfies $j_\Delta^*M=M^{\infty}$, and $j_{\mathcal N}^*$ forgets
  $\tau$.
\end{enumerate}
\end{definition}

\begin{definition}[Ekedahl's functor $S$]
\label{def:ekedahl-S}
For an $R$-module $M$, with $\sigma_*$ the Frobenius twist, the
adjacent components of $S(M)$ and their transition maps are
\cite[Definition~II.3.1]{ekedahl3}:
\begin{equation}
\label{eq:ekedahl-S-component}
\begin{tikzcd}[column sep=1.8em, row sep=2.5em]
S(M)^{r+1}={}
  & \cdots
      \arrow[r, "\sigma_*d"]
      \arrow[d, shift left=.6ex, "p"]
  & \sigma_*M^{r-1}
      \arrow[r, "\sigma_*d"]
      \arrow[d, shift left=.6ex, "p"]
  & \sigma_*M^r
      \arrow[r, "dV"]
      \arrow[d, shift left=.6ex, "V"]
  & M^{r+1}
      \arrow[r, "d"]
      \arrow[d, shift left=.6ex, "1"]
  & \cdots
      \arrow[d, shift left=.6ex, "1"]
\\
S(M)^r={}
  & \cdots
      \arrow[r, "\sigma_*d"]
      \arrow[u, shift left=.6ex, "1"]
  & \sigma_*M^{r-1}
      \arrow[r, "dV"]
      \arrow[u, shift left=.6ex, "1"]
  & M^r
      \arrow[r, "d"]
      \arrow[u, shift left=.6ex, "F"]
  & M^{r+1}
      \arrow[r, "d"]
      \arrow[u, shift left=.6ex, "p"]
  & \cdots
      \arrow[u, shift left=.6ex, "p"]
\end{tikzcd}
\end{equation}
The downward arrows are $t$, the upward arrows are $u$, and each
vertical pair composes to $p$.  In the limits,
$S(M)^{-\infty}=\mathbf s(M)$ and $S(M)^\infty=\sigma_*\mathbf s(M)$,
and the gluing $\tau$ is the identity of the underlying complex,
through $\sigma^*\sigma_*=\mathrm{id}$.
In the notation of Definition~\ref{def:autoequivalence},
$S(M)^r=\sigma_*\mathbf s\bigl(M(-\delta_r)\bigr)$.  For a complex of
$R$-modules, the construction applies to each term
and $S(M)^r$ is the totalization of the resulting double complex.
\end{definition}

The following theorem is the main result of \cite{ekedahl3} and
the bridge between the two sides of this paper.  Ekedahl states it
for his category of coherent complexes of $F$-gauges, which is
$\Perf(k^{\mathrm{Syn}})$ in the language of
Definition~\ref{def:stacks}.

\begin{theorem}[{Ekedahl \cite[\S\S II.3--5]{ekedahl3}}]
\label{thm:ekedahl-equivalence}
Let $k$ be a perfect field of characteristic $p$.  The functor $S$
of Definition~\ref{def:ekedahl-S} is exact, hence induces a
functor on the derived category, which carries $D^b_c(R)$ into
$\Perf(k^{\mathrm{Syn}})$, and the resulting functor
\[
 S\colon D^b_c(R)\longrightarrow\Perf(k^{\mathrm{Syn}})
\]
is an equivalence of triangulated categories.
\end{theorem}

\begin{example}
\label{ex:equivalence-examples}
\begin{enumerate}
\item For a Dieudonn\'e module $M$,
\[
 S(M)=
 \cdots\fgpair{p}{1}\sigma_*M
 \fgpair{p}{1}\sigma_*M
 \fgpair{V}{F}M
 \fgpair{1}{p}M
 \fgpair{1}{p}\cdots ,
\]
with $\sigma_*M$ in levels $\ge1$, $M$ in levels $\le0$, and $\tau$
the identity.  Finite-length Dieudonn\'e modules give
the torsion $F$-gauges of level $[0,1]$ \cite[II.2.4(ii)]{ekedahl3}.
\item Let $\{i\}$ denote tensor product with Bhatt's Breuil--Kisin
  line bundle $\mathcal O_{k^{\mathrm{Syn}}}\{i\}$.  The line
  bundle is trivial on the two open copies of $\operatorname{Spf}W$
  and restricts to $\mathcal O(-i)$ on $k^{\mathcal N}$
  \cite[Definition~4.3.4 and Remark~4.3.5]{BhattFgauges}.  So the
  twist shifts the levels, $(G\{i\})^r=G^{r+i}$, and leaves $u$, $t$,
  the two limits and the gluing unchanged.  Then
  $S(M\{i\})\simeq S(M)\{i\}$, and it suffices to check it on
  $M\in\mathrm{Mod}(R)$: levelwise
\[
\begin{aligned}
 S(M\{i\})^r
 &=\sigma_*\mathbf s\bigl(M\{i\}(-\delta_r)\bigr)
 =\sigma_*\mathbf s\bigl((M(-\delta_{r+i}))\{i\}\bigr)\\
 &=\sigma_*\mathbf s\bigl(M(-\delta_{r+i})\bigr)
 =S(M)^{r+i},
\end{aligned}
\]
  since $(-\delta_r)$ modifies the differential entering internal
  degree $r$, which for $M\{i\}$ is the differential entering
  internal degree $r+i$ of $M$, and
  $\mathbf s(N\{a\})=\mathbf s(N)[a][-a]=\mathbf s(N)$.  Both
  sides have limits $\mathbf s(M)$ and $\sigma_*\mathbf s(M)$ with
  the identity gluing.
\item We have
\[
 S(U_1)=
 \cdots\fgpair{0}{\sim}k
 \fgpair{0}{\sim}k
 \rightleftarrows0\rightleftarrows
 k\fgpair{\sim}{0}k
 \fgpair{\sim}{0}\cdots ,
\]
with the zero term in level $1$, and
\[
 S(U_0[1])=
 \cdots\rightleftarrows0\rightleftarrows k\rightleftarrows0
 \rightleftarrows\cdots ,
\]
with the single term $k$ in level $1$, so that $S(U_0\{r\}[1])$ is
concentrated in level $1-r$.
\end{enumerate}
\end{example}

\begin{remark}[Generating the category of $F$-gauges]
\label{rem:carmeli-feng}
In \cite[Remark~4.4.6]{BhattFgauges}, Bhatt asked whether the
derived $\infty$-category $D_{\mathrm{qc}}(k^{\mathrm{Syn}})$ of prismatic
$F$-gauges is compactly generated.  Carmeli--Feng recently answered
this question affirmatively \cite[\S9.1]{CarmeliFeng25}: the
cohomologies of smooth projective $k$-varieties generate
$D_{\mathrm{qc}}(k^{\mathrm{Syn}})$.  The key step of their proof
reduces the generation to the Breuil--Kisin twists of two objects,
the structure sheaf
$\mathcal O_{k^{\mathrm{Syn}}}$ and the skyscraper sheaf $\delta$
of the Hodge point
$k^{\mathrm{Hodge}}\xhookrightarrow{i_{\mathrm{Hodge}}} k^{\mathrm{Syn}}$
(Definition~\ref{def:stacks}).
The skyscraper sheaf
$\delta=(i_{\mathrm{Hodge}})_*\mathcal O_{k^{\mathrm{Hodge}}}$ is
precisely the $F$-gauge with a single term $k$ in level $0$, and in
the notation above, $\delta\simeq S(U_0[1])\{1\}=S(U_0(1))$.
\end{remark}

\subsection{The $F$-gauge $t$-structure and HRS tilting}
\label{subsec:prelim-tilting}

Besides the Postnikov and diagonal $t$-structures, Ekedahl defines
the \emph{$F$-gauge $t$-structure} on $D^b_c(R)$, whose heart is
denoted by $\mathcal G$.  He shows that under the
derived equivalence $S$ it corresponds to the standard $t$-structure
on $\Perf(k^{\mathrm{Syn}})$, and $\mathcal G$ is equivalent to
$\Coh(k^{\mathrm{Syn}})$.  We first recall Ekedahl's definition of
this $t$-structure on the Raynaud side, and interpret its relation
with the diagonal heart via the Happel--Reiten--Smal\o{} (HRS)
tilting formalism \cite{HRS96}.
For $M\in\Delta$, Ekedahl defines two canonical functorial
filtrations $0\subseteq F^1M\subseteq F^2M\subseteq M$ and
$0\subseteq G^1M\subseteq G^2M\subseteq M$, with $G^1=F^1\cap G^2$
and $F^2=F^1+G^2$ \cite[Theorem~I.4.5]{ekedahl3}.  Here we only
recall the definition of $G^2$: let $G^2M$ be the smallest subobject
of $M$ in $\Delta$ such that $M/G^2M$ is $\mathbf s$-$1$-torsion.
The sequence
\begin{equation}
 0\longrightarrow G^2M\longrightarrow M\longrightarrow M/G^2M\longrightarrow0
 \label{eq:G2-torsion-decomposition}
\end{equation}
is functorial and $G^2(G^2M)=G^2M$
\cite[Section~II.1]{ekedahl3}.

Recall the HRS setup \cite{HRS96}.  A \emph{torsion pair}
$(\mathcal T,\mathcal F)$ in an abelian category $\mathcal A$
consists of two full subcategories with
$\operatorname{Hom}(\mathcal T,\mathcal F)=0$, such that every
$A\in\mathcal A$ sits in a short exact sequence
$0\to T\to A\to F\to0$ with $T\in\mathcal T$ and
$F\in\mathcal F$.  When $\mathcal A$ is the heart of a bounded
$t$-structure on a triangulated category $D$, the \emph{tilt}
$\langle\mathcal F[1],\mathcal T\rangle$ is the full subcategory
of the $E\in D$ with $H^0(E)\in\mathcal T$,
$H^{-1}(E)\in\mathcal F$, and $H^i(E)=0$ otherwise.  It is the
heart of a new bounded $t$-structure on $D$.

\begin{proposition}
\label{prop:ekedahl-heart-hrs}
Put $\mathcal T=\{M\in\Delta:G^2M=M\}$ and
$\mathcal F=\{M\in\Delta:G^2M=0\}$.  Then
$(\mathcal T,\mathcal F)$ is a torsion pair in $\Delta$.
Equivalently, $\mathcal T$ consists of the objects whose maximal
$\mathbf s$-$1$-torsion quotient vanishes, and $\mathcal F$ consists
of the $\mathbf s$-$1$-torsion objects.  The $F$-gauge heart is the
tilt $\mathcal G=\langle\mathcal F[1],\mathcal T\rangle$, and $S$
restricts to an exact equivalence
$S:\mathcal G\xrightarrow{\ \sim\ }\Coh(k^{\mathrm{Syn}})$.
\end{proposition}

\begin{proof}
For the torsion pair, it suffices to check that
$\operatorname{Hom}(\mathcal T,\mathcal F)=0$ and that every
$M\in\Delta$ has a decomposition.  The image of a map $T\to F$ is
an $\mathbf s$-$1$-torsion quotient of $T$
\cite[Lemma~I.4.2(ii)]{ekedahl3}, hence zero, and
\eqref{eq:G2-torsion-decomposition} is a decomposition, with
$G^2M\in\mathcal T$ by idempotence and $M/G^2M\in\mathcal F$ by
the defining property of $G^2$.  Ekedahl's construction of the
$F$-gauge $t$-structure from $G^2$ then gives
$\mathcal G=\langle\mathcal F[1],\mathcal T\rangle$
\cite[Section~II.1]{ekedahl3}, and his Theorem~II.5.3 and
Proposition~II.5.4 give the asserted equivalence.
\end{proof}

\begin{definition}[Torsion]
\label{def:torsion-category}
An object $M\in\Delta$ is \emph{torsion} if it is
$\mathbf s$-torsion, equivalently $p^nM=0$ for some $n$, and then the
$\mathbf s$-cohomology has finite length.  Write
$\Delta_{\mathrm{tor}}$ for the full subcategory of such objects.  It is a Serre
subcategory of $\Delta$ \cite[Lemma~I.4.2(i)]{ekedahl3}.  Write
$\Dtor\subset D^b_c(R)$ for the full triangulated subcategory of
complexes with $W$-torsion $\mathbf s$-cohomology, equivalently with
every $\widetilde H^i$ in $\Delta_{\mathrm{tor}}$.  On the $F$-gauge
side, write $\Coh_{\mathrm{tor}}(k^{\mathrm{Syn}})$ for the objects
of $\Coh(k^{\mathrm{Syn}})$ killed by a power of $p$, and
$\Perf_{\mathrm{tor}}(k^{\mathrm{Syn}})$ for the objects of
$\Perf(k^{\mathrm{Syn}})$ whose standard cohomology objects lie in
it.
\end{definition}

\begin{corollary}
\label{cor:torsion-equivalence}
The equivalence $S$ restricts to an equivalence
$\Dtor\simeq\Perf_{\mathrm{tor}}(k^{\mathrm{Syn}})$.
\end{corollary}

\begin{proof}
A coherent $F$-gauge $M$ is killed by a power of $p$ exactly when
$M^{-\infty}\simeq\sigma^*M^\infty$ is a torsion $W$-module, and
$S(E)^{-\infty}=\mathbf s(E)$ by Definition~\ref{def:ekedahl-S}.
\end{proof}

By Definition~\ref{def:torsion-category}, for
$\mathcal G_{\mathrm{tor}}:=\mathcal G\cap\Dtor$ the equivalence
$S$ restricts to
$\mathcal G_{\mathrm{tor}}\simeq\Coh_{\mathrm{tor}}(k^{\mathrm{Syn}})$.

\begin{remark}[Torsion and isocrystals]
\label{rem:torsion-isocrystal}
For a $\mathbb Z_p$-linear small stable category $\mathcal C$, write
$\mathcal C[1/p]:=\mathcal C\otimes_{\Perf(\mathbb Z_p)}\Perf(\mathbb Q_p)$.
There is a canonical exact sequence in the $\infty$-category of
$\mathbb Z_p$-linear idempotent-complete small stable categories:
\[
 \Perf_{\mathrm{tor}}(k^{\mathrm{Syn}})
 \longrightarrow\Perf(k^{\mathrm{Syn}})
 \longrightarrow\Perf(k^{\mathrm{Syn}})[1/p].
\]
Indeed, since $k^{\mathrm{Syn}}$ is regular and noetherian, a
perfect complex vanishes after inverting $p$ exactly when all of
its standard cohomology sheaves are killed by a power of $p$, so
the first term is the kernel of the second arrow, and the second
arrow realizes the idempotent-completed Verdier quotient.

The last term in the localization sequence is classical.  In a gauge $ut=tu=p$, so inverting
$p$ makes $u$ and $t$ mutually inverse and collapses all the levels
onto the single perfect complex $V:=G^{-\infty}[1/p]$ over
$K:=W[1/p]$.  The gluing survives as a $\sigma$-semilinear
automorphism of $V$, so $\Perf(k^{\mathrm{Syn}})[1/p]$ is the derived
category of $F$-isocrystals over $k$.  For $k$ algebraically closed, the
Dieudonn\'e--Manin classification \cite{Manin63} describes the objects
of this quotient, and the present paper develops a
Harder--Narasimhan theory in its kernel.
\end{remark}

\begin{definition}[Finite torsion and diagonal dominoes]
\label{def:finite-torsion}
An object is \emph{finite torsion} if it is a
finite direct sum of objects $L\{i\}$, where $L$ is a finite-length
Dieudonn\'e module \cite[Theorem~I.1.12]{ekedahl3}.
Write $\Delta_{\mathrm{ft}}$ for this Serre subcategory
\cite[Proposition~I.1.11]{ekedahl3}.  Every $M\in\Delta$ has a
functorial short exact sequence
\begin{equation}
  0\longrightarrow M_{\mathrm{ft}}\longrightarrow M\longrightarrow M/M_{\mathrm{ft}}\longrightarrow0,
  \label{eq:finite-torsion-radical}
\end{equation}
where $M_{\mathrm{ft}}$ is the largest finite-torsion subobject and
$(M/M_{\mathrm{ft}})_{\mathrm{ft}}=0$
\cite[Proposition~I.1.11]{ekedahl3}.\footnote{Ekedahl writes $t^1(M)$
and $t^0(M)$ for $M_{\mathrm{ft}}$ and $M/M_{\mathrm{ft}}$.}
\label{def:diagonal-domino}
A \emph{diagonal domino} is an $M\in\Delta_{\mathrm{tor}}$ with
$M_{\mathrm{ft}}=0$.
Write $\DomD\subset\Delta$ for the full subcategory, with exact
sequences inherited from $\Delta$ \cite[Definition~III.2.1]{ekedahl3}.
The intersection of $\DomD$ with
$\mathrm{Mod}(R)$ is the category of ordinary dominoes
\cite[Proposition~III.3.1]{ekedahl3}.
\end{definition}

\begin{lemma}
\label{lem:torsion-pair}
$(\Delta_{\mathrm{ft}},\DomD)$ is a torsion pair in $\Delta_{\mathrm{tor}}$.
\end{lemma}

\begin{proof}
We need to show that
$\operatorname{Hom}_\Delta(\Delta_{\mathrm{ft}},\DomD)=0$ and that
every object of $\Delta_{\mathrm{tor}}$ has a decomposition.  The image of a map from a finite-torsion object to a diagonal
domino is both a quotient of the source and a subobject of the
target.  It is therefore a finite-torsion subobject of a diagonal
domino and vanishes.  The sequence \eqref{eq:finite-torsion-radical} is a
decomposition, with finite-torsion subobject and diagonal domino
quotient (Definition~\ref{def:diagonal-domino}).
\end{proof}

\begin{proposition}[{\cite[Theorem~III.2.3 and Definition~III.2.4]{ekedahl3}}]
\label{prop:type-filtration}
Every $U\in\DomD$ admits a canonical finite filtration called the
\emph{type filtration} $\FilD^jU$, indexed by half-integers
$j\in\tfrac12\mathbb Z$.  This filtration is uniquely characterized
by three properties:
\begin{enumerate}
\item $\FilD^0U$ is the largest $\mathbf s$-$0$-torsion subobject.
\item $U/\FilD^{1/2}U$ is the largest $\mathbf s$-$1$-torsion
  quotient, so $\FilD^{1/2}U=G^2U$.
\item $(\FilD^jU)(r\mathbf1)=\FilD^{j+r}(U(r\mathbf1))$.
\end{enumerate}
We say $U$ is \emph{of type $j\in\tfrac12\mathbb Z$} if $\FilD^jU=U$
and $\FilD^{j+1/2}U=0$.  For example, when $j\in\mathbb Z$, the
domino $U_j$ is of type $j$.
\end{proposition}

The type filtration is functorial for inclusions: $N\subseteq U$ in
$\DomD$ gives $\FilD^jN\subseteq\FilD^jU$ for every $j$.  For $j=0$
the subobject $\FilD^0N$ is $\mathbf s$-$0$-torsion, hence contained
in the largest one of $U$.  For $j=\tfrac12$ the quotient
$N/(N\cap\FilD^{1/2}U)$ embeds into $U/\FilD^{1/2}U$ and is
$\mathbf s$-$1$-torsion, since the amplitude of $\mathbf s$ makes
$H^0(\mathbf s(-))$ vanish on subobjects of $\mathbf s$-$1$-torsion
objects, so $\FilD^{1/2}N\subseteq N\cap\FilD^{1/2}U$ by minimality.
Property (iii) transports the two cases to every $j$.

\begin{remark}
\label{rem:chain-conditions}
The heart $\Delta$ is noetherian \cite[Proposition~I.1.5]{ekedahl3},
hence so is the Serre subcategory $\Delta_{\mathrm{tor}}$.  Artinian
fails already in $\Delta_{\mathrm{tor}}$.  There is a short exact
sequence $0\to U_{j-1}\to U_j\to k\to0$ in $\Delta_{\mathrm{tor}}$,
where $k$ denotes the Dieudonn\'e module of $\alpha_p$, with $F=V=0$.
This gives an infinite strictly descending chain
$U_j\supset U_{j-1}\supset U_{j-2}\supset\cdots$, whose successive
quotients are finite torsion.  Descending chains whose quotients lie
in $\DomD$ are eventually constant
\cite[Proposition~I.1.11(iv)]{ekedahl3}.
\end{remark}

$\DomD$ is closed under taking subobjects in $\Delta$, since a
finite-torsion subobject of a subobject of $M$ is one of $M$.  By
Remark~\ref{rem:chain-conditions} it is not closed under quotients,
and the following notion repairs a quotient.

\begin{definition}[Saturation]
\label{def:saturated}
For $N\subseteq M$ in $\DomD$, the \emph{saturation}
$N^{\mathrm{sat}}$ of $N$ in $M$ is the smallest subobject of $M$
containing $N$ with $M/N^{\mathrm{sat}}\in\DomD$.  The subobject
$N$ is \emph{saturated} if $N=N^{\mathrm{sat}}$.  A \emph{saturated
sequence} is a short exact
sequence in $\Delta$ whose three terms lie in $\DomD$.
\end{definition}

\begin{proposition}
\label{prop:torsion-basics}
$N^{\mathrm{sat}}$ is the preimage of $(M/N)_{\mathrm{ft}}$ under
$M\twoheadrightarrow M/N$.
\end{proposition}

\begin{proof}
The quotient of $M$ by the preimage is
$(M/N)/(M/N)_{\mathrm{ft}}$, a diagonal domino, and the image of
$(M/N)_{\mathrm{ft}}$ in $M/N'$ vanishes for every $N'\supseteq N$
with $M/N'\in\DomD$, which gives the minimality.
\end{proof}

\section{The Harder--Narasimhan filtration}
\label{sec:hn}

\subsection{HN formalism on \texorpdfstring{$\Delta_{\mathrm{tor}}$}{Delta-tor}}
\label{subsec:hn-construction}

The construction runs parallel to the classical Harder--Narasimhan theory for
coherent sheaves on a curve \cite{HarderNarasimhan1975}.  The finite-length
Dieudonn\'e modules $\Delta_{\mathrm{ft}}$ play the role of the torsion sheaves,
at slope $+\infty$.  The diagonal dominoes $\DomD$ play the role of the
torsion-free sheaves.  On both sides the descending chain condition fails.  The
chain $U_j\supset U_{j-1}\supset\cdots$ of
Remark~\ref{rem:chain-conditions} is the analogue of
$\mathcal O\supset\mathcal O(-x)\supset\cdots$, with quotients of length one in
each case.  Hence, as in the classical case, we work with the
saturated subobjects of Definition~\ref{def:saturated}.

The argument is then the classical one, and what it uses about
$\Delta_{\mathrm{tor}}$ is the following.
\begin{enumerate}
\item The rank and degree defined below are additive
  (Lemma~\ref{lem:additivity}).
\item The rank vanishes exactly on $\Delta_{\mathrm{ft}}$, where the degree is
  positive (Proposition~\ref{prop:torsion-structure}(i)).
\item The pair $(\Delta_{\mathrm{ft}},\DomD)$ is a torsion pair
  (Lemma~\ref{lem:torsion-pair}).
\item The category $\Delta$ is noetherian
  (Remark~\ref{rem:chain-conditions}).
\item The amplitude of $\mathbf s$ bounds $\degD$ above on the subobjects of a
  fixed object (Definition~\ref{def:s-torsion-vocab}).
\end{enumerate}

\begin{definition}[Rank, degree, and slope]
\label{def:rank-degree-slope}
For $M\in\Delta_{\mathrm{tor}}$ put
$\rkD^i(M):=T^{i-1,1-i}(M)$, the domino number of
Definition~\ref{def:domino-numbers} attached to the differential entering
internal degree $i$.\footnote{Ekedahl's index is the source of the
differential, so $\rkD^i$ is his $T^{i-1,-(i-1)}$.  Indexing by the target
instead matches the Nygaard modifications of
Definition~\ref{def:autoequivalence} and the gauge levels of
Definition~\ref{def:ekedahl-S}.}  The \emph{rank} and the
\emph{degree} of $M$ are
\[
  \rkD(M):=\sum_i\rkD^i(M),\qquad
  \degD(M):=\chi_W(\mathbf s(M))=\ell_WH^0(\mathbf s(M))-\ell_WH^1(\mathbf s(M)),
\]
where $\chi_W=\sum_j(-1)^j\ell_WH^j$ and the two cohomology groups have finite
length.  Both are invariant under Breuil--Kisin twist, since the twist
permutes the $\rkD^i$ and $\mathbf s(M\{i\})\simeq\mathbf s(M)$.  The
same formula defines $\degD(E)$ for any $E\in D^b_c(R)$ whose
$\mathbf s$-realization has finite-length cohomology.  The \emph{slope} is
\[
  \muD(M):=
  \begin{cases}
    \degD(M)/\rkD(M),&\rkD(M)>0,\\
    +\infty,&\rkD(M)=0.
  \end{cases}
\]
\end{definition}

On ordinary dominoes the degree is the sum of the types and the rank is the
classical domino dimension.  Such a $U$ is the two-term complex
$U^0\xrightarrow{d}U^1$, so
$\degD(U)=\ell_W\ker d-\ell_W\operatorname{coker}d$, which on $U_j$ is its type
$j$ (Definition~\ref{def:elementary-domino}).  For a filtration with quotients
$U_{j_1},\dots,U_{j_n}$, additivity (Lemma~\ref{lem:additivity} below) gives
$\degD(U)=\sum_\ell j_\ell$, while
$\rkD(U)=n=\dim_k(U^0/VU^0)=\dim_k(U^1[F])$
\cite[Proposition~I.2.18]{IR83}.

\begin{lemma}[Additivity]
\label{lem:additivity}
Let $0\to A\to B\to C\to0$ be a short exact sequence in $\Delta_{\mathrm{tor}}$.
Then
\[
  \degD(B)=\degD(A)+\degD(C),
  \qquad
  \rkD(B)=\rkD(A)+\rkD(C).
\]
\end{lemma}

\begin{proof}
Applying $\mathbf s(-)$ gives a distinguished triangle, whose cohomology has
finite length by coherence and torsion, so the Euler characteristic $\degD$ is
additive along the resulting long exact sequence.

Each $\rkD^i$ is additive by \cite[Lemma~III.2.6.1(i)]{ekedahl3}, and summing
over the finitely many nonzero ones gives the rank.
\end{proof}

The slopes themselves are not additive, but the following weight functional is.
It will be used in the proof of Proposition~\ref{prop:mds}.

\begin{definition}[The weight functional]
\label{def:weight}
For $\lambda=d/q\in\mathbb Q$ with $\gcd(d,q)=1$ and $q>0$, and for
$A\in\Delta_{\mathrm{tor}}$ put
\[
 \theta_\lambda(A):=q\degD(A)-d\rkD(A).
\]
It is additive by Lemma~\ref{lem:additivity}, and
for $A\ne0$ it is positive, zero, or negative according as $\muD(A)$ is greater
than, equal to, or less than $\lambda$.  Geometrically it is $q$ times the
signed vertical distance from $(\rkD(A),\degD(A))$ to the line of slope
$\lambda$ through the origin.
\end{definition}

\begin{definition}[Semistable and stable objects]
\label{def:semistable}
A nonzero $M\in\Delta_{\mathrm{tor}}$ is \emph{semistable} if every nonzero
subobject $Y\subset M$ satisfies $\muD(Y)\le\muD(M)$, and \emph{stable} if the
inequality is strict for every proper nonzero subobject.
\end{definition}

Recall the sequence \eqref{eq:finite-torsion-radical} and the saturated
subobjects of Definition~\ref{def:saturated}.

\begin{proposition}
\label{prop:torsion-structure}
\begin{enumerate}
\item For $M\in\Delta_{\mathrm{tor}}$, one has $\rkD(M)=0$ exactly when $M$ is
  finite torsion, and then $\degD(M)>0$ unless $M=0$.
\item If $N\subset M$ are diagonal dominoes and $M/N$ is finite torsion, then
  $\rkD(N)=\rkD(M)$ and $\degD(N)\le\degD(M)$, with equality exactly when
  $N=M$.  The same holds for $\muD(N)\le\muD(M)$ and for
  $\theta_\lambda(N)\le\theta_\lambda(M)$ at every $\lambda\in\mathbb Q$.
\end{enumerate}
\end{proposition}

\begin{proof}
For (i), finite-length objects have zero domino numbers, so finite torsion
gives $\rkD(M)=0$.  Conversely additivity in
\eqref{eq:finite-torsion-radical} gives $\rkD(M)=\rkD(M/M_{\mathrm{ft}})$.  After an internal shift, the Postnikov
cohomology objects of $M/M_{\mathrm{ft}}$ are ordinary dominoes
\cite[Proposition~III.3.1]{ekedahl3}, and $\rkD^n$ is the dimension of the one
in degree $n$.  So $\rkD(M/M_{\mathrm{ft}})=0$ forces them all to vanish and $M/M_{\mathrm{ft}}=0$.
For the degree, write $M\simeq\bigoplus_iL_i\{i\}$ with the $L_i$
finite-length Dieudonn\'e modules
(Definition~\ref{def:finite-torsion}).  Then
$\mathbf s(L_i\{i\})\simeq L_i$ in degree zero, so
$\degD(M)=\sum_i\ell_W(L_i)$.  This is positive unless every $L_i$ vanishes,
that is unless $M=0$.

For (ii), put $T=M/N$.  Part (i) and additivity give $\rkD(N)=\rkD(M)$ and
$\degD(M)-\degD(N)=\degD(T)$, which is positive unless $T=0$.  If the common
rank vanishes then (i) makes $N$ and $M$ finite torsion diagonal dominoes and
hence zero.  Otherwise the common rank is positive, and dividing the degrees by
it gives $\muD(N)\le\muD(M)$, an equality precisely when $T=0$.  Since the ranks are equal,
$\theta_\lambda(M)-\theta_\lambda(N)$ is a positive multiple of
$\degD(M)-\degD(N)$ for every $\lambda$, which gives the last inequality.
\end{proof}

By part (i) the slope-$+\infty$ part of any HN filtration lies in $M_{\mathrm{ft}}$, so
the finite-slope part lives on $M/M_{\mathrm{ft}}\in\DomD$.  The next proposition reduces
semistability to saturated subobjects.

\begin{proposition}
\label{prop:stability-enlargement}
Finite-torsion objects are semistable of slope $+\infty$, and the stable ones
are the simple objects of $\Delta_{\mathrm{ft}}$.  A positive-rank object
$M$ is semistable if and only if $M_{\mathrm{ft}}=0$ and
$\muD(N)\le\muD(M)$ for every nonzero saturated $N\subset M$, and stable if and
only if in addition the inequality is strict for $N\ne M$.
\end{proposition}

\begin{proof}
The finite-torsion assertions follow from
Proposition~\ref{prop:torsion-structure}(i), and a positive-rank object with
$M_{\mathrm{ft}}\ne0$ is destabilized by its slope-$+\infty$ subobject $M_{\mathrm{ft}}$.

For $M\in\DomD$ and $N\subset M$, the quotient $N^{\mathrm{sat}}/N$ is finite
torsion, so Proposition~\ref{prop:torsion-structure}(ii) gives
$\muD(N)\le\muD(N^{\mathrm{sat}})$, and semistability is equivalent to the
inequalities for the saturated subobjects.
In the stable case the inequality is strict when $N^{\mathrm{sat}}\subsetneq M$,
while $N^{\mathrm{sat}}=M$ with $N\ne M$ makes $M/N$ nonzero finite torsion,
so that $\muD(N)<\muD(M)$ by additivity.
\end{proof}

The two standard ingredients follow: Hom-vanishing between semistable objects of
decreasing slope gives uniqueness, and the maximal destabilizing subobject gives
existence.

\begin{lemma}
\label{lem:hom-vanishing}
Let $A,B\in\DomD$ be semistable with $\muD(A)>\muD(B)$. Then
$\operatorname{Hom}_\Delta(A,B)=0$.
\end{lemma}

\begin{proof}
Let $f\colon A\to B$ be nonzero, with kernel $K$ and image $I\ne0$ computed in
$\Delta$, and write $\theta:=\theta_{\muD(A)}$, so that $\theta(A)=0$.
Semistability of $A$ gives $\theta(K)\le0$, so additivity along
$0\to K\to A\to I\to0$ gives $\theta(I)\ge0$.  On the other hand
$\theta(I)\le\theta(I^{\mathrm{sat}})$ by
Proposition~\ref{prop:torsion-structure}(ii), while semistability of $B$ gives
$\muD(I^{\mathrm{sat}})\le\muD(B)<\muD(A)$ and hence
$\theta(I^{\mathrm{sat}})<0$.  This contradicts $\theta(I)\ge0$.
\end{proof}

\begin{proposition}
\label{prop:mds}
Every nonzero $M\in\DomD$ has a largest saturated subobject $M^{+}$ of maximal
slope.  It is semistable, contains every saturated subobject of the same slope,
and every nonzero saturated subobject of $M/M^{+}$ has strictly smaller
slope.
\end{proposition}

\begin{proof}
There are four things to prove: that the realized slopes attain a maximum, that
the saturated subobjects of maximal slope have a largest member $M^{+}$, that
$M^{+}$ is semistable, and that every nonzero saturated subobject of the
quotient has strictly smaller slope.

For every nonzero saturated $N\subset M$ one has $\rkD(N)>0$ by
Proposition~\ref{prop:torsion-structure}(i) and $\rkD(N)\le\rkD(M)$ by
additivity, since the rank is a sum of dimensions and hence nonnegative.
In the long exact sequence of $\mathbf s(N)\to\mathbf s(M)\to\mathbf s(M/N)$,
the term before $H^0(\mathbf s(N))$ vanishes by the amplitude in
Definition~\ref{def:s-torsion-vocab}, so
$H^0(\mathbf s(N))\to H^0(\mathbf s(M))$ is injective and
$\degD(N)\le\ell_WH^0(\mathbf s(M))$.  The realized slopes are therefore bounded above, with denominators
in $\{1,\dots,\rkD(M)\}$, and two distinct such rationals differ by
at least $\rkD(M)^{-2}$.  So the realized slopes have a maximum
$\mu_{\max}$.

Write $\theta:=\theta_{\mu_{\max}}$ for the weight of
Definition~\ref{def:weight}.  It is nonpositive on every subobject $A\subseteq M$:
$\theta(A)\le\theta(A^{\mathrm{sat}})$ by
Proposition~\ref{prop:torsion-structure}(ii), and
$\theta(A^{\mathrm{sat}})\le0$ by maximality of $\mu_{\max}$.  Moreover $\theta(A)=0$ exactly when $A$ is zero or of slope
$\mu_{\max}$.

Let $N_1,N_2$ be saturated of slope $\mu_{\max}$ and put
$N_1\vee N_2:=(N_1+N_2)^{\mathrm{sat}}$.  Additivity along
$0\to N_1\cap N_2\to N_1\oplus N_2\to N_1+N_2\to0$, followed by
$\theta(N_1+N_2)\le\theta(N_1\vee N_2)$, gives
\[
 \theta(N_1\vee N_2)+\theta(N_1\cap N_2)\ \ge\ \theta(N_1)+\theta(N_2)=0 .
\]
Both terms on the left are nonpositive, so both vanish and $N_1\vee N_2$ is
saturated of slope $\mu_{\max}$.

Since $\Delta$ is noetherian (Remark~\ref{rem:chain-conditions}), the saturated
subobjects of slope $\mu_{\max}$ have a member $M^{+}$ maximal under inclusion.
For any other one $N$, the join $N\vee M^{+}$ again has slope $\mu_{\max}$ and
contains $M^{+}$, so maximality gives $N\vee M^{+}=M^{+}$ and $N\subset M^{+}$.
Semistability of $M^{+}$ follows at once: a nonzero $N'\subset M^{+}$ is a
subobject of $M$, so $\theta(N')\le0$, that is $\muD(N')\le\muD(M^{+})$.

Finally let $Q\subset M/M^{+}$ be nonzero saturated, with inverse image $P$.
Additivity gives $\theta(P)=\theta(M^{+})+\theta(Q)=\theta(Q)$.  If
$\muD(Q)\ge\mu_{\max}$ then $\theta(Q)\ge0$, hence $\theta(P)=0$ and
$\muD(P)=\mu_{\max}$.  As $M/P\cong(M/M^{+})/Q$ is a diagonal domino, $P$ is
saturated, so $P\subset M^{+}$ and $Q=0$.
\end{proof}

\begin{theorem}[Harder--Narasimhan filtration]
\label{thm:hn-existence}
\label{thm:extended-hn}
Every nonzero $M\in\Delta_{\mathrm{tor}}$ has a unique filtration
\[
  0=M_0\subset M_1\subset\cdots\subset M_s=M
\]
whose quotients are semistable of strictly decreasing slope in
$\mathbb Q\cup\{+\infty\}$.  The value $+\infty$ occurs exactly when
$M_{\mathrm{ft}}\ne0$, and then $M_1=M_{\mathrm{ft}}$.  The remaining steps are the preimages
under $\pi:M\twoheadrightarrow M/M_{\mathrm{ft}}$ of those for $M/M_{\mathrm{ft}}$.
\end{theorem}

\begin{proof}
Take first a diagonal domino $M$, where all slopes are finite.  Existence
follows by iterating Proposition~\ref{prop:mds}.  The rank drops at each step because
$\rkD(M^{+})>0$ by Proposition~\ref{prop:torsion-structure}(i), so the iteration
stops after at most $\rkD(M)$ steps.  The maximal slope decreases at each step,
so the quotients have strictly decreasing slope.

For uniqueness, let $0=M_0\subset\cdots\subset M_s=M$ be any filtration with
semistable quotients of strictly decreasing slope.  Subobjects of $M$ are again
diagonal dominoes, so $M_1$ is nonzero of positive rank, and the slopes
decreasing strictly, every quotient has finite slope and positive rank.  By
Proposition~\ref{prop:stability-enlargement} the quotients lie in $\DomD$,
hence so does $M/M_1$ by Lemma~\ref{lem:torsion-pair}.  Thus $M_1$ is saturated
and $\muD(M_1)\le\mu_{\max}$.  Let $a$ be least with $M^{+}\subset M_a$.  If
$a>1$ the induced map $M^{+}\to M_a/M_{a-1}$ is nonzero, while
$\muD(M^{+})=\mu_{\max}\ge\muD(M_1)>\muD(M_a/M_{a-1})$, contradicting
Lemma~\ref{lem:hom-vanishing}.  So $M^{+}\subset M_1$, and semistability of
$M_1$ gives $\mu_{\max}\le\muD(M_1)$, hence $\muD(M_1)=\mu_{\max}$.
Proposition~\ref{prop:mds} then gives $M_1\subset M^{+}$, so $M_1=M^{+}$, and
induction applies to $M/M^{+}$.

For general $M$, consider any filtration as in the statement.  Strict decrease
allows at most one slope-$+\infty$ step, the first, so write $M'$ for it, with
$M'=0$ when none occurs.  It has rank zero, hence is finite torsion, so $M'\subseteq M_{\mathrm{ft}}$.  Conversely the factors above $M'$ are semistable of finite slope, hence
lie in $\DomD$, and $M/M'\in\DomD$ as before.  The torsion pair makes the
image of $M_{\mathrm{ft}}$ in $M/M'$ vanish, so $M_{\mathrm{ft}}\subseteq M'$ and $M'=M_{\mathrm{ft}}$.  The remaining steps are the preimages of a filtration of $M/M_{\mathrm{ft}}$ as
in the first part, with first factor $M_{\mathrm{ft}}$ semistable of slope $+\infty$, so existence and
uniqueness for $M$ follow from the domino case.
\end{proof}

\subsection{Comparison with Ekedahl's type filtration}
\label{subsec:type-fixed-slope}

Ekedahl's type filtration (Proposition~\ref{prop:type-filtration})
has steps indexed by integers and steps indexed by half-integers,
and the graded pieces at the two types of index behave
differently.  At an
integer index the graded piece is pure of type $r$, and such objects
are simple to describe: up to Breuil--Kisin twist they are built from the single
object $U_r$ (Proposition~\ref{thm:integer-simples}).  At a half-integer index the
piece lies strictly between two consecutive types, and here Ekedahl's analysis
is combinatorial, with a classification only under his weak simplicity
\cite[Theorem~III.2.7]{ekedahl3}.  The Harder--Narasimhan filtration explains
the difference.  Theorem~\ref{thm:refinement}(iii) gives
$\FilD^rM=F^{\ge r}_{\HN}M$ and $\FilD^{r+1/2}M=F^{>r}_{\HN}M$, so the piece at
$r$ is semistable with $\muD=r$, while the piece at $r+\tfrac12$ has all its HN
slopes in $(r,r+1)$.  Every rational number in that interval is the slope of a
stable object (Section~\ref{sec:stable}).  The present section makes this
precise, and Section~\ref{sec:stable} studies
$\Domss(\lambda)$ (Theorem~\ref{thm:fixed-slope}) for each fixed
$\lambda$.

The input is Theorem~\ref{thm:refinement}(i)--(ii), which computes
the slopes of the graded pieces at integer and at half-integer
indices.  Both computations rest on the behaviour of the degree under
Nygaard modification, which normalizes an arbitrary integer type to $0$.  The
degree is computed through the Postnikov pieces, so we begin there.

\begin{lemma}[Postnikov degree decomposition]
\label{lem:postnikov}
Let $M\in\DomD$.  For $n\in\mathbb Z$, write $\Dom^n(M)$ for the cohomology
object of $M$ whose differential enters internal degree $n$, shifted into
internal degrees $0,1$.  Explicitly $\Dom^n(M)=H^{1-n}(M)(n-1)$, and we
call it the $n$-th \emph{Postnikov layer} of $M$.  It is an
ordinary domino with $\rkD(\Dom^n(M))=\rkD^n(M)$, so that
\[
 \rkD(M)=\sum_n\rkD(\Dom^n(M)),\qquad
 \degD(M)=\sum_n\degD(\Dom^n(M)).
\]
\end{lemma}

\begin{proof}
Both formulas are formal.  The diagonal condition puts $H^{1-n}(M)$ in internal
degrees $(n-1,n)$, so its differential enters degree $n$ and $\Dom^n(M)$ is an
ordinary domino \cite[Proposition~III.3.1]{ekedahl3} with
$\Dom^n(M)^0=H^{1-n}(M)^{n-1}$.  The rank formula follows, since
$T^{n-1,1-n}$ depends only on that module
(Definition~\ref{def:domino-numbers}).  The Postnikov piece carrying
$H^{1-n}(M)$ is $H^{1-n}(M)[n-1]$, and
$\Dom^n(M)=\bigl(H^{1-n}(M)[n-1]\bigr)\{n-1\}$.  Since $\mathbf s$ totalizes,
it is insensitive to the diagonal-preserving twist, $\mathbf s(N\{i\})\simeq\mathbf s(N)$,
so the piece and $\Dom^n(M)$ have the same $\mathbf s$-realization.  The degree formula is then additivity of
$\chi_W\circ\mathbf s$ along the Postnikov tower.
\end{proof}

\begin{lemma}[Rank and degree under Nygaard modification]
\label{lem:nygaard-shift}
Let $M\in\Delta_{\mathrm{tor}}$ and $\varphi\in\Phi$.  Then
\[
  \rkD^i(M(\varphi))=\rkD^i(M),\qquad
  \degD(M(\varphi))=\degD(M)+\sum_i\varphi(i)\rkD^i(M).
\]
In particular the modification $(r\mathbf1)$ shifts the slope,
$\muD(M(r\mathbf1))=\muD(M)+r$.
\end{lemma}

\begin{proof}
The rank is unchanged because domino numbers depend only on the underlying
$R^0$-module and are invariant under Frobenius twists
\cite[Lemma~III.2.6.1(ii)]{ekedahl3}.  For the degree, Frobenius
twists preserve $W$-length, so $\degD$ does not see the twists
$\sigma_*^{-s_i(\varphi)}$ of Definition~\ref{def:autoequivalence}, and only the
differentials $F^{\varphi(i)}d$ matter.  Modifications act degreewise,
hence commute with the Postnikov truncations, so
Lemma~\ref{lem:postnikov} reduces the claim to a single piece $\Dom^n(M)$.
There the modification replaces the differential by $F^{\varphi(n)}d$, which
raises the degree by $\varphi(n)\rkD^n(M)$
\cite[Proposition~III.3.3]{ekedahl3}.  Summing over $n$ gives the formula for
$M\in\DomD$, and \eqref{eq:finite-torsion-radical} extends it to
$\Delta_{\mathrm{tor}}$, since the finite-torsion part has zero rank and
unchanged Euler characteristic.
\end{proof}

The modifications $(r\mathbf1)$ preserve the diagonal heart.  Beyond
the properties of Proposition~\ref{prop:type-filtration}, the middle
quotient $\FilD^0M/\FilD^{1/2}M$ is $\mathbf s$-acyclic
\cite[Theorem~III.2.3]{ekedahl3}.

For $a\in\mathbb Q\cup\{+\infty\}$ write $F^{\ge a}_{\HN}M$ and $F^{>a}_{\HN}M$
for the largest steps of the Harder--Narasimhan filtration of
Theorem~\ref{thm:extended-hn} whose graded slopes are all $\ge a$, respectively
all $>a$.

\begin{theorem}[Refinement of Ekedahl's type filtration]
\label{thm:refinement}
\label{prop:adjacent-type-slopes}
\label{prop:integer-type}
Let $M\in\DomD$.
\begin{enumerate}
\item If $M$ is pure of integer type $r$, it is semistable of slope $r$.
\item If $M$ is pure of type $r+\tfrac12$, every HN slope of $M$ lies strictly
  in $(r,r+1)$.
\item For every $r\in\mathbb Z$,
  \[
    \FilD^r M=F^{\ge r}_{\HN}M,\qquad
    \FilD^{r+1/2}M=F^{>r}_{\HN}M.
  \]
\end{enumerate}
\end{theorem}

\begin{proof}
The modification $(-r\mathbf1)$ carries pure type $a$ to pure type $a-r$ and
lowers every finite slope by $r$ (Lemma~\ref{lem:nygaard-shift}), so in (i) and
(ii) we may assume $r=0$.

For (i), $\FilD^0M=M$ and $\FilD^{1/2}M=0$, so $M$ is the
$\mathbf s$-acyclic quotient $\FilD^0M/\FilD^{1/2}M$ recalled above and
$\degD(M)=0$.  If
$N\subset M$ is saturated, the amplitude in
Definition~\ref{def:s-torsion-vocab} makes
$H^0(\mathbf s(N))\to H^0(\mathbf s(M))=0$ injective, which gives
$\degD(N)=-\ell_WH^1(\mathbf s(N))\le0$.  Thus $M$ is semistable by
Proposition~\ref{prop:stability-enlargement}.

For (ii), let $N\subset M$ be nonzero.  Functoriality gives $\FilD^1N\subset\FilD^1M=0$, so
$\FilD^0(N(-\mathbf1))=0$, hence $\FilD^{1/2}(N(-\mathbf1))=0$, and
$N(-\mathbf1)$ is $\mathbf s$-$1$-torsion.  Its
degree is negative unless it is $\mathbf s$-acyclic, which would make it pure
of type $0$, contrary to $\FilD^0(N(-\mathbf1))=0$.  So
$\degD(N)-\rkD(N)=\degD(N(-\mathbf1))<0$.  Dually, a nonzero saturated quotient $Q$ of
$M$ has $\FilD^{1/2}Q=Q$, since an $\mathbf s$-$1$-torsion quotient
of $Q$ is one of $M$, hence $\FilD^0Q=Q$ and is $\mathbf s$-$0$-torsion and not
$\mathbf s$-acyclic, so $\degD(Q)>0$.  The first HN step of $M$ is a subobject
and the last HN factor a saturated quotient
(Proposition~\ref{prop:stability-enlargement}), so these two bounds place every
HN slope in $(0,1)$.

For (iii), the graded piece at type $a$ is pure of type $a$, so by (i) and (ii)
its slopes lie in $\{a\}$ for $a\in\mathbb Z$ and in
$(a-\tfrac12,a+\tfrac12)$ otherwise.  These windows are disjoint and increase
with $a$, so refining each piece by its HN filtration gives a filtration of $M$
with semistable quotients of strictly decreasing slope.
Theorem~\ref{thm:hn-existence} identifies it with the HN filtration, and the
two formulas are its steps at $a=r$ and $a=r+\tfrac12$.
\end{proof}

For $\lambda\in\mathbb Q\cup\{+\infty\}$ let
$\Delta_{\mathrm{tor}}^{ss}(\lambda)$ be the full subcategory of
$\Delta_{\mathrm{tor}}$ on the zero object and the semistable objects of slope
$\lambda$.  At finite $\lambda$ its objects are diagonal dominoes
(Proposition~\ref{prop:stability-enlargement}) and we write $\Domss(\lambda)$,
and at $\lambda=+\infty$ it is $\Delta_{\mathrm{ft}}$.

\begin{theorem}[Fixed-slope categories]
\label{thm:fixed-slope}
\label{def:fixed-slope-category}
\label{prop:strictness}
For every $\lambda\in\mathbb Q\cup\{+\infty\}$ the subcategory
$\Delta_{\mathrm{tor}}^{ss}(\lambda)$ is an abelian category of
finite length whose kernels and cokernels are formed in
$\Delta_{\mathrm{tor}}$.  It is closed under extensions in
$\Delta_{\mathrm{tor}}$, and its simple objects are exactly its
stable objects.
\end{theorem}

\begin{proof}
The inputs are additivity (Lemma~\ref{lem:additivity}), the saturation
inequality of Proposition~\ref{prop:torsion-structure}(ii) with its equality
case, the vanishing of the rank only on $\Delta_{\mathrm{ft}}$
(Proposition~\ref{prop:torsion-structure}(i)), and the torsion pair
(Lemma~\ref{lem:torsion-pair}).  The rest is formal.

At $\lambda=+\infty$ the category is the Serre subcategory
$\Delta_{\mathrm{ft}}$ of Definition~\ref{def:finite-torsion},
whose objects are finite direct sums of Breuil--Kisin twists of
finite-length Dieudonn\'e modules.  It is therefore
abelian, closed under extensions and of finite length, and its
simple objects are the stable ones
(Proposition~\ref{prop:stability-enlargement}).

Let $\lambda$ now be finite, let $f:E\to F$ be a morphism in $\Domss(\lambda)$
with image $I$, and write $\theta:=\theta_\lambda$.  Semistability gives
$\theta(\ker f)\le0$ and $\theta(I^{\mathrm{sat}})\le0$, additivity gives
$\theta(I)\ge0$, and the saturation inequality gives
$\theta(I)\le\theta(I^{\mathrm{sat}})$.  All three weights vanish,
and the equality case gives $I=I^{\mathrm{sat}}$.  So $\operatorname{coker}f$
lies in $\DomD$, and it and $\ker f$ have slope $\lambda$.  Both are
semistable, since a saturated subobject of $\ker f$ is saturated in $E$ by
$E/\ker f\simeq I$, and one of $\operatorname{coker}f$ has saturated preimage
in $F$ of the same weight.

An extension $0\to A\to E\to B\to0$ of objects of $\Domss(\lambda)$ has middle
term a diagonal domino of slope $\lambda$, and every $N\subset E$ satisfies
$\theta(N)=\theta(N\cap A)+\theta(I)\le\theta(I^{\mathrm{sat}})\le0$ for $I$
its image in $B$.

A proper inclusion in $\Domss(\lambda)$ has nonzero cokernel, hence of positive
rank, so the rank bounds the length of a chain.  The subobjects of $E$ in $\Domss(\lambda)$ are
its saturated subobjects of slope $\lambda$, so simplicity and stability both
say that $E$ has no nonzero proper one
(Proposition~\ref{prop:stability-enlargement}).
\end{proof}

\subsection{The stability condition and \texorpdfstring{$F$}{F}-gauges}
\label{subsec:transport}

We first upgrade the HN formalism on $\Delta_{\mathrm{tor}}$ to a Bridgeland
stability condition on $\Dtor$ (Theorem~\ref{thm:bridgeland-rotation}).  The $F$-gauge heart is then the tilt
of the diagonal heart at the torsion pair
$(\Delta_{\mathrm{tor}}^{\HN>0},\Delta_{\mathrm{tor}}^{\HN\le0})$, and in the
language of phases the tilt raises the phase window by $\tfrac12$
(Theorem~\ref{thm:gauge-heart-slope-zero}).  Finally, rank and degree are
computed from the $F$-gauge itself
(Proposition~\ref{prop:gauge-euler-invariants}).

We recall Bridgeland's definitions
\cite[Definitions~3.3, 5.1, and~5.7]{Bridgeland07}.  Let $\mathscr D$ be a
triangulated category.  A \emph{slicing} $\mathcal P$ of $\mathscr D$ assigns
to each $\phi\in\mathbb R$ a full additive subcategory
$\mathcal P(\phi)\subset\mathscr D$, compatibly with shift,
$\mathcal P(\phi+1)=\mathcal P(\phi)[1]$, and with
$\operatorname{Hom}(\mathcal P(\phi_1),\mathcal P(\phi_2))=0$ for
$\phi_1>\phi_2$.  It is further required that every nonzero object of $\mathscr D$ be a finite
iterated extension of objects of the $\mathcal P(\phi)$, with strictly
decreasing phases.  This extension is the HN filtration of the object.  A \emph{stability condition} $\sigma=(Z,\mathcal P)$ on $\mathscr D$ is a
slicing together with a group homomorphism $Z:K_0(\mathscr D)\to\mathbb C$
such that every nonzero $E\in\mathcal P(\phi)$ satisfies
$Z(E)\in\mathbb R_{>0}e^{i\pi\phi}$.  The map $Z$ is the \emph{central
charge}, and the nonzero objects of $\mathcal P(\phi)$ are the
\emph{semistable} objects of \emph{phase} $\phi$.  The condition is \emph{locally
finite} if the quasi-abelian categories
$\mathcal P((\phi-\varepsilon,\phi+\varepsilon))$ have finite length for some
uniform $\varepsilon>0$, and it has the \emph{support property} with respect
to a homomorphism $v:K_0(\mathscr D)\to\mathbb Z^2$ through which $Z$ factors
if $\|v(E)\|\le C\,|Z(E)|$ for a fixed norm, a fixed constant $C$, and every
semistable $E$.

Here the data are the following.  The rank and the degree of
Definition~\ref{def:rank-degree-slope} are additive
(Lemma~\ref{lem:additivity}), hence define homomorphisms on
$K_0(\Delta_{\mathrm{tor}})$, and the bounded diagonal $t$-structure
identifies $K_0(\Delta_{\mathrm{tor}})$ with $K_0(\Dtor)$, where the class of
a complex $E$ is $\sum_i(-1)^i[\widetilde H^i(E)]$.  Put
\[
 Z_\Delta:=-\degD+i\,\rkD:
 K_0(\Dtor)\longrightarrow\mathbb Z[i]\subset\mathbb C.
\]
For
$0<\phi\le1$ let $\mathcal P_\Delta(\phi)$ consist of zero and the objects
that are semistable in the sense of Definition~\ref{def:semistable} with
$Z_\Delta(E)\in\mathbb R_{>0}e^{i\pi\phi}$, set
$\mathcal P_\Delta(\phi+n)=\mathcal P_\Delta(\phi)[n]$ for $n\in\mathbb Z$, and
write $\mathcal P_\Delta(I)$ for the extension closure of the
$\mathcal P_\Delta(\phi)$ with $\phi$ in an interval $I$.  Finally let
$\Delta_{\mathrm{tor}}^{\HN>0}$ and $\Delta_{\mathrm{tor}}^{\HN\le0}$ be the
full subcategories of objects whose HN slopes are all $>0$, respectively all
$\le0$, the zero object lying in both.

\begin{theorem}[The stability condition]
\label{thm:bridgeland-rotation}
The pair $\sigma_\Delta=(Z_\Delta,\mathcal P_\Delta)$ is a Bridgeland
stability condition on $\Dtor$ with heart
$\Delta_{\mathrm{tor}}=\mathcal P_\Delta((0,1])$.  On the heart, phase and
slope determine each other: a positive-rank semistable object $E$ has
$\muD(E)=-\cot(\pi\phi_\Delta(E))$, and the semistables of phase $1$ are the
finite-torsion objects.  In particular
\[
 \Delta_{\mathrm{tor}}^{\HN>0}
   =\mathcal P_\Delta((1/2,1]),\qquad
 \Delta_{\mathrm{tor}}^{\HN\le0}
   =\mathcal P_\Delta((0,1/2]).
\]
Moreover $\sigma_\Delta$ is locally finite and has the support property with
respect to $(\degD,\rkD)$.
\end{theorem}

\begin{proof}
The diagonal $t$-structure restricts to a bounded $t$-structure on $\Dtor$
with heart $\Delta_{\mathrm{tor}}$.  By
\cite[Proposition~5.3]{Bridgeland07}, the first sentence then amounts to two
properties of $Z_\Delta$ on this heart: nonzero objects have charge in the
upper half-plane or on the negative real axis, and HN filtrations exist.
For the first, $\operatorname{Im}Z_\Delta(E)=\rkD(E)\ge0$, and at rank zero
$Z_\Delta(E)=-\degD(E)<0$ by Proposition~\ref{prop:torsion-structure}(i).
Comparing real and imaginary parts gives $\muD=-\cot(\pi\phi_\Delta)$ at
positive rank and phase $1$ at rank zero.  The phase is therefore an
increasing function of the slope, so the two notions of semistability agree,
the filtration of Theorem~\ref{thm:extended-hn} is the required HN
filtration, and the two windows follow.

For local finiteness, let $I$ be a phase interval of length $<1$.  The
charges of the nonzero objects of $\mathcal P_\Delta(I)$ lie in the lattice
$\mathbb Z[i]$ and in a convex cone of angle less than $\pi$, so some
linear functional on $\mathbb R^2$, integral on the lattice and positive on
the cone, takes values in $\mathbb Z_{\ge1}$ on them.  Charges add along the
strict short exact sequences of the quasi-abelian category
$\mathcal P_\Delta(I)$ \cite[\S4]{Bridgeland07}, so the functional strictly
increases along strict inclusions, and its value at $E$ bounds the length of
every strict chain in $E$.  Integrality of rank and degree is the one input
here.  The support property is formal: $|Z_\Delta(E)|$ is the Euclidean norm
of $(\degD E,\rkD E)$, so it holds with constant $1$.
\end{proof}

Proposition~\ref{prop:ekedahl-heart-hrs} presents Ekedahl's heart
$\mathcal G$ as the HRS tilt of $\Delta$ at the torsion pair defined by his
$G^2$, with no reference to a slope.  The next theorem identifies the
restriction of this torsion pair to $\Delta_{\mathrm{tor}}$ with
$(\Delta_{\mathrm{tor}}^{\HN>0},\Delta_{\mathrm{tor}}^{\HN\le0})$.  In phase
language the cut is at phase $\tfrac12$, and the tilt rotates the window
$(0,1]$ to $(1/2,3/2]$.

\begin{theorem}[Phase rotation and torsion $F$-gauges]
\label{thm:gauge-heart-slope-zero}
\label{thm:phase-rotation}
\label{cor:standard-gauge-hn}
The heart $\mathcal G_{\mathrm{tor}}$ admits the two descriptions
\[
 \mathcal G_{\mathrm{tor}}
 =\bigl\langle
   \Delta_{\mathrm{tor}}^{\HN\le0}[1],
   \Delta_{\mathrm{tor}}^{\HN>0}
  \bigr\rangle
 =\mathcal P_\Delta((1/2,3/2]).
\]
The $F$-gauge heart $\Coh_{\mathrm{tor}}(k^{\mathrm{Syn}})$ therefore carries
the rotated stability condition, with central charge
$Z_{\mathsf{FG}}=-iZ_\Delta=\rkD+i\degD$.  For $M\in\Delta_{\mathrm{tor}}$,
the $F$-gauge cohomology of $S(M)$ is $S(F^{>0}_{\HN}M)$ in degree $0$,
$S\bigl((M/F^{>0}_{\HN}M)[1]\bigr)$ in degree $1$, and zero elsewhere.  In
particular $S(M)$ lies in the $F$-gauge heart precisely when every HN slope
of $M$ is positive.  The stable torsion $F$-gauges are $S(M)$ for
$\Delta$-stable $M$ of positive slope and $S(M[1])$ for those of nonpositive
slope.
\end{theorem}

\begin{proof}
The key step identifies Ekedahl's $G^2$ with the positive-slope HN
truncation: $G^2M=F^{>0}_{\HN}M$ for every $M\in\Delta_{\mathrm{tor}}$.  For a
diagonal domino $U$, we have $G^2U=\FilD^{1/2}U$
(Proposition~\ref{prop:type-filtration}), and Theorem~\ref{thm:refinement}(iii)
gives $G^2U=F^{>0}_{\HN}U$.  For finite-torsion $T$, both sides are $T$ itself: $T$ has vanishing
$\mathbf s$-$1$-torsion quotient \cite[Lemma~I.4.2(iii)]{ekedahl3},
so $G^2T=T$, and $F^{>0}_{\HN}T=T$ at slope $+\infty$.  For general $M$,
both $G^2M$ and $F^{>0}_{\HN}M$ are preimages along
$\pi:M\twoheadrightarrow M/M_{\mathrm{ft}}$: for $F^{>0}_{\HN}M$ this is the description in
Theorem~\ref{thm:extended-hn}, and for $G^2M$ it holds because $M_{\mathrm{ft}}$ lies in
$\mathcal T$, so that
$\pi^{-1}(G^2(M/M_{\mathrm{ft}}))$ lies in $\mathcal T$ with quotient in $\mathcal F$ and is
therefore $G^2M$.  The domino case applied to $M/M_{\mathrm{ft}}$ concludes.

The rest is bookkeeping.  Intersecting with $\Delta_{\mathrm{tor}}$ turns the
torsion pair of Proposition~\ref{prop:ekedahl-heart-hrs} into
$(\Delta_{\mathrm{tor}}^{\HN>0},\Delta_{\mathrm{tor}}^{\HN\le0})$ and its HRS
description into the first equality.  The windows of
Theorem~\ref{thm:bridgeland-rotation} turn the first equality into the second,
with the right endpoint closed because slope zero lands at phase $\tfrac32$.
Multiplying $Z_\Delta$ by $-i$ returns the phases to $(0,1]$, and transport
along $S$ gives the rotated stability condition on the $F$-gauge heart.

Finally, the truncation sequence
$0\to F^{>0}_{\HN}M\to M\to M/F^{>0}_{\HN}M\to0$, followed by $F$-gauge
cohomology
and the $t$-exactness of $S$, gives the two cohomology formulas.  Since $S$ is
an equivalence, $H^1_{\mathsf{FG}}(S(M))$ vanishes exactly when
$M=F^{>0}_{\HN}M$, that is, when every HN slope of $M$ is positive.

The rotation preserves each slice and hence its simple objects, so the
$Z_{\mathsf{FG}}$-stables of phase in $(0,1]$ are the $\sigma_\Delta$-stables
of phase in $(1/2,3/2]$: the objects $S(M)$ for stable $M$ of slope $>0$, and
$S(M[1])$ for stable $M$ of slope $\le0$.
\end{proof}

We can finally translate the Harder--Narasimhan formalism on
$\Delta_{\mathrm{tor}}$ to the $F$-gauge heart
$\Coh_{\mathrm{tor}}(k^{\mathrm{Syn}})$: its rank and degree are computed
from the $F$-gauge itself.  For a coherent torsion $F$-gauge $G$
put
\[
 \rkFG(G):=\chi_W(G^{-\infty}),\qquad
 \degFG(G):=\sum_i\chi_W\bigl(\fib(t^{-\infty}:G^i\to G^{-\infty})\bigr),
\]
a finite sum, and $\muFG:=\degFG/\rkFG$.  Here $G^{-\infty}$ is a
finite-length module, so $\rkFG(G)=\ell_W(G^{-\infty})$.

\begin{proposition}
\label{prop:gauge-euler-invariants}
The slope $\muFG$ defines on $\Coh_{\mathrm{tor}}(k^{\mathrm{Syn}})$
the Harder--Narasimhan formalism obtained from
$\Delta_{\mathrm{tor}}$ by the phase rotation of
Theorem~\ref{thm:phase-rotation}: the invariants are additive,
$Z_{\mathsf{FG}}=-\degFG+i\,\rkFG$ on the $F$-gauge heart, and
$\muFG(S(M))=-1/\muD(M)$ for $M$ of positive HN slopes and
$\muFG(S(M[1]))=-1/\muD(M)$ for $M$ of nonpositive HN slopes, with
$-1/(+\infty)=0$ and
$-1/0=+\infty$.
\end{proposition}

\begin{proof}
The formulas are Euler characteristics of terms exact in $G$, so
both invariants are additive.  Write $G=S(E)$ with $E\in\Dtor$, and
extend $\degD$ and $\rkD$ additively along triangles.  Ekedahl's
construction gives $G^{-\infty}=\mathbf s(E)$ and
$G^i=\sigma_*\mathbf s(E(-\delta_i))$
\cite[Definition~II.3.1]{ekedahl3}.
Since Frobenius twists preserve $W$-length, $\rkFG(G)=\degD(E)$ and
$\chi_W\bigl(\fib(t^{-\infty})\bigr)=\degD(E(-\delta_i))-\degD(E)=-\rkD^i(E)$
(Lemma~\ref{lem:nygaard-shift}).  Almost all $\rkD^i(E)$ vanish, so
the sum is finite, $\degFG(G)=-\rkD(E)$, and
$-\degFG(G)+i\,\rkFG(G)=\rkD(E)+i\,\degD(E)=-iZ_\Delta(E)=Z_{\mathsf{FG}}(G)$.
The HN formalism on a heart is determined by the central charge, so
$\muFG$ induces the rotated one.
\end{proof}

Both slopes are used below, and \emph{HN slope} means $\muD$ unless we
name $\muFG$ explicitly.  In particular, in Section~\ref{sec:stable} we
study diagonal dominoes with all HN slopes in $(0,1)$, and their images
under $S$ lie in the $F$-gauge heart with no shift
(Theorem~\ref{thm:phase-rotation}).  So there an object of HN slope
$d/q$ with $0<d<q$ has $\muFG=-q/d$.

\section{Stable objects and the Christoffel classification}
\label{sec:stable}

This section classifies the stable objects of $\Delta_{\mathrm{tor}}$ and,
through the dictionary of Theorem~\ref{thm:phase-rotation}, the
stable torsion $F$-gauges.  By
Theorem~\ref{thm:fixed-slope} the former are the simple objects of the
categories $\Delta_{\mathrm{tor}}^{ss}(\lambda)$, diagonal dominoes when
$\lambda$ is finite and simple finite-torsion objects when
$\lambda=+\infty$.

The slopes $\lambda=+\infty$ and $\lambda\in\mathbb Z$ come first, and a
modification then reduces the remaining slopes to $(0,1)\cap\mathbb Q$.  At
the integers the classification reduces to Ekedahl.

\begin{proposition}
\label{thm:integer-simples}
For every $r\in\mathbb Z$, the simple objects of $\Domss(r)$ are exactly
the Breuil--Kisin twists $U_r\{m\}$, $m\in\mathbb Z$.
\end{proposition}

\begin{proof}
The modification $M\mapsto M(-r\mathbf1)$ reduces to $r=0$.  There the
objects are the $\mathbf s$-acyclic diagonal dominoes by
Theorem~\ref{thm:refinement}(i),(iii).  Both $\Domss(0)$ and Ekedahl's
category of such dominoes form kernels and cokernels in
$\Delta_{\mathrm{tor}}$, so they have the same subobjects and the same
simple objects (Theorem~\ref{thm:fixed-slope}).  By
\cite[Lemma~II.1.2.1]{ekedahl3} the simple objects are exactly the twists
$U_0\{m\}$.
\end{proof}

At slope $+\infty$ classical Dieudonn\'e theory takes over.

\begin{lemma}
\label{lem:simple-finite-dieudonne}
The simple objects of $\Delta_{\mathrm{tor}}^{ss}(+\infty)$ are the
Breuil--Kisin twists of the simple finite-length Dieudonn\'e
modules.\footnote{The Dieudonn\'e module of $\mu_p$ has bijective
$V$ and is not $V$-complete. Its $F$-gauge is $S(T)\{-1\}$, where
$T=(k,F=\sigma,V=0)$ is the Dieudonn\'e module of $\mathbb Z/p$.}
If $k$ is algebraically closed, the simple finite-length
Dieudonn\'e modules are those of $\mathbb Z/p$ and $\alpha_p$.
\end{lemma}

\begin{proof}
The category $\Delta_{\mathrm{tor}}^{ss}(+\infty)$ is
$\Delta_{\mathrm{ft}}$, so the first assertion is Ekedahl's decomposition
\cite[Theorem~I.1.12(i)]{ekedahl3}.  The second assertion is the
classical classification \cite{Manin63}: the Dieudonn\'e modules of
$\mathbb Z/p$ and $\alpha_p$ are $(k,F=\sigma,V=0)$ and $(k,F=V=0)$.
\end{proof}

The modification $X\mapsto X(-r\mathbf1)$ shifts every slope by $-r$
(Lemma~\ref{lem:nygaard-shift}), so we are left with
$\lambda\in(0,1)\cap\mathbb Q$.  Fix $\lambda=d/q$ with $0<d<q$ and
$\gcd(d,q)=1$.

\begin{lemma}
\label{lem:annihilated-by-p}
Every stable diagonal domino is annihilated by $p$ and is indecomposable.
\end{lemma}

\begin{proof}
A stable object $M$ is simple in the abelian category $\Domss(\lambda)$
(Theorem~\ref{thm:fixed-slope}).  A power of $p$ annihilates $M$
(Definition~\ref{def:torsion-category}), so multiplication by $p$ has a
nonzero kernel.  The kernel is a subobject, hence all of $M$, and $pM=0$.
An idempotent endomorphism of $M$ has kernel $0$ or $M$ and is therefore
$1$ or $0$, so $M$ is indecomposable.
\end{proof}

\subsection{Ekedahl's interval presentation}

The $F$-gauges $M(I,J)$
defined below are the images $S(U[1])$ of the $\mathbf s$-acyclic
diagonal dominoes $U$, semistable of slope zero, with the shift
dictated by Theorem~\ref{thm:phase-rotation}. The modified interval
objects $M(I,J;\epsilon)$ satisfying the source-sink condition
have every HN slope in $(0,1)$
(Lemmas~\ref{lem:normalization} and~\ref{lem:chain-endo}) and provide
the combinatorial model for this section. We show in
Theorem~\ref{thm:stable-classification} that the stable objects of
slope $\lambda$ are exactly the Breuil--Kisin twists of a single
$M(I,J;\epsilon)$.

\begin{definition}[Interval objects and modified interval objects]
\label{def:modified-interval}
Let $I=[m,n]$ be a finite interval of levels and $J\subseteq[m,n-1]$ a set
of edges.  Ekedahl's \emph{interval object} $M(I,J)$ is the $F$-gauge
with $M(I,J)^i=k$ for $i\in I$ and $0$ otherwise, in which
$u:M^i\to M^{i+1}$ is the identity for $i\in J$, $t:M^{i+1}\to M^i$ is the
identity for $i\in[m,n-1]\setminus J$, and every other component is zero
\cite[Definition~III.1.2]{ekedahl3}.  As there, $M(I,J)$ also denotes
the $\mathbf s$-acyclic diagonal domino $U$ with $S(U[1])$ the above
$F$-gauge (Theorem~\ref{thm:phase-rotation}).  The ambient category
determines the reading.  Drawing each identity as an
arrow gives one oriented edge per pair of consecutive levels, hence an
oriented graph on the level set $I$.  A vertex is a
\emph{source} if no arrow ends at it, and a \emph{sink} if no arrow starts
at it.  For example
\[
  M([1,3],\{2\})\ =\ \bigl(\,k\xleftarrow{\ t\ }k\xrightarrow{\ u\ }k\,\bigr)
  \ =\ \bigl(\,1\leftarrow2\rightarrow3\,\bigr),
\]
with the middle vertex a source and the two endpoints sinks.  For
$\epsilon:I\to\{0,1\}$, extended by zero outside $I$, the
\emph{modified interval object} is $M(I,J;\epsilon):=M(I,J)(\epsilon)$, with $\epsilon_j:=\epsilon(j)$, so
that $\epsilon=\sum_{\epsilon_j=1}\delta_j$ in the notation of
Definition~\ref{def:autoequivalence}.  The
\emph{source-sink condition} asks $\epsilon_i=1$ at every source and
$\epsilon_i=0$ at every sink.  In the oriented graph, a star
$i^*$ marks a vertex with $\epsilon_i=1$.\footnote{Ekedahl marks a
modified vertex with a vertical bar; we use a star.}  For example
$M([1,3],\{2\};\delta_2)=(1\leftarrow2^*\rightarrow3)$ satisfies the
source-sink condition.
\end{definition}

Ekedahl's move brings every presentation of a modified interval
object to a normal form, and the proof of
Proposition~\ref{prop:chain-presentation} below draws its subobjects and
numerical invariants from that form.

\begin{lemma}[Ekedahl's move]
\label{lem:normalization}
The move
\[
\begin{tikzcd}[column sep=large]
 i & (i+1)^{*} \arrow[l, "t"']
\end{tikzcd}
\qquad\longleftrightarrow\qquad
\begin{tikzcd}[column sep=large]
 i^{*} \arrow[r, "u"] & (i+1)
\end{tikzcd}
\]
is an isomorphism of the underlying diagonal dominoes. Under these moves, every
$M(I,J;\epsilon)$ satisfying the source-sink condition has a unique presentation in which all
arrows point left.
\end{lemma}

\begin{proof}
The first assertion is \cite[Proposition~III.2.6]{ekedahl3}.

Suppose an arrow points right, and choose vertices $a<b$ with
$a\rightarrow a+1\rightarrow\cdots\rightarrow b$ and with the edges
$(a-1,a)$ and $(b,b+1)$, where present, pointing left.  Every arrow at
$a$ leaves it and every arrow at $b$ ends there, so $a$ is a source with
$\epsilon(a)=1$ and $b$ is a sink with $\epsilon(b)=0$.  The first
$j\in[a,b-1]$ with $\epsilon_j=1$ and $\epsilon_{j+1}=0$ yields an edge
$j^*\to j+1$, where the displayed move applies.  The move decreases the
number of right arrows by one and preserves the source-sink condition,
checked at the two vertices it touches.  The moves therefore terminate
in finitely many steps, with all arrows then pointing left.

For uniqueness, set $I_j:=\sum_{i=m}^{j}\epsilon_i-\rho_j$, where
$\rho_j=1$ if the edge $(j,j+1)$ points right and $\rho_j=0$ otherwise.
For each $j$, the number $I_j$ is invariant under the moves.
With all arrows pointing left, $\rho\equiv0$ and
$I_j=\sum_{i=m}^{j}\epsilon_i$, so every letter is recovered as
$\epsilon_j=I_j-I_{j-1}$, with $I_{m-1}=0$.
Two leftward presentations related by the moves therefore carry the same
word.
\end{proof}

When all arrows point left, that is $J=\varnothing$, normalize $I$ to
begin at $1$ by a Breuil--Kisin twist, write $M(\epsilon)$ for
$M(I,\varnothing;\epsilon)$, and record it by its \emph{word}
$\epsilon_1\epsilon_2\cdots$.  The source-sink condition says the first
letter is $0$ and the last is $1$.  For
example, the move normalizes the modified example of
Definition~\ref{def:modified-interval}:
\[
 (1\leftarrow2^{*}\rightarrow3)\ \rightsquigarrow\ (1\leftarrow2\leftarrow3^{*})\ =\ 001 .
\]

\begin{remark}[Leftward and rightward normal forms]
\label{rem:two-normal-forms}
Running the move in the opposite direction normalizes $M(\epsilon)$
instead to a unique presentation with all arrows pointing right.  If
the leftward normal form has the word
$\epsilon=0\,\epsilon_2\cdots\epsilon_{m-1}\,1$ on $[1,m]$, then its
rightward normal form has the word
$\epsilon'=1\,\epsilon_2\cdots\epsilon_{m-1}\,0$.  Indeed, the invariant $I_j$ of the proof above equals
$\epsilon_1+\cdots+\epsilon_j$ in the leftward presentation and
$\epsilon'_1+\cdots+\epsilon'_j-1$ in the rightward one, for every
edge $j$.  Taking differences gives $\epsilon'_j=\epsilon_j$ for
$1<j<m$ and $\epsilon'_1=\epsilon_1+1=1$, and the degree
$\sum_i\epsilon_i$ is the same in both presentations
(Lemma~\ref{lem:dictionary}), so $\epsilon'_m=\epsilon_m-1=0$.  By
default, we will use the leftward normal form.
\end{remark}

\begin{lemma}
\label{lem:chain-endo}
A modified interval object $M(\epsilon)$ satisfying the source-sink
condition is indecomposable, annihilated by $p$, and has every HN slope
in $(0,1)$.
\end{lemma}

\begin{proof}
The interval object $M(I,\varnothing)$ is annihilated by $p$, and
modifications preserve both this and the endomorphism ring.  Under
Ekedahl's equivalence $S$, an endomorphism is a scalar at each
one-dimensional vertex, and every identity edge equates adjacent scalars
along the connected chain, so
$\operatorname{End}_\Delta(M(\epsilon))=k$
\cite[Propositions~III.1.1 and~III.1.3]{ekedahl3} and $M(\epsilon)$ is
indecomposable.  The slopes lie in
$(0,1)$ by Ekedahl's criterion \cite[Ch.~III, \S2, the discussion
preceding Proposition~2.6]{ekedahl3} and Theorem~\ref{thm:refinement}(ii).
\end{proof}

\begin{lemma}[Rank and degree]
\label{lem:dictionary}
For a modified interval object $M(I,J;\epsilon)$, the rank
$\rkD(M(I,J;\epsilon))$ is the number of vertices of $I$, and
$\degD(M(I,J;\epsilon))=\sum_i\epsilon_i$.
\end{lemma}

\begin{proof}
The underlying $\mathbf s$-acyclic interval object has rank equal to the
number of vertices and degree zero
\cite[Lemma~III.2.6.1]{ekedahl3}.  Modifications preserve rank, while each
modification raises degree by one (Lemma~\ref{lem:nygaard-shift}).
\end{proof}

\begin{lemma}
\label{lem:prefix-cut}
Let $M(\epsilon)$ be a modified interval object satisfying the source-sink
condition.  Every proper nonempty prefix of $I$ spans a saturated
subobject, with complementary quotient the modified interval object
on the complementary suffix.
\end{lemma}

\begin{proof}
A set of vertices with no arrow leaving it spans a subobject and a
complementary quotient.  Their
underlying interval objects are $\mathbf s$-acyclic, so their modifications remain
diagonal dominoes
\cite[Lemma~III.2.2(iii)]{ekedahl3}.  The inclusion is therefore saturated.
\end{proof}

Ekedahl gives a characterization of diagonal complexes of type
$\tfrac12$ via an explicit criterion on their associated $F$-gauges.\footnote{The
printed statement of \cite[Corollary~III.2.5.1]{ekedahl3} omits the necessary
conditions $t_ix=0$ and $u_{i-1}y=0$.
The category is labelled type $-\tfrac12$ there.
With the type convention of Proposition~\ref{prop:type-filtration},
it has type $\tfrac12$, corresponding to HN slopes in $(0,1)$.}

\begin{lemma}[{cf.\ \cite[Corollary~III.2.5.1]{ekedahl3}}]
\label{lem:boundary-condition}
The equivalence $S$ identifies the diagonal complexes of type $\tfrac12$,
that is with every HN slope in $(0,1)$, with the torsion coherent
$F$-gauges satisfying the \emph{boundary condition}
\[
 t_ix=0,\qquad u_{i-1}y=0,\qquad \tau(u^\infty x)=t^{-\infty}y
 \qquad\Longrightarrow\qquad x=y=0
\]
for every $i\in\mathbb Z$ and all $x\in G^{i+1}$ and $y\in G^{i-1}$.
\end{lemma}

\begin{proof}
By the Hom criterion in
\cite[proof of Proposition~III.2.5]{ekedahl3}, the category
in question is characterized by
$\operatorname{Hom}(S(U_1\{1-i\}),G)=0$ for every $i$.
By Example~\ref{ex:equivalence-examples}(iii), such a morphism
is determined by the images $x,y$ of the two tail generators
in levels $i+1,i-1$.
The arrows adjacent to the zero level $i$ give
$t_ix=0$ and $u_{i-1}y=0$, which imply $px=py=0$,
and compatibility with the gluing gives the third condition.
Conversely, any such pair extends uniquely to a morphism.
\end{proof}

The next proposition gives a family of examples of $F$-gauges
satisfying the lemma above, and in
Section~\ref{subsec:stable-classification} we will show that, up to
Breuil--Kisin twists, these are precisely the stable $F$-gauges of
slope $1/q$.

\begin{proposition}
\label{prop:one-star-gauge}
For $q\ge1$ write $M_q:=M([1,q],\varnothing;\delta_q)$.  Then
\[
 S(M_q)=
 \cdots\fgpair{0}{\sim}k
 \fgpair{0}{\sim}k
 \rightleftarrows
 \underbrace{0\rightleftarrows\cdots\rightleftarrows0}_{q\ {\rm levels}}
 \rightleftarrows k
 \fgpair{\sim}{0}k
 \fgpair{\sim}{0}\cdots ,
\]
with the zero levels at $[1,q]$ and $\tau$ the Frobenius
$\sigma\colon k\to k$.
\end{proposition}

\begin{proof}
The case $q=1$ is
$M_1=U_0(\delta_1)=U_1$, displayed in Example~\ref{ex:equivalence-examples}(iii).  For $q\ge2$
we use induction.  The move at the edge $(q-1,q)$ (Lemma~\ref{lem:normalization})
presents $M_q$ as $1\leftarrow\cdots\leftarrow(q-1)^{*}\rightarrow q$.
The last vertex $q$ by itself represents a saturated subobject isomorphic
to $U_0\{1-q\}$, with quotient $M_{q-1}$:
\[
 0\longrightarrow U_0\{1-q\}\longrightarrow M_q\longrightarrow
 M_{q-1}\longrightarrow 0 .
\]
The
$F$-gauge cohomology sequence of Theorem~\ref{thm:gauge-heart-slope-zero}
reads
\[
 0\longrightarrow S(M_q)\longrightarrow S(M_{q-1})\longrightarrow
 S\bigl(U_0\{1-q\}[1]\bigr)\longrightarrow 0 ,
\]
levelwise exact, and the last term is the single $k$ in level $q$,
by Example~\ref{ex:equivalence-examples}(iii).  By induction $S(M_{q-1})$ is the
display with $q-1$ zero levels, and the kernel vanishes in level $q$ and
agrees with $S(M_{q-1})$ in every other level, giving the display with
$q$ zero levels.
\end{proof}

Ekedahl claims the following proposition, without proof, in the course
of proving his classification of the weakly simple objects
\cite[proof of Theorem~III.2.7]{ekedahl3}.  It is the converse of
Lemma~\ref{lem:chain-endo}.  We give the following constructive proof,
from which Example~\ref{ex:sorted-path} reads off an explicit algorithm
carrying a word $\epsilon$ to the $F$-gauge $S(M(\epsilon))$.

\begin{proposition}
\label{prop:chain-presentation}
Every indecomposable diagonal domino $U$ annihilated by $p$ with every HN
slope in $(0,1)$ is isomorphic to a Breuil--Kisin twist of a modified
interval object $M(\epsilon)$ satisfying the source-sink condition.
\end{proposition}

\begin{proof}
Put $G=S(U)$: every HN slope of $U$ lies in $(0,1)$, so by
Lemma~\ref{lem:boundary-condition} $G$ is a torsion coherent
$F$-gauge satisfying the boundary condition, and it is annihilated
by $p$ and indecomposable.  After a Breuil--Kisin twist we may
assume that $G$ has level $[0,N]$.
Lemma~\ref{lem:regluing-normal-form} below presents $G$ as the
chain of interval gauges $Q_0,\dots,Q_{n+1}$ with one-dimensional
levels, joined by identity matchings, with sources satisfying
\[
\begin{cases}
 b\ge v_1+2,\quad v_1>v_2>\cdots>v_n,\quad v_n\ge c+2, & n\ge1,\\[1mm]
 b\ge c+3, & n=0.
\end{cases}
\]

We prove, by induction on $n$, that every $F$-gauge $G$ of this
shape is isomorphic to $S(M(\epsilon))$ for the word $\epsilon$
defined below.  Such a $G$ is indecomposable.  Its summands are
pairwise nonisomorphic, since their sources strictly decrease and the
two ends have distinct shapes, so modulo the nilpotent ideal of the
morphisms between distinct summands an endomorphism of $G$ is a
scalar $\lambda_j$ on each summand, and commuting with the identity
matchings forces $\lambda_{j+1}=\sigma(\lambda_j)$.  An idempotent
endomorphism therefore has every $\lambda_j$ equal to $0$ or to $1$
and constant along the path, and is $0$ or $1$.

Let $\epsilon$ be the word on $[c+1,b-1]$ with letters $1$ exactly at
$v_n,\dots,v_1$ and at $b-1$, and let $M(\epsilon)$ be the modified
interval object with this word.  The sorting makes the first letter
$0$ and the last letter $1$, so $M(\epsilon)$ satisfies the
source-sink condition.  We prove $G\simeq S(M(\epsilon))$, and the
equivalence $S$ then gives $U\simeq M(\epsilon)$.

For $n=0$, the two subgauges and the identity matching form the gauge
of Proposition~\ref{prop:one-star-gauge} with its zero levels placed
on $[c+1,b-1]$,
which is $S(M(\epsilon))$.

For $n\ge1$, we consider the sub-$F$-gauge $G'\subseteq G$ whose
shape is the open path
$(v_1+1\rightarrow\cdots\rightarrow N),Q_2,\dots,Q_{n+1}$.  The
inductive hypothesis gives $G'\simeq S(M(\epsilon'))$, with $\epsilon'$
the prefix of $\epsilon$ on $[c+1,v_1]$.  The quotient $G'':=G/G'$ is
the reglued pair
$(b\rightarrow\cdots\rightarrow N),(0\leftarrow\cdots\leftarrow v_1)$,
the gauge of Proposition~\ref{prop:one-star-gauge} placed on
$[v_1+1,b-1]$, so
$G''\simeq S(M(\epsilon''))$ with $\epsilon''$ the one-star suffix of
$\epsilon$ on $[v_1+1,b-1]$, a Breuil--Kisin twist of $U_1$ when
$b=v_1+2$.

Both $G$ and $S(M(\epsilon))$ are extensions of $G''$ by $G'$.
For $G$ this is the sequence $0\to G'\to G\to G''\to0$.  For
$S(M(\epsilon))$ it is the image under $S$ of the saturated exact
sequence $0\to M(\epsilon')\to M(\epsilon)\to M(\epsilon'')\to0$
of Lemma~\ref{lem:prefix-cut}.  Every HN slope of its three terms
is positive, for $M(\epsilon')$ and $M(\epsilon)$ by
Lemma~\ref{lem:chain-endo} and for $M(\epsilon'')$ because it
satisfies the source-sink condition or is a twist of $U_1$ of slope
$1$, so $S$ carries it to a short exact sequence of $F$-gauges
(Theorem~\ref{thm:gauge-heart-slope-zero}).  Both
middles are killed by $p$ and indecomposable, the middle $G$ by the
argument above and $S(M(\epsilon))$ by Lemma~\ref{lem:chain-endo}.  In the coherent $F$-gauges annihilated by $p$ the extensions of
$G''$ by $G'$ form the group $(k,+)$, with the automorphisms of
$G'$ transitive on the nonzero classes, a computation we defer to
Lemma~\ref{lem:extension-uniqueness} below.  The two classes are
nonzero, since the middles are indecomposable, so a pushout along
an automorphism of $G'$ identifies the middles, and
$G\simeq S(M(\epsilon))$, completing the induction.
\end{proof}

\begin{example}
\label{ex:sorted-path}
We illustrate how to convert a word $\epsilon$ satisfying the
source-sink condition into the $F$-gauge $G=S(M(\epsilon))$ as in
the proof above.

For the word $\epsilon=01011$ on $[1,5]$, the proof sorts
$G=S(M(01011))$ into the path $Q_0,\dots,Q_3$ with $b=6$, sources
$v_1=4$ and $v_2=2$, and $c=0$:
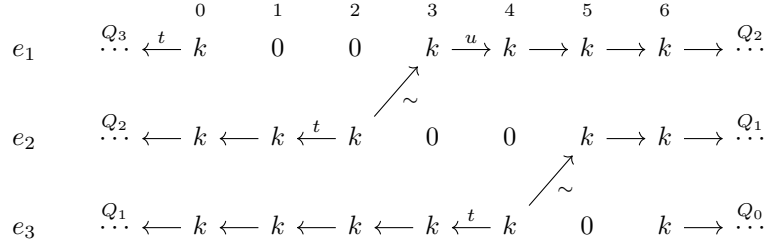
\[
\begin{tikzcd}[column sep=1.5em, row sep=1.2em]
 & & \scriptstyle0 & \scriptstyle1 & \scriptstyle2 & \scriptstyle3
 & \scriptstyle4 & \scriptstyle5 & \scriptstyle6 & \\[-1.7em]
 e_1 & \overset{Q_3}{\cdots} & k \arrow[l, "t"'] & 0 & 0
 & k \arrow[r, "u"] & k \arrow[r] & k \arrow[r] & k \arrow[r]
 & \overset{Q_2}{\cdots} \\
 e_2 & \overset{Q_2}{\cdots} & k \arrow[l] & k \arrow[l]
 & k \arrow[l, "t"'] \arrow[ur, "\sim"'] & 0 & 0 & k \arrow[r]
 & k \arrow[r] & \overset{Q_1}{\cdots} \\
 e_3 & \overset{Q_1}{\cdots} & k \arrow[l] & k \arrow[l] & k \arrow[l]
 & k \arrow[l] & k \arrow[l, "t"'] \arrow[ur, "\sim"'] & 0
 & k \arrow[r] & \overset{Q_0}{\cdots}
\end{tikzcd}
\]
The rows are indexed by the $k$-basis vectors $e_1,e_2,e_3$ at level
$\infty$, and the columns are indexed by the levels.  The slanted arrows sit at the sources
$v_1=4$ and $v_2=2$, and the matchings $\tau(e_\gamma)=e_\gamma$ join
$Q_0\to Q_1\to Q_2\to Q_3$.  In the induction step in the proof above, $G'$ is spanned by $e_1$
and $e_2$, which is $S(M(0101))$, and
$G''$ is the row $e_3$, with the single zero level $5$.

If we start with a general $\epsilon$ satisfying the source-sink
condition, written on $[1,q]$ with its letters $1$ at
$c_1<\cdots<c_d=q$, the same reading gives $G$ at once.  Put
$c_0:=0$ and $B_a:=[c_{a-1}+1,c_a]$.  The chain of
Lemma~\ref{lem:regluing-normal-form} has $c=0$,
$b=q+1$, and sources $v_j=c_{d-j}$.  The limits are
$G^{\pm\infty}=k^d$ with basis $e_1,\dots,e_d$ and
$\tau(1\otimes e_a)=e_a$, the level $G^i$ is spanned by the $e_a$
with $i\notin B_a$, and on the row $e_a$ the map $t$ is the identity
at the levels $i\le c_{a-1}$ and the map $u$ at the levels
$i\ge c_a+1$.  For $a\ge2$ the map $u$ carries in addition $e_a$ to
$e_{a-1}$ at the level $c_{a-1}$, the source $v_{d-a+1}$.  Exactly
one basis vector is therefore missing at each level of $[1,q]$,
which gives $\rkFG(G)=d$ and $\degFG(G)=-q$
(Lemma~\ref{lem:dictionary} and
Proposition~\ref{prop:gauge-euler-invariants}).  The word
$\epsilon=01011$ above has blocks $[1,2]$, $[3,4]$, and $[5,5]$.
\end{example}

We now state and prove the two lemmas used in the proof of
Proposition~\ref{prop:chain-presentation}.

\begin{lemma}
\label{lem:regluing-normal-form}
Let $G$ be an indecomposable $F$-gauge of level $[0,N]$, annihilated
by $p$ and satisfying the boundary condition of
Lemma~\ref{lem:boundary-condition}.  Then there are integers
$b>v_1>\cdots>v_n>c$ in $[0,N]$, with $b\ge v_1+2$ and $v_n\ge c+2$,
and with $b\ge c+3$ when $n=0$, such that $G$ is the regluing of the
chain of interval gauges
\[
 Q_0=(b\to\cdots\to N),\qquad
 Q_j=(0\leftarrow\cdots\leftarrow v_j\to\cdots\to N)\quad(1\le j\le n),
\]
and $Q_{n+1}=(0\leftarrow\cdots\leftarrow c)$, all with
one-dimensional levels, by the identity matchings from the
level-$N$ line of each $Q_j$ to the level-$0$ line of $Q_{j+1}$.
\end{lemma}

\begin{proof}
We fix the level $[0,N]$ throughout the proof and translate $G$
into a module over a semilinear string algebra $\Lambda_N$.  We
then apply the classification of
Bennett-Tennenhaus and Crawley-Boevey \cite{BTCB}.  We recall their
definitions and their theorem in the generality needed here, with
$k$ in place of their division ring.

Let $Q$ be a finite quiver with an automorphism $\sigma_a$ of $k$
attached to each arrow $a$.  The \emph{semilinear path algebra}
$k_\sigma Q$ has left $k$-basis the paths of $Q$, including the
trivial paths $e_i$ at the vertices.  Its product is the
concatenation of paths and its commutation rule is
$a\lambda=\sigma_a(\lambda)a$ for $\lambda\in k$
\cite[Section~2.1]{BTCB}.  A finite-dimensional left module is a
\emph{semilinear representation} of $Q$: a finite-dimensional
vector space $V_i$ at each vertex and a $\sigma_a$-semilinear map
$V_a\colon V_i\to V_j$ for each arrow $a\colon i\to j$.  Morphisms
are $k$-linear maps at the vertices commuting with the $V_a$.  Let
$Z$ be a set of paths of length at least two.  The quotient
$k_\sigma Q/(Z)$ is a \emph{semilinear string algebra} if at most
two arrows start and at most two arrows end at each vertex and if
each arrow $a$ admits at most one arrow $b$ with $ba\notin Z$ and
at most one arrow $c$ with $ac\notin Z$ \cite[Introduction and
Section~2.3]{BTCB}.  These are the semilinear clannish algebras of
loc.\ cit.\ without special loops.  For them the three conditions
on special loops in the Main Theorem of loc.\ cit.\ are empty.

Let $\Lambda_N$ be the semilinear string algebra of the quiver with
vertices $0,\dots,N$, arrows $u_i\colon i\to i+1$ and
$t_i\colon i+1\to i$ for $0\le i<N$, one arrow
$\widetilde\tau\colon N\to0$, and $Z=\{u_it_i,\,t_iu_i\}$.  It is
one: at each vertex at most two arrows start and at most two end,
and a path of length two avoids $Z$ exactly when it is
$u_{i+1}u_i$, $t_it_{i+1}$, $\widetilde\tau u_{N-1}$ or
$u_0\widetilde\tau$.  The automorphisms are dictated by the
$F$-gauges.  The maps $u$ and $t$ of a gauge are $W$-linear, so
$\sigma_a$ is the identity for $u_i$ and $t_i$.  The gluing
$\tau\colon\sigma^*M^N\to M^0$ is a $\sigma$-semilinear map
$M^N\to M^0$, so $\sigma_a$ is the Frobenius $\sigma$ for
$\widetilde\tau$.  Explicitly, as a $k$-ring,
\[
 \Lambda_N=k_\sigma\langle e_0,\dots,e_N,\,u,\,t,\,\widetilde\tau\rangle
 \big/\bigl(ut=tu=0,\ \ \widetilde\tau=e_0\widetilde\tau e_N,\ \
 \widetilde\tau\lambda=\sigma(\lambda)\widetilde\tau\bigr).
\]
Here the $e_i$ are the trivial paths, orthogonal idempotents with
sum $1$.  The elements $u=\sum_{0\le i<N}e_{i+1}ue_i$ and
$t=\sum_{0\le i<N}e_ite_{i+1}$ satisfy $ue_N=te_0=0$, and $e_i$,
$u$, $t$ commute with $k$.  A module $M$ decomposes as
$\bigoplus_ie_iM$ with $e_iM$ the level $M^i$, and $u_i=ue_i$,
$t_i=te_{i+1}$.  The $k$-basis of $\Lambda_N$ consists of the
paths $e_i$, $u^{j-i}e_i$, $t^{i-j}e_i$ and
$u^{j}\widetilde\tau u^{N}\cdots\widetilde\tau u^{N-i}e_i$, so
$\Lambda_N$ is infinite-dimensional.  A $\Lambda_N$-module is thus
a graded module $M=\bigoplus_{i=0}^NM^i$ over $k[u,t]/(ut)$ with
$\deg u=1$ and $\deg t=-1$, that is a quasi-coherent sheaf on the
mod $p$ fibre $[\operatorname{Spec}k[u,t]/(ut)/\mathbb G_m]$ of
the Nygaard filtered prismatization $k^{\mathcal N}$
(Definition~\ref{def:stacks}), together with a $\sigma$-semilinear
map $M^N\to M^0$ supplying the $F$-module gluing after adjoining
the constant left tail $M^0$ with $t=1,u=0$ and the constant right
tail $M^N$ with $u=1,t=0$.  So a
finite-dimensional $\Lambda_N$-module is the same as an $F$-module
(Definition~\ref{def:f-gauge}) of level $[0,N]$ annihilated by
$p$.  The vertex spaces are the levels $0,\dots,N$ and
$\widetilde\tau$ acts by $(t^{-\infty})^{-1}\tau\,\sigma^*(u^\infty)$.
The tails carry no further data, since one of $u$, $t$ is
invertible there and the other zero.  Morphisms match on both
sides and the $F$-gauges are the modules on which $\widetilde\tau$
acts bijectively.  So $G$ is an indecomposable finite-dimensional
$\Lambda_N$-module.

The finite-dimensional indecomposable modules over a semilinear
string algebra are classified by words.  A \emph{letter} is an
arrow $a$ or a formal inverse $a^{-1}$, whose head and tail are
those of $a$ interchanged.  A \emph{word} is a finite sequence
$w=w_1\cdots w_n$ of letters with $n\ge0$ such that the tail of
$w_i$ is the head of $w_{i+1}$, consecutive letters are never
mutually inverse, and no two consecutive letters form a path in
$Z$ or the inverse of one \cite[Section~2.4]{BTCB}.\footnote{Loc.\
cit.\ encodes the last two conditions by signs on the letters and
by relation-admissibility, and the further condition of
end-admissibility is empty without special loops.}  A
\emph{string} is a word up to inversion.  A \emph{band} is a
$\mathbb Z$-indexed periodic sequence of letters subject to the
same conditions, up to shift and inversion.  The \emph{string
module} $M(w)$ of a word of length $n$ has $k$-basis
$b_0,\dots,b_n$ with $b_i$ at the common vertex of $w_i$ and
$w_{i+1}$.  An arrow $a$ carries $b_i$ to $b_{i-1}$ if $w_i=a$, to
$b_{i+1}$ if $w_{i+1}=a^{-1}$, and to $0$ otherwise
\cite[Definitions~2.16 and~2.17, Lemmas~2.19 and~2.21]{BTCB}.  For
a band $w$ of period $n$ the same recipe on a basis
$(b_i)_{i\in\mathbb Z}$ defines $M(w)$.  It is a free right module
of rank $n$ over the skew Laurent ring $R_w=k_{\sigma^m}[x,x^{-1}]$,
the Laurent polynomials in $x$ with $x\lambda=\sigma^m(\lambda)x$,
where $x$ acts by $b_ix=b_{i-n}$ and $k$ acts on $b_i$ through the
twist by the $\widetilde\tau$-letters between the base point and
the position $i$ \cite[Section~2.10]{BTCB}.  Here $m$ is the number
of $\widetilde\tau$-letters in a period, since the scalars are
twisted by $\sigma$ at each passage through $\widetilde\tau$ and
by nothing else.  The \emph{band modules} are the
$M(w)\otimes_{R_w}V$ for finite-dimensional $R_w$-modules $V$, that
is, vector spaces with a bijective $\sigma^m$-semilinear
endomorphism \cite[Section~2.10]{BTCB}.  The Main Theorem of
\cite{BTCB} states that the string modules $M(w)$ and the band
modules $M(w)\otimes_{R_w}V$ with $V$ indecomposable form a
complete list of the finite-dimensional indecomposable modules,
without repetition as $w$ runs over representatives of the strings
and bands.  So $G$ is a string or band module on a word $w$.

A word of $\Lambda_N$ never turns back in the levels, since a
letter returning to the previous level is the inverse of the
letter before it or spells a relation in $Z$ with it.  So $w$ is a
sequence of \emph{segments}, walks in the levels $[0,N]$ with each
edge carrying $u$ or $t$, joined by $\widetilde\tau$-letters.
Thus $G$ is the corresponding sequence of interval gauges joined by
matchings (Figure~\ref{fig:string-band}).

\begin{figure}[htbp]
\centering
\begin{tikzcd}[column sep=1.05em, row sep=1.6em, cells={nodes={inner sep=1.4pt}}, baseline=(current bounding box.north)]
 & \scriptstyle0 & \scriptstyle1 & \scriptstyle2 & \scriptstyle3
 & \scriptstyle4 \\[-1.4em]
 Q_0 & & & \bullet \arrow[r, "u"] & \bullet & \bullet \arrow[l, "t"'] \\
 Q_1 & \bullet & \bullet \arrow[l] \arrow[r] & \bullet \arrow[r]
 & \bullet & \bullet \arrow[l] \\
 Q_2 & \bullet & \bullet \arrow[l] \arrow[r] & \bullet
 & \bullet \arrow[l] &
 \arrow[from=2-6, to=3-2, "\widetilde\tau"']
 \arrow[from=3-6, to=4-2, "\widetilde\tau"']
\end{tikzcd}
\qquad\qquad
\begin{tikzcd}[column sep=1.05em, row sep=1.6em, cells={nodes={inner sep=1.4pt}}, baseline=(current bounding box.north)]
 & \scriptstyle0 & \scriptstyle1 & \scriptstyle2 & \scriptstyle3
 & \scriptstyle4 \\[-1.4em]
 P_1 & \bullet & \bullet \arrow[l] \arrow[r] & \bullet \arrow[r]
 & \bullet & \bullet \arrow[l] \\
 P_2 & \bullet \arrow[r] & \bullet \arrow[r] & \bullet
 & \bullet \arrow[l] & \bullet \arrow[l]
 \arrow[from=2-6, to=3-2]
 \arrow[from=3-6, to=2-2, densely dashed]
\end{tikzcd}
\caption{A string (left) and a band (right) of $\Lambda_N$ on the
levels $[0,4]$.  Dots are copies of $k$, left arrows are $t$, right
arrows are $u$, diagonal arrows are $\widetilde\tau$-letters, and
rows are segments.  The string has free ends at level $2$ of $Q_0$
and level $3$ of $Q_2$.  The dashed matching closes the period of
the band and carries its parameter $x$.}
\label{fig:string-band}
\end{figure}
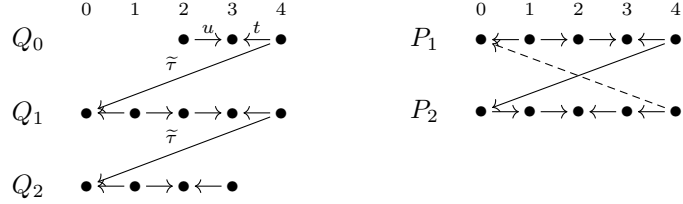
\FloatBarrier

The boundary condition of Lemma~\ref{lem:boundary-condition}
constrains the segments.  Let $x$ be a basis vector at a level
$j\in(0,N)$ with $t_{j-1}x=0$ and $u_jx=0$.  Then $x$ and $y=0$
satisfy the hypotheses of the boundary condition, so $x=0$.  Such
an $x$ sits at every interior sink of a segment, at every free end
whose edge points into it, and at every vertex without edges, so
none of these occurs in $(0,N)$.  Bijectivity of $\widetilde\tau$ matches every basis line at level
$N$ with one at level $0$.  So every segment has a single
source, and its shape is $(0\leftarrow\cdots\leftarrow c)$ with the
level-$0$ end matched, $(0\leftarrow\cdots\leftarrow
s\to\cdots\to N)$ with both ends matched, or $(b\to\cdots\to N)$
with the level-$N$ end matched.
We allow singleton end segments ($c=0$ or $b=N$).  In the
boundary-condition arguments below, segments meeting $0$ or $N$
are continued along the corresponding constant tails.

Write $v_P$ for the source of a segment $P$.  Across every matching
$v_P>v_{P'}$.  If not, put $i:=v_P$ and take $x\ne0$ at level $i+1$
of $P$: the edge $(i,i+1)$ leaves the source, so $t_ix=0$, and the
$u$-edges give $u^\infty x\ne0$.  The matching carries
$\tau(u^\infty x)\ne0$ into the level-$0$ line of $P'$, and the
$t$-edges of $P'$ lift it to $y$ at level $i-1$ with
$t^{-\infty}y=\tau(u^\infty x)$ and $u_{i-1}y=0$, since the edge
$(i-1,i)$ sits below the source of $P'$.  This violates the
boundary condition.  A band is a periodic word, its segments all of
the middle shape, and the claim around the loop gives
$v>v>\cdots>v$.  So $w$ is a string, its two extreme segments
carrying the free ends, of the first and third shapes, and $G$ is
the chain $Q_0,\dots,Q_{n+1}$ of the statement, with
$b>v_1>\cdots>v_n>c$.  The string module $M(w)$ has one basis vector $b_i$ at each
position of $w$, so the levels of each segment $Q_j$ are
one-dimensional, and every arrow, $\widetilde\tau$ included,
carries basis vectors to basis vectors: in this basis every
matching is the identity.

At the two ends the inequalities improve by one.  Suppose $n\ge1$
and $v_1=b-1$, and take $x\ne0$ at level $b$ of $Q_0$.  Then
$t_{b-1}x=0$, since $Q_0$ has no level $b-1$, and the $u$-edges
give $u^\infty x\ne0$.  The matching and the $t$-edges of $Q_1$
lift $\tau(u^\infty x)$ to $y$ at level $b-2$ with
$t^{-\infty}y=\tau(u^\infty x)$ and $u_{b-2}y=0$, since $v_1=b-1$
makes the edge $(b-2,b-1)$ of $Q_1$ a $t$-edge.  This violates the
boundary condition, so $b\ge v_1+2$.  Dually, if $v_n=c+1$, a
nonzero $x$ at level $c+2$ of $Q_n$ lifts to $y$ at level $c$ of
$Q_{n+1}$ with $u_cy=0$, since $Q_{n+1}$ has no level $c+1$, so
$v_n\ge c+2$.  For $n=0$ and $b\le c+2$ the same $x$ lifts to $y$
at level $b-2$ of $Q_1$, with $u_{b-2}y=0$ since the edge
$(b-2,b-1)$ is a $t$-edge or its level is absent, so $b\ge c+3$.
\end{proof}

The linear counterpart of the classification used above is the
theory of representations of \emph{bunches of chains}
\cite{Bondarenko92,BurbanDrozd}.

It remains to prove the $\operatorname{Ext}^1$ statement to
conclude the proof of Proposition~\ref{prop:chain-presentation}.
Let $\epsilon=\epsilon'\epsilon''$ be a word satisfying the
source-sink condition, the prefix $\epsilon'$
ending with the letter $1$ at the position $v_1$ and the suffix
$\epsilon''$ containing exactly one letter $1$, at its last
position $b-1$.  Fix an interval $[0,N]$ whose interior contains
the positions of $\epsilon$, and put $G'=S(M(\epsilon'))$ and
$G''=S(M(\epsilon''))$.

\begin{lemma}
\label{lem:extension-uniqueness}
In the category of coherent $F$-gauges annihilated by $p$, the
group $\operatorname{Ext}^1(G'',G')$ is isomorphic to $(k,+)$, and
the automorphisms of $G'$ act transitively on its nonzero elements.
\end{lemma}

\begin{proof}
Let $0\to G'\to E\to G''\to0$ be an extension of $F$-gauges
annihilated by $p$.  Since $pE=0$, the levels of $E$
are $k$-vector spaces and splittings $E^r=G'^{\,r}\oplus G''^{\,r}$
exist.  Choose one at the level $N$, transported along the
isomorphisms $u_r$ for $r\ge N$, one at the level $0$, transported
along the isomorphisms $t_r$ for $r\le-1$, and arbitrary splittings in
between.
In these splittings the structure maps of $E$ are the triangular
matrices
\[
 u_r=\begin{pmatrix}u'_r & c^u_r\\ 0 & u''_r\end{pmatrix},\qquad
 t_r=\begin{pmatrix}t'_r & c^t_r\\ 0 & t''_r\end{pmatrix},
\]
with diagonal blocks the structure maps of $G'$ and $G''$ and with
entries $c^u_r:G''^{\,r}\to G'^{\,r+1}$ and
$c^t_r:G''^{\,r+1}\to G'^{\,r}$.  Since $E$, $G'$, and $G''$ are killed
by $p$,
\[
 0\ =\ u_rt_r\ =\
 \begin{pmatrix}u'_rt'_r & \ u'_rc^t_r+c^u_rt''_r\\ 0 & u''_rt''_r\end{pmatrix}
\]
has vanishing diagonal blocks, and its corner is the relation
\begin{equation}
 u'_rc^t_r+c^u_rt''_r=0.
 \label{eq:splitting-relation}
\end{equation}
Moreover $u_r=0$ for
$r\le-1$ and $t_r=0$ for $r\ge N$, since $E$ has level $[0,N]$ and
$u_rt_r=t_ru_r=p$.  In the transported splittings every
entry therefore vanishes outside $0\le r\le N-1$.  A change of splittings is conjugation by the matrices
$g_r=\left(\begin{smallmatrix}1 & h_r\\ 0 & 1\end{smallmatrix}\right)$
with $h_r:G''^{\,r}\to G'^{\,r}$:
\[
 g_{r+1}^{-1}\,u_r\,g_r=
 \begin{pmatrix}u'_r & \ c^u_r+u'_rh_r-h_{r+1}u''_r\\ 0 & u''_r\end{pmatrix},
 \qquad
 g_r^{-1}\,t_r\,g_{r+1}=
 \begin{pmatrix}t'_r & \ c^t_r+t'_rh_{r+1}-h_rt''_r\\ 0 & t''_r\end{pmatrix}.
\]

We show that suitable $h_r$ make $c^u_r=c^t_r=0$ for every $r$, except
for a single component of $c^u_{v_1}$ at the edge $(v_1,v_1+1)$.

On the edges $r\ge v_1+1$ the maps $t'_r$ and $t''_r$ vanish and
$u'_r$ is invertible, so \eqref{eq:splitting-relation} gives $c^t_r=0$.
Put $h_r:=0$ for $r\ge N$.  The recursion
$h_r:=-(u'_r)^{-1}(c^u_r-h_{r+1}u''_r)$ then descends to $r=v_1+1$ and
clears every $c^u_r$, while no $c^t_r$ changes.

At the edge
$(v_1,v_1+1)$, $G''^{\,v_1+1}=0$ gives $c^t_{v_1}=0$ and $u''_{v_1}=0$,
so $c^u_{v_1}$ changes only by $u'_{v_1}h_{v_1}$: a choice of $h_{v_1}$
clears the components of $c^u_{v_1}$ in the image of $u'_{v_1}$, which
is spanned by the rows of $G'$ present at level $v_1$.  The remaining
component maps the line $G''^{\,v_1}$ into the line of the absent row
in $G'^{\,v_1+1}$.  Write $\zeta\in k$ for this scalar: it is unchanged
under every $(h_r)$.

On the edges
$r\le v_1-1$ the maps $t''_r$ are invertible, so
\eqref{eq:splitting-relation} gives
$c^u_r=-u'_rc^t_r(t''_r)^{-1}$.  The recursion
$h_r:=(c^t_r+t'_rh_{r+1})(t''_r)^{-1}$ descends from the chosen
$h_{v_1}$ and clears every $c^t_r$, hence every $c^u_r$.

It remains to normalize the gluing.  In the splittings of the limits
$E^{\pm\infty}$ induced at the two tails, the gluing of $E$ and its
changes are the triangular matrices
\[
 \tau_E=\begin{pmatrix}\tau' & c^\tau\\ 0 & \tau''\end{pmatrix},
 \qquad
 g_{-\infty}^{-1}\,\tau_E\,\sigma^*(g_{+\infty})=
 \begin{pmatrix}\tau' & \ c^\tau+\tau'\sigma^*(h_{+\infty})-h_{-\infty}\tau''\\
 0 & \tau''\end{pmatrix},
\]
with diagonal blocks the gluings of $G'$ and $G''$ and with
$h_{\pm\infty}$ the limit values of the family $(h_r)$.  The family
$h_r:=(u'_r)^{-1}h_{r+1}u''_r$ for $r\ge b$, determined by an arbitrary
$h_{+\infty}$, vanishes for $r\le b-1$ and changes no $c^u_r$ or
$c^t_r$, and the choice
$\sigma^*(h_{+\infty}):=-(\tau')^{-1}c^\tau$ makes the gluing
block-diagonal.  The normalizations above are equivalences of
extensions, so the class of the extension is determined by $\zeta$,
with $\zeta=0$ the split extension, and every $\zeta$ occurs,
through the triangular matrices with $c^u_{v_1}$ the single nonzero
entry.  These satisfy $t_ru_r=0$ as well: the image of $c^u_{v_1}$
lies in the row absent at the level $v_1$, so $t'_{v_1}$ kills it.
Under Baer sum the entries $c^u_r$ and $c^t_r$ add, and $\zeta$
with them, so $E\mapsto\zeta$ is an isomorphism
$\operatorname{Ext}^1(G'',G')\simeq(k,+)$.

For $a\in k^\times$, let an automorphism of $G'$ act by $a$ on the
summand $(v_1+1\rightarrow\cdots\rightarrow N)$ and by the successive
powers $\sigma(a),\sigma^2(a),\dots$ on the following summands of its
open path.  Scalars commute with the structure maps inside each
summand, and across each identity matching
(Lemma~\ref{lem:regluing-normal-form}) the two scalars differ by
$\sigma$, so this is an automorphism of $G'$.  Its pushout multiplies
$\zeta$ by $a$, since $c^u_{v_1}$ has image in the
level-$(v_1+1)$ line of the first summand.  So the automorphisms of
$G'$ act transitively on the nonzero classes.
\end{proof}

\begin{corollary}
\label{cor:postnikov-layers}
Let $M(\epsilon)$ be a modified interval object on $[1,m]$
satisfying the source-sink condition, with rightward word
$\epsilon'$ as in Remark~\ref{rem:two-normal-forms}.  The Postnikov
layers of Lemma~\ref{lem:postnikov} are
\[
 \Dom^j(M(\epsilon))=U_{\epsilon'_j},\qquad 1\le j\le m.
\]
Moreover $\Dom^n(N\{a\})=\Dom^{n+a}(N)$ for every $N$ and
$a\in\mathbb Z$.
\end{corollary}

\begin{proof}
Write $F_n\subseteq M(\epsilon)$ for the subobject spanned by the
suffix $[n,m]$ in the rightward presentation, as in
Lemma~\ref{lem:prefix-cut}.  The graded piece $F_n/F_{n+1}$ is the
one-vertex object
$M([n,n],\varnothing;\epsilon'_n\delta_n)=U_{\epsilon'_n}\{1-n\}$,
with $U_1=U_0(\delta_1)$, concentrated in cohomological degree
$1-n$.  So $F_n$ lies in $D^{\le1-n}$ and $M(\epsilon)/F_n$ lies in
$D^{\ge2-n}$, which identifies $F_n=\tau_{\le1-n}M(\epsilon)$ and
$H^{1-n}(M(\epsilon))=U_{\epsilon'_n}(1-n)$.  The twist in
Lemma~\ref{lem:postnikov} gives
$\Dom^n(M(\epsilon))=U_{\epsilon'_n}$.  The last claim follows from
$H^q(N\{a\})=H^{q-a}(N)(a)$.
\end{proof}

\subsection{Classification of stable objects}
\label{subsec:stable-classification}

It remains to identify the stable objects among the candidates
$M(\epsilon)$ of Proposition~\ref{prop:chain-presentation}.  Stability
bounds the weight of every saturated subobject, in particular of the
prefix and suffix subobjects of $M(\epsilon)$, and the words meeting these
numerical bounds are exactly the Christoffel words, defined next.

\begin{definition}[Christoffel words and Christoffel dominoes]
\label{def:christoffel-domino}
For $0\le d\le q$ with $\gcd(d,q)=1$, the \emph{Christoffel word} of slope $d/q$ is
\[
  \epsilon_j(d/q):=\Bigl\lfloor\frac{jd}{q}\Bigr\rfloor
  -\Bigl\lfloor\frac{(j-1)d}{q}\Bigr\rfloor,
  \qquad 1\le j\le q.
\]
For $0<d<q$ it begins with $0$ and ends with $1$, so $M(\epsilon(d/q))$
satisfies the source-sink condition.  The diagonal domino
\[
  U_{d/q}:=M\bigl(\epsilon(d/q)\bigr)
\]
is called the \emph{Christoffel domino} of slope $d/q$.
\end{definition}

In the geometric picture of \cite[\S 1.1]{BLRS09}, the letters $0$ and
$1$ mark unit steps to the right and upward, and $\epsilon(d/q)$ traces
the lower lattice approximation to the segment from $(0,0)$ to
$(q-d,d)$: its first $n$ letters contain $\lfloor nd/q\rfloor$ letters
$1$.

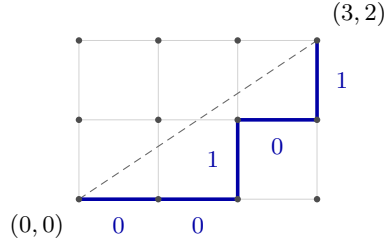
\begin{figure}[htbp]
\centering
\begin{tikzpicture}[scale=1.05, every node/.style={font=\small}]
  \draw[step=1, gray!35, very thin] (0,0) grid (3,2);
  \draw[densely dashed, black!60] (0,0) -- (3,2);

  \draw[blue!65!black, line width=1.35pt]
    (0,0) -- node[midway, below=3pt] {$0$} (1,0)
          -- node[midway, below=3pt] {$0$} (2,0)
          -- node[midway, left=3pt]  {$1$} (2,1)
          -- node[midway, below=3pt] {$0$} (3,1)
          -- node[midway, right=3pt] {$1$} (3,2);

  \foreach \x in {0,...,3}
    \foreach \y in {0,...,2}
      \fill[black!70] (\x,\y) circle (1.15pt);

  \node[below left=2pt] at (0,0) {$(0,0)$};
  \node[above right=2pt] at (3,2) {$(3,2)$};
\end{tikzpicture}
\caption{The Christoffel word $\epsilon(2/5)=00101$ as a lattice
path.}
\label{fig:christoffel-path}
\end{figure}
\FloatBarrier

Two reduced fractions
$\tfrac ab<\tfrac ce$ in $[0,1]$ are \emph{Farey neighbours} if
$bc-ae=1$.  Their \emph{mediant} $\tfrac{a+c}{b+e}$ again satisfies $\gcd(a+c,b+e)=1$, and every
$\tfrac dq$ with $0<d<q$ is the mediant of a unique Farey pair, its
\emph{parents}.  These are determined by $b\equiv d^{-1}\pmod q$ with $0<b<q$
and $a=(db-1)/q$, and then $db-aq=cq-ed=bc-ae=1$: the mediant is a
Farey neighbour of each of its parents.

\begin{lemma}[{Farey concatenation, cf.\ \cite[Chapter~1]{BLRS09}}]
\label{lem:farey-concat}
Let $\tfrac ab<\tfrac ce$ be Farey neighbours with mediant $\tfrac dq$, $d=a+c$,
$q=b+e$. Then the Christoffel word of slope $\tfrac dq$ is the concatenation of the
Christoffel words of its parents:
\[
  \epsilon_j(d/q)=
  \begin{cases}
    \epsilon_j(a/b), & 1\le j\le b,\\[1mm]
    \epsilon_{j-b}(c/e), & b<j\le q.
  \end{cases}
\]
\end{lemma}

\begin{proof}
The $n$-th partial sum of the Christoffel word of slope $d/q$ is
$\lfloor nd/q\rfloor$, so it suffices to compare the partial sums on the
two blocks.  For $0\le n\le b$, since $db-aq=1$,
\[
  \frac{nd}{q}-\frac{na}{b}=\frac{n(db-aq)}{qb}=\frac{n}{qb}\in\Bigl[0,\tfrac1q\Bigr].
\]
The fractional part of $na/b$ is a multiple of $1/b$, at most $1-1/b$,
and the difference $n/qb$ is at most $1/q<1/b$, so their sum stays below
$1$ and $\lfloor nd/q\rfloor=\lfloor na/b\rfloor$.  Similarly, for $0\le m\le e$,
since $cq-ed=1$,
\[
  \frac{(b+m)d}{q}-a-\frac{mc}{e}
  =\frac{e(db-aq)-m(cq-ed)}{qe}
  =\frac{e-m}{qe}\in\Bigl[0,\tfrac1q\Bigr]
\]
gives $\lfloor(b+m)d/q\rfloor=a+\lfloor mc/e\rfloor$.  Taking first
differences proves the claim.
\end{proof}

Next we characterize the stable diagonal dominoes of slope
$\lambda=d/q$.

\begin{theorem}
\label{thm:christoffel-stability}
Let $M(\epsilon)$ be a modified interval object satisfying the
source-sink condition, of slope $d/q$ with $\gcd(d,q)=1$.  Then
$M(\epsilon)$ is stable if and only if $\epsilon$ is the Christoffel
word $\epsilon(d/q)$.  Up to Breuil--Kisin twist, the Christoffel
domino $U_{d/q}$ is thus the unique stable diagonal domino of slope
$d/q$.
\end{theorem}

\begin{proof}
Write $m$ for the number of vertices and
$s_n:=\sum_{j=1}^{n}\epsilon_j$ for the partial sums, $0\le n\le m$,
so that $s_m$ is the degree of $M(\epsilon)$.
Stability will be tested on two families of saturated subobjects,
indexed by $0<n<m$.  The prefix $[1,n]$ spans a saturated subobject
$P_n$ of rank $n$ and degree $s_n$ (Lemma~\ref{lem:prefix-cut}).
For the suffix $[n+1,m]$ it is slightly trickier, but we can still
define a saturated subobject $P'_n$ of rank $m-n$ and degree
$s_m-s_n-1$ by a sequence of moves of Lemma~\ref{lem:normalization},
as follows.  If $n$ is starred, let $b-1$ be the nearest unstarred vertex below
it, which exists because $\epsilon_1=0$, so that the vertices
$b,\dots,n$ are starred.  The moves shift their modifications one step left, at $b$
first and at $n$ last, and the vertex $n$ ends unstarred.  If $n+1$
is unstarred, let $c+1$ be the nearest starred vertex above it,
which exists because $\epsilon_m=1$, so that the vertices
$n+1,\dots,c$ are unstarred.  The moves carry its modification to
$c$, then to $c-1$, and so on down to $n+1$.  The move
$n\leftarrow(n+1)^*\mapsto n^*\to(n+1)$ then reverses the one arrow
leaving the suffix.  In this chart the suffix spans a saturated subobject, and only the
last move carries a modification across the edge, so the degree is
$s_m-s_n-1$ (Lemma~\ref{lem:dictionary}).

By Definition~\ref{def:weight},
\[
 h_n:=\theta_{d/q}(P_n)=q\,s_n-d\,n,
 \qquad
 \theta_{d/q}(P'_n)=q(s_m-s_n-1)-d(m-n)=-h_n-q,
\]
using $qs_m=dm$, and $h_0=h_m=0$.

Suppose $M(\epsilon)$ is stable.  Then $\theta_{d/q}(P_n)\le-1$ and
$\theta_{d/q}(P'_n)\le-1$, so $-q<h_n<0$ for every $0<n<m$.

Suppose $-q<h_n<0$ for every $0<n<m$.  Division by $q$ turns the
inequalities into
\[
  \frac{dn}{q}-1<s_n<\frac{dn}{q}.
\]
An open interval of length one contains an integer only when its
endpoints are not integers, and that integer is then
$\lfloor dn/q\rfloor$.  So $s_n=\lfloor dn/q\rfloor$, and $q\nmid dn$ for $0<n<m$, hence
$q\nmid n$ by $\gcd(d,q)=1$.  On the other hand $qs_m=dm$ gives
$q\mid dm$, hence $q\mid m$.  So $m=q$ and $\epsilon$ is the
Christoffel word of slope $d/q$.

Conversely, write $U=U_{d/q}=M(\epsilon(d/q))$.  We show that every
proper nonzero saturated subobject $N\subset U$ satisfies
$\theta_{d/q}(N)\le-1$, and
induct on $q$ along the Farey parents $\tfrac ab<\tfrac dq<\tfrac ce$, so that
$d=a+c$, $q=b+e$ and $db-aq=cq-ed=1$.  Both parents have strictly smaller
denominators, and the statement is invariant under Breuil--Kisin twist, so the
induction hypothesis applies to them.  The base is the boundary
slopes $\tfrac01$ and $\tfrac11$.  The one-vertex chains
$U_{0/1}=U_0$ and $U_{1/1}=U_1$ fail the source-sink condition, so
the base is not an instance of the theorem, and we check it
directly: a one-vertex chain has a single one-dimensional level,
hence no nonzero proper saturated subobject, and the statement holds
vacuously.

By Lemma~\ref{lem:farey-concat} the word $\epsilon(d/q)$ restricts on $[1,b]$ to the
Christoffel word of $\tfrac ab$ and on $[b+1,q]$ to that of $\tfrac ce$, so the prefix
$[1,b]$ gives a saturated exact sequence (Lemma~\ref{lem:prefix-cut})
\[
  0\longrightarrow A\longrightarrow U\longrightarrow C\longrightarrow0
\]
in which $A=U_{a/b}$ and $C$ is the Breuil--Kisin twist of $U_{c/e}$
placed on the levels $[b+1,q]$, with $\theta_{d/q}(A)=qa-db=-1$ and
$\theta_{d/q}(C)=qc-de=1$.

Let $N\subset U$ be a proper nonzero saturated subobject. Put $X:=N\cap A$, let $Y'$
be the image of $N$ in $C$, and let $Y:=(Y')^{\mathrm{sat}}\subseteq C$. Then $X$ is
saturated in $A$ and in $N$, since $A/X$ and $N/X\cong Y'$ embed into the diagonal
dominoes $U/N$ and $C$, and $\DomD$ is closed under subobjects. Additivity and
Proposition~\ref{prop:torsion-structure}(ii) give
\[
  \theta_{d/q}(N)=\theta_{d/q}(X)+\theta_{d/q}(Y')
  =\theta_{d/q}(X)+\theta_{d/q}(Y)-q\,t,
  \qquad t:=\degD(Y/Y')\ge0.
\]
The relations $db-aq=cq-ed=1$ transfer the parent weights: for
subobjects $X\subseteq A$ and $Y\subseteq C$,
\[
  \theta_{d/q}(X)=q\,\degD(X)-d\,\rkD(X)
  =\tfrac qb\,\theta_{a/b}(X)-\tfrac{db-aq}{b}\,\rkD(X)
  =\tfrac qb\,\theta_{a/b}(X)-\tfrac{\rkD(X)}{b}
\]
and
\[
  \theta_{d/q}(Y)=q\,\degD(Y)-d\,\rkD(Y)
  =\tfrac qe\,\theta_{c/e}(Y)+\tfrac{cq-ed}{e}\,\rkD(Y)
  =\tfrac qe\,\theta_{c/e}(Y)+\tfrac{\rkD(Y)}{e}.
\]
The induction hypothesis, $\theta_{a/b}(X)\le-1$ for
$0\ne X\subsetneq A$ and likewise on $C$, turns these into
\[
  \theta_{d/q}(X)\le
  \begin{cases}
    0, & X=0,\\
    -1, & X=A,\\
    -\tfrac{q+1}{b}, & 0\ne X\subsetneq A,
  \end{cases}
  \qquad
  \theta_{d/q}(Y)\le
  \begin{cases}
    0, & Y=0,\\
    1, & Y=C,\\
    -\tfrac{b+1}{e}, & 0\ne Y\subsetneq C.
  \end{cases}
\]
Only $Y=C$ contributes a positive bound.  For $Y\ne C$ the bounds
are nonpositive, and one of $X$, $Y$ is nonzero, so
$\theta_{d/q}(N)<0$.  So assume $Y=C$.  For $t\ge1$ the correction dominates,
$\theta_{d/q}(N)\le1-q\le-1$.  For $t=0$ we get $Y'=Y=C$, so $N$
maps onto $C$, and we claim $0\ne X\subsetneq A$.  If $X=0$, then
$N\to C$ is an isomorphism and splits the sequence, contradicting
the indecomposability of $U$ (Lemma~\ref{lem:chain-endo}).  If
$X=A$, then $N$ has full rank $q=b+e$, hence $N=U$
(Proposition~\ref{prop:torsion-structure}(i)), contradicting
properness.  So $0\ne X\subsetneq A$ and
$\theta_{d/q}(N)\le1-\tfrac{q+1}{b}<0$, since $b<q$.  The integer
$\theta_{d/q}(N)$ is thus at most $-1$.

Finally, a stable diagonal domino of slope $d/q$ is annihilated by
$p$ and indecomposable (Lemma~\ref{lem:annihilated-by-p}), hence a
Breuil--Kisin twist of a modified interval object $M(\epsilon)$
satisfying the source-sink condition
(Proposition~\ref{prop:chain-presentation}), and the equivalence
just proved forces $\epsilon=\epsilon(d/q)$.
\end{proof}

In the lattice path of Figure~\ref{fig:christoffel-path}, the vertex
reached after $n$ letters is $(n-s_n,\,s_n)$, with $s_n$ the number
of ones among the first $n$ letters.  By
Theorem~\ref{thm:christoffel-stability}, a word is stable exactly
when its path is the \emph{highest} lattice path strictly below the
segment from $(0,0)$ to $(q-d,d)$.  A path that crosses the segment
meets a destabilizing prefix, and a path that drops a step below it
meets a destabilizing suffix.  Figure~\ref{fig:staircase-failures}
shows both at slope $2/5$.

\begin{figure}[htbp]
\centering
\begin{tikzpicture}[scale=1.0, every node/.style={font=\small}]
  \draw[step=1, gray!35, very thin] (0,0) grid (3,2);
  \draw[densely dashed, black!60] (0,0) -- (3,2);
  \draw[black!25, line width=1.1pt] (0,0) -- (2,0) -- (2,1) -- (3,1) -- (3,2);
  \draw[blue!65!black, line width=1.35pt] (0,0) -- (1,0) -- (1,1) -- (2,1) -- (3,1) -- (3,2);
  \foreach \x in {0,...,3} \foreach \y in {0,...,2}
    \fill[black!70] (\x,\y) circle (1.15pt);
  \fill[blue!65!black] (1,1) circle (2.6pt);
  \node at (1.5,-0.75) {(a)\quad $\epsilon=01001$};
  \begin{scope}[xshift=5.4cm]
    \draw[step=1, gray!35, very thin] (0,0) grid (3,2);
    \draw[densely dashed, black!60] (0,0) -- (3,2);
    \draw[black!25, line width=1.1pt] (0,0) -- (2,0) -- (2,1) -- (3,1) -- (3,2);
    \draw[blue!65!black, line width=1.35pt] (0,0) -- (3,0) -- (3,1) -- (3,2);
    \foreach \x in {0,...,3} \foreach \y in {0,...,2}
      \fill[black!70] (\x,\y) circle (1.15pt);
    \fill[blue!65!black] (3,0) circle (2.6pt);
    \node at (1.5,-0.75) {(b)\quad $\epsilon=00011$};
  \end{scope}
\end{tikzpicture}
\caption{The two other words of length $5$ and degree $2$ satisfying
the source-sink condition, drawn as in
Figure~\ref{fig:christoffel-path}, with $\epsilon(2/5)$ in grey.  In
(a) the path crosses the segment at the marked vertex, and the
prefix $[1,2]=01$ defines a subobject isomorphic to $U_{1/2}$.  In (b) it drops a step below at the marked vertex, the move of
Lemma~\ref{lem:normalization} rewrites
$00011=(1\leftarrow2\leftarrow3\leftarrow4^{*}\leftarrow5^{*})$ as
$(1\leftarrow2\leftarrow3^{*}\rightarrow4\leftarrow5^{*})$, and the
suffix $4\leftarrow5^{*}$ defines a subobject isomorphic to a
Breuil--Kisin twist of $U_{1/2}$.  Both subobjects have rank $2$,
degree $1$, and weight $1$, so both destabilize.}
\label{fig:staircase-failures}
\end{figure}
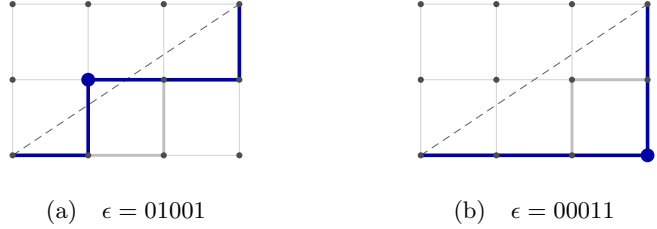
\FloatBarrier

\begin{theorem}[Classification of stable objects]
\label{thm:stable-classification}
Let $\lambda\in\mathbb Q$ and write it uniquely as
$\lambda=r+d/q$, where $r\in\mathbb Z$, $0<d\le q$, and $\gcd(d,q)=1$.
The simple objects of $\Delta_{\mathrm{tor}}^{ss}(\lambda)$ are precisely the
Breuil--Kisin twists of $U_{d/q}(r\mathbf1)$.
The simple objects of $\Delta_{\mathrm{tor}}^{ss}(+\infty)$ are precisely the
Breuil--Kisin twists of the simple finite-length Dieudonn\'e
modules.  Thus every
finite rational slope has a unique stable object up to Breuil--Kisin twist.
\end{theorem}

\begin{proof}
Theorem~\ref{thm:fixed-slope} identifies the simple objects with the stable
ones at every slope.  At slope $+\infty$ the assertion is
Lemma~\ref{lem:simple-finite-dieudonne}.  If $d/q=1$, then
$U_{1/1}=U_1$ and $U_{1/1}(r\mathbf1)=U_{r+1}$, so the assertion is
Proposition~\ref{thm:integer-simples}.  Suppose that $0<d<q$.  The exact
autoequivalence $M\mapsto M(-r\mathbf1)$ shifts every slope by $-r$ and
preserves stability (Lemma~\ref{lem:nygaard-shift}), so we may take
$\lambda=d/q$, where the assertion is
Theorem~\ref{thm:christoffel-stability}.
\end{proof}

Ekedahl calls an object $U$ of the exact category of diagonal
dominoes of pure type $\tfrac12$ \emph{weakly simple} if there is
no short exact sequence $0\to U'\to U\to U''\to0$ where both $U'$
and $U''$ are nonzero and of pure type $\tfrac12$.  He asserts
that the weakly simple objects are the two families $N_{m,n}$ and
$N_m$ \cite[Theorem~III.2.7]{ekedahl3}.  In our notation, these are
the Breuil--Kisin twists of the modified interval objects
$M(\epsilon)$ with $\epsilon=0\cdots01$ and $\epsilon=011$, of
slopes $1/(n-m+1)$ and $2/3$.  The list is incomplete, and the
classification of stable objects corrects it.

\begin{theorem}[Ekedahl's Theorem III.2.7, corrected]
\label{thm:weak-simplicity}
\label{rem:weak-simplicity}
Up to Breuil--Kisin twist, the weakly simple diagonal dominoes of
pure type $\tfrac12$ are exactly the Christoffel dominoes $U_{1/q}$
and $U_{(q-1)/q}$ for $q\ge2$.
\end{theorem}

\begin{proof}
We claim that weak simplicity implies stability.  If $U$ is of pure type
$\tfrac12$, that is of HN slopes in $(0,1)$
(Theorem~\ref{thm:refinement}), then it must have exactly one HN
slope, because otherwise the HN filtration gives rise to a short
exact sequence violating weak simplicity.  In other words, $U$ must
be semistable.  Moreover, it is stable: a proper nonzero subobject
in $\Domss(\lambda)$ gives rise to the same violation.  So $U$ is
a Breuil--Kisin twist of $U_{d/q}$
(Theorem~\ref{thm:stable-classification}).  For $2\le d\le q-2$
both Farey parents $\tfrac ab<\tfrac dq<\tfrac ce$ lie in $(0,1)$,
and the saturated exact sequence
$0\to U_{a/b}\to U_{d/q}\to U_{c/e}\{-b\}\to0$
from the proof of Theorem~\ref{thm:christoffel-stability} has both
pieces of pure type $\tfrac12$ and violates weak simplicity.  For $d=1$ and $d=q-1$ one
parent is the boundary slope $\tfrac01$ or $\tfrac11$, and the
corresponding piece leaves pure type $\tfrac12$.  Conversely, a
violating short exact sequence for $U_{1/q}$ or $U_{(q-1)/q}$
splits the degree as $d=d_1+d_2$ with $1\le d_i\le r_i-1$
(Theorem~\ref{thm:refinement}(ii)), impossible for $d=1$ and, with
$r_1+r_2=q$, for $d=q-1$.  So the weakly simple objects are exactly
the Breuil--Kisin twists of $U_{1/q}$ and $U_{(q-1)/q}$.
\end{proof}

\begin{remark}[The gap in Ekedahl's proof]
\label{rem:weak-simplicity-gap}
Towards the end of the proof \cite[p.~67]{ekedahl3}, Ekedahl asserts
that the modified interval object $M(0111)$, in his notation
$m\leftarrow(m+1)^{*}\leftarrow(m+2)^{*}\leftarrow(m+3)^{*}$, is an
extension of $M(01)$ by the Breuil--Kisin twist $M(01)\{-2\}$, hence
not weakly simple.  Such an extension is impossible: degrees are
additive, and $\degD M(0111)=3$ while the two pieces have degree $1$
each (Lemma~\ref{lem:dictionary}).  This is where the slopes
$(q-1)/q$ with $q\ge4$ disappear from his list.
\end{remark}

\subsection{Extensions between stable objects}
\label{subsec:ext-stable}

Theorem~\ref{thm:fixed-slope} presents each
$\Delta^{ss}_{\mathrm{tor}}(\lambda)$ as a finite-length abelian
category, and Theorem~\ref{thm:stable-classification} lists its
simple objects.  In order to describe the category itself, we
compute the extensions between the simple objects.  The
modification $(r\mathbf1)$ is an exact autoequivalence of
$D^b_c(R)$ shifting every slope by $r$ (Lemma~\ref{lem:nygaard-shift},
cf.\ \cite[Theorem~III.2.3(ii)]{ekedahl3}), so it suffices to treat
$\lambda\in[0,1)$.  The slope $0$ is described by
Theorem~\ref{thm:integer-slope-category}.  For $\lambda\in(0,1)$ the
objects of $\Delta^{ss}_{\mathrm{tor}}(\lambda)$ lie in
$\Delta_{\mathrm{tor}}$ and, through $S$ with no shift, in
$\Coh_{\mathrm{tor}}(k^{\mathrm{Syn}})$
(Theorem~\ref{thm:phase-rotation}), and we compute on the
$F$-gauge side.

For $A,B\in\Delta_{\mathrm{tor}}$ write
$\operatorname{Ext}^i(A,B):=\operatorname{Hom}_{D^b_c(R)}(A,B[i])$,
so that $\operatorname{Ext}^1(A,B)$ classifies the extensions of $A$
by $B$ in $\Delta$.  These are $\mathbb Z_p$-modules, since $W$ is
not central in $R$.  When $\operatorname{End}(B)=k$ they are
$k$-vector spaces through the action of $\operatorname{End}(B)$, and
we put $\chi(A,B):=\sum_i(-1)^i\dim_k\operatorname{Ext}^i(A,B)$.  For the
objects below $\operatorname{End}(A)=k$ acts as well, through a
structure twisted by a power of $\sigma$
(Lemma~\ref{lem:extension-uniqueness}), with the same dimension.

\begin{theorem}
\label{thm:integer-slope-category}
Let $r\in\mathbb Z$.  The functor $M\mapsto S(M(-r\mathbf1)[1])$ is
an exact equivalence from $\Delta^{ss}_{\mathrm{tor}}(r)$ onto the
category of finite-length $F$-gauges.  These are the coherent
sheaves on $k^{\mathrm{Syn}}$ supported on the Hodge point
$k^{\mathrm{Hodge}}$ of Definition~\ref{def:stacks} and form the
kernel of
the restriction $\Coh(k^{\mathrm{Syn}})\to\Coh(k^{\Delta})$ to the
open point.  The equivalence carries $U_r\{m\}$ to the $F$-gauge
$k_{1-m}$ given by $k$ placed in level $1-m$.
\end{theorem}

\begin{proof}
By the reduction above we may take $r=0$, where
$\Delta^{ss}_{\mathrm{tor}}(0)$ is the category of
$\mathbf s$-acyclic diagonal dominoes, Ekedahl's diagonal dominoes
of type $0$ (Theorem~\ref{thm:refinement}(i),(iii) and
\cite[Definition~III.2.4]{ekedahl3}).  For such $U$ the
$F$-gauge $S(U[1])$ lies in $\Coh_{\mathrm{tor}}(k^{\mathrm{Syn}})$
(Theorem~\ref{thm:phase-rotation}) and satisfies
$S(U[1])^{\pm\infty}=\mathbf s(U)[1]=0$.  The restriction of a
coherent $F$-gauge $G$ to the open point is
$G^{-\infty}\cong\sigma^*G^{\infty}$, so the $F$-gauges supported on
the Hodge point are those with $G^{\pm\infty}=0$, and these have
finite length by coherence.  Conversely let $G$ be such an $F$-gauge, and
write $G=S(E)$ with
$E\in\mathcal G$.  Then $\mathbf s(E)=0$, so the amplitude of
$\mathbf s$ forces
$\mathbf s(\widetilde H^0E)=\mathbf s(\widetilde H^{-1}E)=0$.  The
$\mathbf s$-acyclic object $\widetilde H^0E$ lies in
$\mathcal T\cap\mathcal F=0$, so $E=U[1]$ with
$U:=\widetilde H^{-1}E$ $\mathbf s$-acyclic, hence torsion, and
$U_{\mathrm{ft}}=0$, since a nonzero finite-torsion subobject has
positive degree while $H^0(\mathbf s(N))\subseteq H^0(\mathbf s(U))=0$
for every $N\subseteq U$.  So $U\in\Delta^{ss}_{\mathrm{tor}}(0)$.
Both functors are exact, and $S(U_0\{m\}[1])$ is $k$ in level $1-m$
by Example~\ref{ex:equivalence-examples}(iii).
\end{proof}

The extensions between the simple objects $U_r\{m\}$ follow from
the computation of $\operatorname{Tor}^R_\bullet(R_1,U_j)$ in
\cite[Corollary~I.3.7]{IR83} and Ekedahl's comparison
\cite[Corollary~III.1.5.4(iii)]{ekedahl2}
\[
 \operatorname{RHom}_R(U_0,N)\simeq(R_1\otimes^{\mathbf L}_RN)[-2](1).
\]
For the graded group
$\underline{\operatorname{Ext}}^i_R(U_0,U_0)
=\bigoplus_a\operatorname{Ext}^i_R(U_0,U_0(a))$ these give
$\underline{\operatorname{Ext}}^0_R(U_0,U_0)=k[d]$,
$\underline{\operatorname{Ext}}^1_R(U_0,U_0)=0$ and
$\underline{\operatorname{Ext}}^2_R(U_0,U_0)=k[d](1)$, with $d^2=0$ and $d$ in
degree $1$.  Since $U_0\{m\}=U_0(m)[-m]$, the group
$\operatorname{Ext}^i(U_0\{m\},U_0\{n\})$ is the part of degree
$n-m$ of $\underline{\operatorname{Ext}}^{\,i+m-n}_R(U_0,U_0)$, and
the modification $(r\mathbf1)$ carries $U_0\{m\}$ to $U_r\{m\}$.
Reading off the degrees,
\begin{equation}
\label{eq:ext-table}
\begin{array}{c|ccc}
 \operatorname{Ext}^i(U_r\{m\},U_r\{n\}) & n=m & |n-m|=1 & |n-m|\ge2\\
 \hline
 i=0 & k & 0 & 0\\
 i=1 & 0 & k & 0\\
 i=2 & k & 0 & 0\\
 i\ge3 & 0 & 0 & 0
\end{array}
\end{equation}

Let $M$ and $N$ be coherent $F$-gauges.  There is a fibre sequence
in $D(\mathbb Z_p)$,
\begin{equation}
 \operatorname{RHom}_{D_{\mathrm{qc}}(k^{\mathrm{Syn}})}(M,N)\longrightarrow
 \operatorname{RHom}_{D_{\mathrm{qc}}(k^{\mathcal N})}(j_{\mathcal N}^*M,j_{\mathcal N}^*N)\longrightarrow
 \operatorname{RHom}_{D_{\mathrm{qc}}(k^{\Delta})}(j_\Delta^*M,j_\Delta^*N),
 \label{eq:gluing-fibre}
\end{equation}
whose second map is
$f\mapsto f^{-\infty}-\tau_N\,\sigma^*(f^\infty)\,\tau_M^{-1}$, by
the equalizer description of Definition~\ref{def:stacks}.  We write
$M|_{k^{\mathcal N}}:=j_{\mathcal N}^*M$ for the underlying gauge,
$\operatorname{RHom}_A$ for the middle term, the $\operatorname{RHom}$
over the graded Rees algebra $A=W[u,t]/(ut-p)$ of
Definition~\ref{def:stacks}, and $\operatorname{RHom}_W$ for the last.

For $a\in\mathbb Z$ let $K_u(a)$, $K_t(a)$ and $L(a)$ be the graded
modules $A/(t)$, $A/(u)$ and $A/p$ generated in degree $a$.  They
are the shifted structure sheaves of the divisors $t=0$, $u=0$ and
$p=ut=0$ of $k^{\mathcal N}$, and their levels are $k$ in the
degrees $\ge a$, in the degrees $\le a$ and in all degrees
respectively.  Resolving by the nonzerodivisors $t$,
$u$ and $p$, for a graded $A$-module $N$ with finite-length levels
\begin{equation}
\begin{gathered}
 \operatorname{RHom}_A(K_u(a),N)=\fib(t\colon N^a\to N^{a-1}),\qquad
 \operatorname{RHom}_A(K_t(a),N)=\fib(u\colon N^a\to N^{a+1}),\\
 \operatorname{RHom}_A(L(a),N)=\fib(p\colon N^a\to N^a),
\end{gathered}
\label{eq:interval-ext}
\end{equation}
so that $\chi_A(K_u(a),N)=\ell(N^a)-\ell(N^{a-1})$,
$\chi_A(K_t(a),N)=\ell(N^a)-\ell(N^{a+1})$ and $\chi_A(L(a),N)=0$.

\begin{proposition}
\label{prop:ext-stable}
Let $\lambda=d/q\in(0,1)$ with $\gcd(d,q)=1$, and write
$U:=U_{d/q}$.  The following hold in
$D^b_c(R)\simeq\Perf(k^{\mathrm{Syn}})$.
\begin{enumerate}
\item $\operatorname{Hom}(U,U)=k$, $\operatorname{Ext}^1(U,U)=0$,
  $\operatorname{Ext}^2(U,U)=k$, and $\operatorname{Ext}^i(U\{n\},U\{m\})=0$
  for $i\ge3$ and all $n,m$.
\item For $n\ne m$,
  $\dim_k\operatorname{Ext}^1(U\{n\},U\{m\})-\dim_k\operatorname{Ext}^2(U\{n\},U\{m\})
  =\delta_{|n-m|,q}$.
\end{enumerate}
\end{proposition}

\begin{proof}
The twists $U\{n\}$ are pairwise nonisomorphic simple objects of
$\Domss(d/q)$ (Theorem~\ref{thm:fixed-slope}), so
$\operatorname{Hom}(U\{n\},U\{m\})=0$ for $n\ne m$, and
$\operatorname{Hom}(U,U)=k$ by Lemma~\ref{lem:chain-endo}.  It
therefore suffices to prove three facts:
(1) $\operatorname{Ext}^i(U\{n\},U\{m\})=0$ for $i\ge3$ and all $n,m$,
(2) $\chi(U\{n\},U\{m\})=2\delta_{n,m}-\delta_{|n-m|,q}$, and
(3) $\operatorname{Ext}^1(U,U)=0$.  Put $G:=S(U)$ and $G_m:=S(U\{m\})=G\{m\}$, coherent $F$-gauges with
$G_m^{\pm\infty}\cong k^d$ and $G\{m\}^r=G^{r+m}$.  Let
$c_1<\dots<c_d=q$ be the positions of the letters $1$ in
$\epsilon(d/q)$ and put $v_j:=c_{d-j}$.  By
Lemma~\ref{lem:regluing-normal-form} and
Example~\ref{ex:sorted-path},
\begin{equation}
 G|_{k^{\mathcal N}}=K_u(q+1)\oplus L(v_1)\oplus\cdots\oplus L(v_{d-1})\oplus K_t(0),
 \label{eq:stable-gauge-shape}
\end{equation}
with $\tau$ carrying the line at $+\infty$ of each summand to the
line at $-\infty$ of the next, and $\ell(G^r)$ is $d-1$ for
$1\le r\le q$ and $d$ otherwise.

Consider the long exact sequence of $\operatorname{Ext}$ groups
induced by \eqref{eq:gluing-fibre} for $G_n$ and $G_m$.  Its terms
$\operatorname{Ext}^i_A(G_n|_{k^{\mathcal N}},G_m|_{k^{\mathcal N}})$ vanish for $i\ge2$, since
$\operatorname{RHom}_A(G_n|_{k^{\mathcal N}},G_m|_{k^{\mathcal N}})$ is a sum of fibres of single maps
by \eqref{eq:stable-gauge-shape} and \eqref{eq:interval-ext}, and
its terms $\operatorname{Ext}^i_W(k^d,k^d)$ vanish for $i\ge2$ since
$W$ is a discrete valuation ring.  So
$\operatorname{Ext}^i(U\{n\},U\{m\})=0$ for $i\ge3$ and
\begin{equation}
 \operatorname{Ext}^2(U\{n\},U\{m\})
 =\operatorname{coker}\bigl(\operatorname{Ext}^1_A(G_n|_{k^{\mathcal N}},G_m|_{k^{\mathcal N}})\to
 \operatorname{Ext}^1_W(k^d,k^d)\bigr).
 \label{eq:ext2-coker}
\end{equation}

The sequence \eqref{eq:gluing-fibre} for $G_n$ and $G_m$ is one of
$\operatorname{End}(G_m)$-modules, and $\operatorname{End}(G_m)=k$
acts on its last two terms level by level through scalars twisted
by powers of $\sigma$ (Lemma~\ref{lem:extension-uniqueness}), so
there the dimension over $k$ is the $W$-length.  Hence
$\chi(U\{n\},U\{m\})=\chi_A(G_n|_{k^{\mathcal N}},G_m|_{k^{\mathcal N}})-\chi_W(k^d,k^d)$.
The second term vanishes, since $\operatorname{Hom}_W(k^d,k^d)$ and
$\operatorname{Ext}^1_W(k^d,k^d)$ both have length $d^2$.  By
\eqref{eq:stable-gauge-shape}, $G_n|_{k^{\mathcal N}}$ is a direct
sum of $L$-summands and of the two $K$-summands $K_u(q+1-n)$ and
$K_t(-n)$.  Since $\chi_A(L(a),N)=0$ for every $N$ by
\eqref{eq:interval-ext},
\[
 \chi(U\{n\},U\{m\})=\chi_A(K_u(q+1-n),G_m|_{k^{\mathcal N}})
 +\chi_A(K_t(-n),G_m|_{k^{\mathcal N}}).
\]
The levels of $G_m$ of length $d-1$ are those in $[1-m,q-m]$, the
others having length $d$, so $\ell(G_m^a)-\ell(G_m^{a-1})$ is $1$
for $a-1=q-m$, is $-1$ for $a=1-m$, and vanishes otherwise.  With
$a=q+1-n$ and, symmetrically, $a=-n$, \eqref{eq:interval-ext} gives
\[
\begin{aligned}
 \chi_A(K_u(q+1-n),G_m|_{k^{\mathcal N}})&=\ell(G_m^{q+1-n})-\ell(G_m^{q-n})=\delta_{n,m}-\delta_{n-m,q},\\
 \chi_A(K_t(-n),G_m|_{k^{\mathcal N}})&=\ell(G_m^{-n})-\ell(G_m^{1-n})=\delta_{n,m}-\delta_{n-m,-q},
\end{aligned}
\]
and $\chi(U\{n\},U\{m\})=2\delta_{n,m}-\delta_{|n-m|,q}$.

Let $0\to U\to E\to U\to0$ be a nonsplit extension in
$\Delta_{\mathrm{tor}}$, so that $E\in\Domss(d/q)$ has rank $2q$
and $\ell(S(E)^r)$ is $2d-2$ for $1\le r\le q$ and $2d$ otherwise.
Then $E$ is indecomposable.  Indeed, let $E=E_1\oplus E_2$ with
both summands nonzero.  Since $E$ has length two in $\Domss(d/q)$,
each $E_i$ is simple, so the composite $U\to E\to E_i$ is zero or
an isomorphism.  Since
$U\to E$ is nonzero, it is an isomorphism for some $i$, and then
$E\to E_i\cong U$ retracts $U\to E$, so the sequence splits, a
contradiction.  If $pE=0$, then $E$ is a twist of a modified
interval object on $2q$ vertices
(Proposition~\ref{prop:chain-presentation}), whose levels drop by
one exactly on an interval of length $2q$
(Example~\ref{ex:sorted-path}), contrary to the lengths of the
$S(E)^r$.  If $pE\ne0$, then $p$ kills the subobject and the
quotient, so it factors as
$E\twoheadrightarrow U\xrightarrow{\varphi}U\hookrightarrow E$ with
$\varphi\in\operatorname{End}(U)=k$ nonzero.  So $pE=U$, and
$E[p]=U$, since $E[p]\supseteq U$ has rank $q$ and $E[p]/U$ is
finite torsion in the diagonal domino $E/U$.  Every level of $S(E)$
is then a free $W/p^2$-module lifting the corresponding level of
$G$.  By \eqref{eq:stable-gauge-shape}, $u$ vanishes on $G^0$, so
$u(S(E)^0)\subseteq pS(E)^1$, and $tu=p$ gives
$pS(E)^0\subseteq t(pS(E)^1)$.  The left side has length
$\ell(G^0)=d$ and the right side at most $\ell(G^1)=d-1$, a
contradiction.  So $\operatorname{Ext}^1(U,U)=0$.
\end{proof}

We describe the fixed-slope category and thereby separate the two
extension groups in Proposition~\ref{prop:ext-stable}(ii).
Fix
$\lambda=d/q\in(0,1)$ with $\gcd(d,q)=1$, and put
$U:=U_{d/q}$.  For $0\le r<q$, let
$\Domss(\lambda)_r$ be the full subcategory of
$\Domss(\lambda)$ whose composition factors are
$U\{r+aq\}$, $a\in\mathbb Z$.

\begin{theorem}[Structure of fixed-slope categories]
\label{thm:rational-slope-category}
There are exact $\mathbb Z_p$-linear equivalences
\[
 \Domss(\lambda)
 \simeq \prod_{r=0}^{q-1}\Domss(\lambda)_r
 \simeq \Domss(0)^{\,q}.
\]
The first is the canonical decomposition by the residue of the
twist modulo $q$, and the equivalence on the $r$th factor sends
$U\{r+aq\}$ to $U_0\{a\}$.  Moreover,
\[
 \dim_k\operatorname{Ext}^1(U\{n\},U\{m\})=\delta_{|n-m|,q},
 \qquad
 \dim_k\operatorname{Ext}^2(U\{n\},U\{m\})=\delta_{n,m}.
\]
\end{theorem}

\begin{proof}
We separate the residue classes, remove the Christoffel markings,
and contract the resulting intervals.

For $n\ne m$, a nonsplit extension
$0\to U\{m\}\to E\to U\{n\}\to0$ is indecomposable, since it has
length two.  Since $p$ annihilates both ends, multiplication by $p$ on $E$
factors as $E\twoheadrightarrow U\{n\}\to U\{m\}\hookrightarrow E$.
The middle map is zero for $n\ne m$, so $pE=0$.
Since $E\in\Domss(\lambda)$ with $\lambda\in(0,1)$,
Proposition~\ref{prop:chain-presentation} now identifies $E$
with a modified interval object.
Its rank is $2q$ and its degree is $2d$, so its interval $J$
has $2q$ vertices and $2d$ markings
(Lemma~\ref{lem:dictionary}).

This interval description gives, by Example~\ref{ex:sorted-path},
\[
 \ell_W(S(E)^i)=2d-\mathbf1_J(i).
\]
On the other hand, the same example gives
$\ell_W(S(U\{v\})^i)=d-\mathbf1_{I_v}(i)$, where
$I_v=[1-v,q-v]$.  Since $S$ takes the extension above to a
short exact sequence of $F$-gauges, additivity of length gives
\[
 \ell_W(S(E)^i)=2d-\mathbf1_{I_n}(i)-\mathbf1_{I_m}(i).
\]
Comparing the two formulas yields
$\mathbf1_{I_n}+\mathbf1_{I_m}=\mathbf1_J$.
The intervals $I_n$ and $I_m$ cannot overlap, since the left
side would then take the value $2$.  Their union is the interval
$J$, so they have no gap and must be adjacent.
If $n>m$, adjacency means $q-n+1=1-m$, hence $n-m=q$.
The case $m>n$ is symmetric, so $|n-m|=q$.

Induction on composition length gives vanishing of
$\operatorname{Hom}$ and $\operatorname{Ext}^1$ between distinct
$\Domss(\lambda)_r$, and hence the canonical product
decomposition.

The Breuil--Kisin twist $M\mapsto M\{-r\}$ identifies
$\Domss(\lambda)_r$ with $\Domss(\lambda)_0$, sending
$U\{r+aq\}$ to $U\{aq\}$, so it suffices to treat $r=0$.
Let $\epsilon$ be the $q$-periodic extension of the Christoffel
word $\epsilon(d/q)$.  By Definition~\ref{def:autoequivalence},
$M\mapsto M(-\epsilon)$ is a diagonal-exact autoequivalence.
The modification by $-\epsilon$ cancels the marking of $U\{aq\}$,
up to an overall Frobenius twist, which preserves
the standard interval gauge up to isomorphism.

Consequently, by Theorem~\ref{thm:integer-slope-category},
$M\mapsto S\bigl(M(-\epsilon)[1]\bigr)$
sends $U\{aq\}$ to the interval gauge
\[
 S\bigl(U\{aq\}(-\epsilon)[1]\bigr)
 \simeq
 \underbrace{k\xleftarrow{\ t=1\ }k\xleftarrow{\ t=1\ }\cdots
 \xleftarrow{\ t=1\ }k}_{q\text{ copies of }k},
\]
supported on the levels $I_{aq}=[1-aq,q-aq]$, with all
$u$-arrows and all levels outside $I_{aq}$ zero.
This functor comes from a derived equivalence.  Both sides are
finite extension closures of these corresponding objects,
so it restricts to an exact equivalence from
$\Domss(\lambda)_0$ onto the extension closure of these interval
gauges.

For a gauge $G$ in this closure, $t_i$ is invertible for
$i\notin q\mathbb Z$, by extension-closedness of this condition.
Set $N^j=G^{qj}$ and define its arrows by
\[
 u=(t_{qj+1}\cdots t_{q(j+1)-1})^{-1}u_{qj},
 \qquad
 t=t_{qj}\cdots t_{q(j+1)-1}.
\]
They satisfy $ut=tu=p$, so $N$ is a graded
$A=W[u,t]/(ut-p)$-module of finite total $W$-length.
Conversely, repeat $N^j$ on $[q(j-1)+1,qj]$, with $t=1$,
$u=p$ internally and the given arrows at the boundaries.
Expanding a composition series places this gauge in the stated
extension closure.  Both constructions act on morphisms
levelwise and are exact.  Starting with $N$ recovers $N$
identically; starting with $G$, the maps
\[
 t_i\cdots t_{qj-1}:G^{qj}\xrightarrow{\sim}G^i,
 \qquad q(j-1)+1\le i\le qj,
\]
with the empty product interpreted as the identity, give a
natural isomorphism from the resulting gauge to $G$.
Compatibility with the internal arrows follows from $ut=tu=p$,
and with the boundary arrows from the definitions above.
Thus we obtain an exact $\mathbb Z_p$-linear equivalence;
composing with the preceding equivalence sends $U\{aq\}$ to
$k$ in degree $1-a$, hence to $U_0\{a\}$ by
Theorem~\ref{thm:integer-slope-category}.  The initial twist
gives the equivalence on every factor.

Finally, extension-closedness of $\Domss(\lambda)$ identifies
its Yoneda $\operatorname{Ext}^1$ with the ambient group.
The equivalences and \eqref{eq:ext-table} give the
$\operatorname{Ext}^1$ formula, with dimensions over the
endomorphism fields.  Proposition~\ref{prop:ext-stable}(ii) then gives
$\operatorname{Ext}^2(U\{n\},U\{m\})=0$ for $n\ne m$, and part~(i)
gives the case $n=m$.
\end{proof}

For example, the prefix $01$ in $M(0101)$ gives an extension
of $U_{1/2}\{-2\}$ by $U_{1/2}$.
For $M(0011)$, Ekedahl's move (Lemma~\ref{lem:normalization})
at the middle edge gives
\[
 M(0011)=(1\leftarrow2\leftarrow3^*\leftarrow4^*)
 \simeq(1\leftarrow2^*\rightarrow3\leftarrow4^*).
\]
No arrow leaves the suffix $[3,4]$ in this presentation, so it
defines a subobject $U_{1/2}\{-2\}$, with quotient $U_{1/2}$
on $[1,2]$.  Thus $M(0011)$ gives an extension in the opposite
direction.  Both extensions are nonsplit, since their middle
terms are indecomposable (Lemma~\ref{lem:chain-endo}).

\begin{remark}
\label{rem:ext-comparison}
Proposition~\ref{prop:ext-stable}(i), the table~\eqref{eq:ext-table},
and integer Nygaard modifications
(Lemma~\ref{lem:nygaard-shift}) show that every stable diagonal
domino is rigid, that is, it has no self-extensions and hence
no first-order deformations.  On a curve of genus $g\ge2$ a stable bundle of rank $r$
has $\operatorname{Ext}^1$ of dimension $r^2(g-1)+1$, and on the
Fargues--Fontaine curve the stable bundles $\mathcal O(\lambda)$ are
rigid with $\operatorname{Ext}^2=0$.
Here $\operatorname{Ext}^2(U_{d/q},U_{d/q})=k$.
By \eqref{eq:gluing-fibre}, the groups $\operatorname{Ext}^i$
between coherent $F$-gauges vanish for $i\ge3$, since $A$ is
regular of dimension $2$ and $W$ of dimension $1$
\cite[Construction~3.3.1]{BhattFgauges}.
Thus $\Coh(k^{\mathrm{Syn}})$ has global dimension $2$.

\end{remark}

\section{Mazur--Ogus varieties}
\label{sec:abelian}

In this final section we place the torsion $F$-gauges studied so
far in a geometric setting: through a Mazur--Ogus d\'evissage, we
extract from an $F$-crystal its largest Hodge--Witt subobject and a
torsion quotient, which is a diagonal domino by
Proposition~\ref{prop:MO-devissage} below, so the theory developed
above applies to it.  As an
application, we show that the stable object $U_{1/2}$ is the middle
diagonal domino of the superspecial abelian threefold, and that
$U_{1/3}$ enters the middle diagonal domino of the superspecial
abelian fourfold, whose structure shows that its slope spectral
sequence does not degenerate at $E_2$.  To our
knowledge, these are the first geometric examples of diagonal
dominoes with non-integer slopes.

\begin{definition}
\label{def:MO-HW}
A coherent complex $M\in D^b_c(R)$ is \emph{Hodge--Witt} if all
$H^j(M)^i$ are finitely generated $W$-modules.  An object
$M\in\Delta$ is \emph{Mazur--Ogus} if
$\operatorname{Hom}_\Delta(M,N)=\operatorname{Hom}_\Delta(N,M)=0$
for all $\mathbf s$-acyclic $N$ \cite[Section~I.4]{ekedahl3}.  A
smooth proper variety $X/k$ is \emph{Mazur--Ogus} if every diagonal
cohomology $\widetilde H^{n}\bigl(R\Gamma(X,W\Omega^\bullet)\bigr)$
is a Mazur--Ogus object, equivalently if $H^*_{\mathrm{crys}}(X/W)$
is torsion-free and the Hodge--de Rham spectral sequence degenerates
at $E_1$ \cite[Definition~IV.1.1 and Theorem~IV.1.2]{ekedahl3}.  For
example, abelian varieties are Mazur--Ogus.
\end{definition}

Mazur--Ogus objects admit a particularly clean d\'evissage theorem,
which we recall in Proposition~\ref{prop:MO-devissage} below
together with its translation to $F$-gauges through the functor
$S$.  We begin with the dictionary between vector bundle $F$-gauges
and virtual $F$-crystals.

\begin{definition}[Virtual $F$-crystals]
\label{def:virtual-crystal}
A \emph{virtual $F$-crystal} over $k$ is a pair $(H,\varphi)$ with
$H$ a finite free $W$-module and $\varphi\colon H\to H[1/p]$ an
injective $\sigma$-semilinear map.  Define its \emph{Nygaard
$F$-gauge} $\mathcal N(H,\varphi)$ to be the $F$-gauge with levels
$\mathcal N^r:=\{x\in H:\varphi(x)\in p^rH\}$, with $t$ the inclusion
$\mathcal N^{r+1}\subseteq\mathcal N^{r}$, and with $u$ multiplication
by $p$.  The gluing $\tau$ sends the image in $\mathcal N^{\infty}$
of $x\in\mathcal N^{r}$ to $p^{-r}\varphi(x)\in H$, and this is well
defined: $x$ and $u(x)=px$ represent the same element of
$\mathcal N^{\infty}$, and $p^{-(r+1)}\varphi(px)=p^{-r}\varphi(x)$.
\end{definition}

Along $t$, the levels form the decreasing filtration
$\varphi^{-1}(p^{\bullet}H)$ of $H$, and the $F$-gauge is this filtration
together with $\varphi$.  For a Mazur--Ogus variety, Mazur's theorem
identifies the levels of
$\mathcal N\bigl(H^{n}_{\mathrm{crys}}(X/W),\varphi\bigr)$ with the
Nygaard filtration
$\mathrm{Fil}^{\ge\bullet}_{\mathcal N}H^{n}_{\mathrm{crys}}(X/W)$
\cite[Section~3.5]{BhattFgauges}.

Ekedahl constructs the Nygaard $F$-gauge as the right adjoint of
$G\mapsto G^{-\infty}$ on torsion-free $F$-gauges and calls it the
Hodge $F$-gauge structure \cite[Section~II.2]{ekedahl3}.  In modern
terms, vector bundle $F$-gauges are equivalent to virtual
$F$-crystals via $G\mapsto G^{-\infty}$
\cite[Proposition~4.3.1]{BhattFgauges}, with inverse the Nygaard
$F$-gauge.  The precise recovery statement we need is the following.

\begin{lemma}[{\cite[Proposition~III.4.1 and Theorem~III.4.2]{ekedahl3}}]
\label{lem:ekedahl-reconstruction}
Let $M\in\Delta$ be Mazur--Ogus.  Then $\mathbf s(M)$ is a finite free $W$-module
in degree $0$, the gluing of $S(M)$ makes
$S(M)^{-\infty}=\mathbf s(M)$ a virtual $F$-crystal
$(\mathbf s(M),\varphi)$, and
$S(M)=\mathcal N\bigl(\mathbf s(M),\varphi\bigr)$.
\end{lemma}

In order to describe the Nygaard $F$-gauge explicitly, we normalize
$\varphi$.  Since $\sigma$ is bijective, $\varphi(H)$ is a
$W$-lattice in $H[1/p]$, so there are $W$-bases $(e_j)$, $(f_j)$ of
$H$ and integers $a_1\le\dots\le a_d$ with
$\varphi(e_j)=p^{a_j}f_j$.  We call the $a_j$ the \emph{elementary
divisors} of $\varphi$.

\begin{lemma}
\label{lem:nygaard-gauge}
$\mathcal N(H,\varphi)$ is a coherent $F$-gauge of level $[a_1,a_d]$
with $\mathcal N^{-\infty}=H$, and
$\mathcal N(H,p^{j}\varphi)=\mathcal N(H,\varphi)\{-j\}$.  For a
finite free $V$-complete Dieudonn\'e module $L$ and $i\in\mathbb Z$,
$S(L)\{-i\}=\mathcal N(L,p^{i}F)$, with elementary divisors in
$\{i,i+1\}$ and slopes in $[i,i+1)$.
\end{lemma}

\begin{proof}
An element $x=\sum_jw_je_j$ of $H$ has
$\varphi(x)=\sum_j\sigma(w_j)p^{a_j}f_j$, which lies in $p^{r}H$
exactly when $w_j\in p^{\max(0,r-a_j)}W$ for all $j$, since $(f_j)$
is a basis and $\sigma$ preserves valuations.
The single basis $(e_j)$ therefore presents all levels at once,
$\mathcal N^{r}=\bigoplus_j p^{\max(0,r-a_j)}We_j$, so $ut=tu=p$,
$t$ is bijective for $r\le a_1$, and $u$ is bijective for $r\ge a_d$.
The levels are finite free, so the $F$-gauge is derived $p$-complete and
coherent.  For $r\ge a_d$ one computes
$\varphi(\mathcal N^{r})=p^{r}H$, so $\tau$ is bijective from
$\mathcal N^\infty\cong\mathcal N^{r}$ onto $H=\mathcal N^{-\infty}$.  The twist rule is immediate from
$\{x:p^{j}\varphi(x)\in p^{r}H\}=\mathcal N^{r-j}$.

For the second assertion, $V$ is injective on $L$, as $Vx=0$ gives
$px=FVx=0$.  The maps $t$ in the display of
Example~\ref{ex:equivalence-examples}(i) therefore embed the levels
of $S(L)$ into $S(L)^{-\infty}=L$, with image
$\{x\in L:F(x)\in p^{r}L\}$ in level $r$ (equal to $L$ for $r\le0$
and to $p^{r-1}VL$ for $r\ge1$, as $F$ is injective and $FV=p$),
with $u$ acting as multiplication by $p$, and with the same gluing.
So $S(L)=\mathcal N(L,F)$, and twisting gives
$S(L)\{-i\}=\mathcal N(L,p^{i}F)$.  Finally the containments
$pL\subseteq F(L)\subseteq L$ bound the elementary divisors of $F$
by $0$ and $1$, and $V$-completeness makes $V$ topologically
nilpotent, so the slopes of $F$ lie in $[0,1)$.  Multiplying by
$p^{i}$ shifts divisors and slopes by $i$.
\end{proof}

We call an $F$-gauge \emph{Hodge--Witt} if it is a finite direct sum
$\bigoplus_i S(L_i)\{-i\}$ with each $L_i$ a finite free
$V$-complete Dieudonn\'e module.

\begin{proposition}[Mazur--Ogus d\'evissage]
\label{prop:MO-devissage}
Let $X/k$ be a smooth proper Mazur--Ogus variety and let
$(H,\varphi)$ be $H^{n}_{\mathrm{crys}}(X/W)$ with its Frobenius.
Each diagonal cohomology
$\widetilde H^n=\widetilde H^n\bigl(R\Gamma(X,W\Omega^\bullet)\bigr)$
has a largest Hodge--Witt subobject $HW(\widetilde H^n)$, which is
finite free, and in the short exact sequence
\[
 0\longrightarrow HW(\widetilde H^n)\longrightarrow\widetilde H^n
 \longrightarrow D_n\longrightarrow0
\]
the quotient $D_n$ is a diagonal domino with $F^{>0}_{\HN}D_n=D_n$.
The functor $S$ carries this
sequence to a short exact sequence of coherent $F$-gauges
\[
0\longrightarrow S\bigl(HW(\widetilde H^{n})\bigr)
\longrightarrow\mathcal N\bigl(H^{n}_{\mathrm{crys}}(X/W),\varphi\bigr)
\longrightarrow S(D_n)\longrightarrow0,
\]
in which $S\bigl(HW(\widetilde H^{n})\bigr)=HW\bigl(\mathcal N(H,\varphi)\bigr)$,
the largest Hodge--Witt subobject of the middle term.
\end{proposition}

\begin{proof}
The d\'evissage on the diagonal complexes side is Ekedahl's
\cite[Theorem~III.4.6]{ekedahl3}.

We first identify the middle term.  Every $\widetilde H^{j}$ is
Mazur--Ogus (Definition~\ref{def:MO-HW}), so each
$\mathbf s(\widetilde H^{j})$ is concentrated in degree $0$
(Lemma~\ref{lem:ekedahl-reconstruction}).  The spectral sequence
of the diagonal filtration therefore degenerates and identifies
$\mathbf s(\widetilde H^{n})$ with $H^{n}_{\mathrm{crys}}(X/W)$, on
which the gluing acts through $p^{j}F$ on the $W\Omega^{j}$ part,
that is as the crystalline Frobenius
\cite[Theorem~II.1.4]{Illusie79}.
Lemma~\ref{lem:ekedahl-reconstruction} now gives
$S(\widetilde H^{n})=\mathcal N(H,\varphi)$.

Next, the two outer terms lie in the heart, so the d\'evissage
becomes a short exact sequence in $\Coh(k^{\mathrm{Syn}})$.  Indeed
$S(HW(\widetilde H^{n}))$ is a Hodge--Witt $F$-gauge, since a
Hodge--Witt object of $\Delta$ splits into Breuil--Kisin twists of
finite free Dieudonn\'e modules \cite[Section~I.2]{ekedahl3}, and
every HN slope of $D_n$ is positive, so $S(D_n)$ lies in the
$F$-gauge heart (Theorem~\ref{thm:phase-rotation}).

Finally we prove the maximality of $S(HW(\widetilde H^{n}))$.  Let
$\mathcal H$ be a Hodge--Witt subobject of
$\mathcal N(H,\varphi)$, with quotient $\mathcal Q$.  The kernel of
$S^{-1}\mathcal H\to\widetilde H^{n}$ is
$\widetilde H^{-1}(S^{-1}\mathcal Q)$, torsion by the tilt
(Proposition~\ref{prop:ekedahl-heart-hrs}), hence zero in the
$p$-torsion-free $S^{-1}\mathcal H$.  So $S^{-1}\mathcal H$ is a
Hodge--Witt subobject of $\widetilde H^{n}$, and the maximality of
$HW(\widetilde H^{n})$ gives
$\mathcal H\subseteq S(HW(\widetilde H^{n}))$.
\end{proof}

It remains to compute the largest Hodge--Witt subobject from the
$F$-crystal.  Fix now a virtual $F$-crystal $(H,\varphi)$.  Let
$H[1/p]=\bigoplus_\lambda V_\lambda$ be the isoclinic decomposition,
which descends from $\bar k$ to the perfect field $k$ \cite{Manin63}.
Put
\[
 V_{[i,i+1)}:=\bigoplus_{i\le\lambda<i+1}V_\lambda,\qquad
 H_{[i,i+1)}:=H\cap V_{[i,i+1)},\qquad
 F:=p^{-i}\varphi\ \text{on }V_{[i,i+1)},
\]
so that $(H_{[i,i+1)},F)$ is a virtual $F$-crystal with slopes in
$[0,1)$.
Each $H_{[i,i+1)}$ is a lattice in $V_{[i,i+1)}$, since $p^{N}x\in H$
for $x\in V_{[i,i+1)}$ and $N\gg0$, so the sublattice
$\bigoplus_iH_{[i,i+1)}$ of $H$ has finite colength.  It is in
general not all of $H$: for $\varphi(e_1)=e_1$,
$\varphi(e_2)=pe_2$, and $H=We_1+Wp^{-1}(e_1+e_2)$, the two
intersections are $We_1$ and $We_2$, of index $p$ in $H$.

\begin{proposition}
\label{prop:crystal-devissage}
\begin{enumerate}
\item Each $H_{[i,i+1)}$ has a unique maximal sublattice $N$ with
  $pN\subseteq F(N)\subseteq N$, which we denote by
  $HW(H_{[i,i+1)})$.
\item Put $L:=\bigoplus_iHW(H_{[i,i+1)})$.  Then $\ell_W(H/L)<\infty$
  and $\mathcal N(L,\varphi|_L)=HW(\mathcal N(H,\varphi))$.
\end{enumerate}
\end{proposition}

\begin{proof}
For (i), a $W$-submodule $N\subseteq H_{[i,i+1)}$ satisfies
$pN\subseteq F(N)\subseteq N$ exactly when it is stable under $F$
and under $V:=pF^{-1}$, so the largest one is the intersection
\[
 HW\bigl(H_{[i,i+1)}\bigr)=\bigcap_{r,s\ge0}V^{-r}F^{-s}\bigl(H_{[i,i+1)}\bigr),
\]
stable under $F$ and $V$ since $V^{r}F^{s}(Fx)=V^{r}F^{s+1}(x)$ and
$V^{r}F^{s}(Vx)=V^{r+1}F^{s}(x)$.  To show that it is a lattice: the slopes of $F$ and of $V$ on
$V_{[i,i+1)}$ are nonnegative, so
$M:=\sum_{r,s\ge0}V^{r}F^{s}H_{[i,i+1)}$ is bounded and hence a
lattice.  It is stable under $F$ and $V$, hence there exists
$m\ge0$ such that $p^{m}M$ is a sub-Dieudonn\'e module of
$H_{[i,i+1)}$, and by maximality $p^{m}M$ is contained in
$HW(H_{[i,i+1)})$.  This shows that $HW(H_{[i,i+1)})$ is a lattice
as well, and the infinite intersection terminates in finitely many
steps.

For (ii), each $HW(H_{[i,i+1)})$ is a lattice in $V_{[i,i+1)}$ and
$\bigoplus_iV_{[i,i+1)}=H[1/p]$, so $L$ is a lattice contained in
$H$ and $\ell_W(H/L)<\infty$.  The summands $V_{[i,i+1)}$ are
$\varphi$-stable, so
\[
 \mathcal N(L,\varphi|_L)
 =\bigoplus_i\mathcal N\bigl(HW(H_{[i,i+1)}),p^{i}F\bigr)
 =\bigoplus_iS(L_i)\{-i\},
\]
where the Dieudonn\'e module
$L_i:=(HW(H_{[i,i+1)}),F,pF^{-1})$ is $V$-complete because the
slopes of $V$ are positive (Lemma~\ref{lem:nygaard-gauge}).  The
levelwise inclusions
$\{x\in L:\varphi(x)\in p^{r}L\}\subseteq\{x\in H:\varphi(x)\in p^{r}H\}$
are compatible with $u$, $t$, and $\tau$, so
$\mathcal N(L,\varphi|_L)$ is a Hodge--Witt subobject.

For its maximality, let $\mathcal H\simeq\bigoplus_iS(L'_i)\{-i\}$
be a Hodge--Witt subobject of $\mathcal N(H,\varphi)$.  Since the
levels of $\mathcal N(H,\varphi)$ are torsion free with injective
$t$, the inclusion embeds $L':=\mathcal H^{-\infty}=\bigoplus_iL'_i$
into $H$ compatibly with the gluings, identifying $p^{-i}\varphi$
with the $F$ of $L'_i$.  By Lemma~\ref{lem:nygaard-gauge}, $L'_i$ is
then an $F$- and $V$-stable submodule of $H_{[i,i+1)}$, so
$L'_i\subseteq HW(H_{[i,i+1)})$ and $L'\subseteq L$, and the level
$\mathcal H^{r}$ is carried to
$\{x\in L':\varphi(x)\in p^{r}L'\}\subseteq\mathcal N(L,\varphi|_L)^{r}$.
So $\mathcal H\subseteq\mathcal N(L,\varphi|_L)$, that is
$\mathcal N(L,\varphi|_L)=HW(\mathcal N(H,\varphi))$.

\end{proof}

In the situation of Proposition~\ref{prop:MO-devissage}, both
$\mathcal N(L,\varphi|_L)$ and $S(HW(\widetilde H^{n}))$ equal
$HW(\mathcal N(H,\varphi))$.  Taking quotients gives
$S(D_n)=\mathcal N(H,\varphi)/\mathcal N(L,\varphi|_L)$.  This
quotient is a coherent torsion $F$-gauge: the functor
$(-)^{-\infty}$ is exact as a filtered colimit, so its
$-\infty$-level is $H/L$, of finite length, and it is killed by a
power of $p$ (proof of Corollary~\ref{cor:torsion-equivalence}).

\begin{example}[Artin invariants]
\label{ex:artin-invariant}
Let $k$ be algebraically closed and let $X/k$ be a supersingular
K3 surface or a supersingular abelian surface, with Artin invariant
$\sigma_0$ \cite{Artin74}.  The only
nonzero diagonal domino is $D_2\cong U_{\sigma_0}$.  Indeed the
largest Hodge--Witt subobject of $H^2_{\mathrm{crys}}$ is
$\operatorname{NS}(X)\otimes_{\mathbb Z}W$, of colength $\sigma_0$
\cite{Ogus79}.  The Hodge number $h^{0,2}=1$ leaves a single gap in
the levels of Lemma~\ref{lem:nygaard-gauge}, so the rank is one, and
all the information is the degree, equal to the slope
$\muD(D_2)=\sigma_0$.  It takes the values $1,2,\dots,10$ for K3
surfaces and $1,2$ for abelian surfaces, with $\sigma_0=1$ precisely
in the superspecial case.
\end{example}

Beyond dimension $2$ we consider the superspecial abelian
varieties $X=E^{g}$, one in every dimension $g$, and their middle
diagonal dominoes $D_g$.  Proposition~\ref{prop:middle-domino-slope}
computes the rank and the degree of $D_g$ from the levels of
$S(D_g)$ through the phase rotation.  The slope lies strictly
between $0$ and $1$ as soon as $g\ge3$.  Corollary~\ref{thm:superspecial-threefold} and
Theorem~\ref{thm:superspecial-fourfold} then identify $D_3$ and
$D_4$ themselves.

\begin{proposition}
\label{prop:middle-domino-slope}
Let $k$ be algebraically closed, let $E$ be a supersingular elliptic
curve, and write $D_g$ for the diagonal domino part of
Proposition~\ref{prop:MO-devissage} in the middle degree $n=g$ of
$X=E^{g}$.  Then
\[
 \rkD(D_g)=\sum_{2a<g}
 \Bigl(\Bigl\lfloor\frac g2\Bigr\rfloor-a\Bigr)
 \Bigl(\Bigl\lceil\frac g2\Bigr\rceil-a\Bigr)\binom ga^{2},
 \qquad
 \degD(D_g)=\sum_{2a<g}
 \Bigl(\Bigl\lfloor\frac g2\Bigr\rfloor-a\Bigr)\binom ga^{2}.
\]
\end{proposition}

\begin{proof}
The Dieudonn\'e module of $E$ is $M=R^0/R^0(F-V)$, with basis
$x_0,x_1$ and $\varphi x_0=x_1$, $\varphi x_1=px_0$, so
$H^1_{\mathrm{crys}}$ is $M^{\oplus g}$, with basis
$x_0^{(i)},x_1^{(i)}$ in the $i$-th summand.  By crystalline
K\"unneth \cite{Illusie83}, $H:=H^{g}_{\mathrm{crys}}(X/W)$ is
$\wedge^{g}H^1$, isoclinic of slope $g/2$.  A basis is given by the
monomials $y_{A,B}$ wedging $x_1^{(i)}$ for $i\in A$ and $x_0^{(i)}$
for $i\in B$, over the pairs of subsets $A,B$ of
$\{1,\dots,g\}$ with $|A|=a$ and $|B|=g-a$, so there are
$\binom ga^{2}$ monomials of \emph{valuation} $a$.  Frobenius gives
$\varphi\,y_{A,B}=\pm p^{a}y_{B,A}$, so $H$ splits into
$\varphi$-stable lines with $A=B$, two-dimensional summands
with $A\ne B$ and $|A|=|B|=g/2$, and the remaining summands
$\langle y,y'\rangle=\langle y_{A,B},y_{B,A}\rangle$ on which
$\varphi$ acts by the matrix
\[
 \begin{pmatrix}0&\pm p^{b}\\ \pm p^{a}&0\end{pmatrix},
 \qquad b:=g-a,\quad 2a<g,
\]
that is $\varphi y=\pm p^{a}y'$ and $\varphi y'=\pm p^{b}y$, which
we record as the \emph{$\varphi$-valuations}
$(a,b)$.

All slopes equal $g/2$, so $H=H_{[v,v+1)}$ with
$v:=\lfloor g/2\rfloor$ and $F=p^{-v}\varphi$.  Projections to the
summands commute with $F$, so the maximal sublattice $L$ of
Proposition~\ref{prop:crystal-devissage} decomposes summandwise, and
the lines and the summands with $|A|=|B|=g/2$ have bijective
$F$, hence are Hodge--Witt and contribute nothing to $D_g$.
On a summand $\langle y,y'\rangle$ of valuations $(a,b)$ put
$s:=v-a=\lfloor g/2\rfloor-a$, a nonnegative integer since $2a<g$,
which vanishes exactly for $a=\lfloor g/2\rfloor$, when the summand
is Hodge--Witt.  Every $x$
with $F(x)$ integral lies in the sublattice
$\langle p^{s}y,y'\rangle$, which satisfies the condition of
Proposition~\ref{prop:crystal-devissage}(i) since
$F(p^{s}y)=\pm y'$ and $F(y')=\pm p^{g-2v}\cdot p^{s}y$.  So
$L\cap\langle y,y'\rangle=\langle p^{s}y,y'\rangle$, of colength
$\ell_W\bigl(\langle y,y'\rangle/\langle p^{s}y,y'\rangle\bigr)=s$
in the summand.
As in Lemma~\ref{lem:nygaard-gauge}, the Nygaard $F$-gauge of a
rank-one virtual $F$-crystal with elementary divisor $c$ is
\[
 \cdots\fgpair{p}{1}W\fgpair{p}{1}
 \underset{\scriptstyle c}{W}
 \fgpair{1}{p}W\fgpair{1}{p}\cdots,
\]
drawn with the levels decreasing to the right as in
Example~\ref{ex:equivalence-examples}, with the level $c$ marked.
So $\mathcal N(\langle y,y'\rangle,\varphi)$ consists of two such
rows, $y$ with $c=a$ and $y'$ with $c=b$, and $\tau$ carries the
line at $+\infty$ of each row to the line at $-\infty$ of the
other, since $p^{-r}\varphi(p^{r-a}y)=\pm y'$ and
$p^{-r}\varphi(p^{r-b}y')=\pm y$.  Likewise
$\mathcal N(\langle p^{s}y,y'\rangle,\varphi)$ consists of the rows
$p^{s}y$ with $c=a+s$ and $y'$ with $c=b-s$.  The inclusion
$\mathcal N(\langle p^{s}y,y'\rangle,\varphi)\subseteq
\mathcal N(\langle y,y'\rangle,\varphi)$ is levelwise, and on the
generators of the row $y$ it is multiplication by
$p^{s},p^{s-1},\dots,1$ at the levels $a,a+1,\dots,a+s$ and
constant on either side:
\[
\begin{tikzcd}[column sep=2.2em, row sep=1.3em]
 \cdots \arrow[r, shift left, "p"]
 & \underset{\scriptstyle a+s}{W} \arrow[l, shift left, "1"] \arrow[r, shift left, "p"]
 & \underset{\scriptstyle a+s-1}{W} \arrow[l, shift left, "1"] \arrow[r, shift left, "p"]
 & \cdots \arrow[l, shift left, "1"] \arrow[r, shift left, "p"]
 & \underset{\scriptstyle a}{W} \arrow[l, shift left, "1"] \arrow[r, shift left, "1"]
 & \cdots \arrow[l, shift left, "p"] \\
 \cdots \arrow[r, shift left, "p"]
 & \underset{\scriptstyle a+s}{W} \arrow[l, shift left, "1"] \arrow[r, shift left, "1"] \arrow[u, "1"]
 & \underset{\scriptstyle a+s-1}{W} \arrow[l, shift left, "p"] \arrow[r, shift left, "1"] \arrow[u, "p"]
 & \cdots \arrow[l, shift left, "p"] \arrow[r, shift left, "1"]
 & \underset{\scriptstyle a}{W} \arrow[l, shift left, "p"] \arrow[r, shift left, "1"] \arrow[u, "p^{s}"]
 & \cdots \arrow[l, shift left, "p"]
\end{tikzcd}
\]
Here the top row is the row $y$ of
$\mathcal N(\langle y,y'\rangle,\varphi)$ and the bottom row is
the row $p^{s}y$ of $\mathcal N(\langle p^{s}y,y'\rangle,\varphi)$,
both generated by $y$ and $p^{s}y$ multiplied by the indicated
powers of $p$.  The arrows to the right are $t$ and those to the
left are $u$, the vertical arrows are the inclusion, and every
square commutes.  On the row $y'$ the inclusion is multiplication by
$1,p,\dots,p^{s}$ at the levels $b-s,\dots,b$ and constant on
either side, and the gluings agree since both are $p^{-r}\varphi$.
Taking quotients row by row, the levels of $S(D_g)$ on this summand
are
\[
\begin{array}{c|ccccc}
 & r\le a & a\le r\le a+s & a+s\le r\le b-s & b-s\le r\le b & r\ge b\\
 \hline
 y & W/p^{s} & W/p^{a+s-r} & 0 & 0 & 0\\
 y' & 0 & 0 & 0 & W/p^{r+s-b} & W/p^{s}
\end{array}
\]
with $u$ the projection and $t$ multiplication by $p$ on the row
$y$, and the reverse on the row $y'$.  The quotient is therefore
\[
\begin{aligned}
 y'&\colon\ \cdots\fgpair{p}{1}
 \underset{\scriptstyle b}{W/p^{s}}
 \fgpair{1}{p}W/p^{s-1}\fgpair{1}{p}\cdots\fgpair{1}{p}W/p
 \fgpair{1}{p}\underset{\scriptstyle b-s}{0},\\
 y&\colon\ \underset{\scriptstyle a+s}{0}
 \fgpair{p}{1}W/p\fgpair{p}{1}\cdots\fgpair{p}{1}W/p^{s-1}
 \fgpair{p}{1}\underset{\scriptstyle a}{W/p^{s}}
 \fgpair{1}{p}\cdots,
\end{aligned}
\]
so its level at $r$ has length $s$ for $r\le a$, descends by unit
steps to $0$ at $a+s$, vanishes on
$[a+s,\,b-s]=[\lfloor g/2\rfloor,\lceil g/2\rceil]$, and ascends
back to $s$ at $b$.

The HN slopes of $D_g$ are positive
(Proposition~\ref{prop:MO-devissage}), so $S(D_g)$ lies in the
$F$-gauge heart and the phase rotation $Z_{\mathsf{FG}}=-iZ_\Delta$
identifies $\rkFG(S(D_g))=\degD(D_g)$ and
$\degFG(S(D_g))=-\rkD(D_g)$ (Theorem~\ref{thm:phase-rotation} and
Proposition~\ref{prop:gauge-euler-invariants}).  Write $G:=S(D_g)$,
with $G^{-\infty}=H/L$.  The first identification gives
$\degD(D_g)=\ell_W(H/L)=\sum_{2a<g}s\binom ga^{2}$, the degree
formula.  In the second, $\chi_W$ of
$\fib(t^{-\infty}\colon G^{r}\to G^{-\infty})$ is
$\ell_W(G^{r})-\ell_W(G^{-\infty})$, so $\rkD(D_g)$ is the sum
$\sum_{r}\bigl(\ell_W(G^{-\infty})-\ell_W(G^{r})\bigr)$ of the drops
of the levels below the limit.  On a summand the drop is $s$ on the
$b-a-2s+1$ levels of
$[a+s,b-s]$ and $1,\dots,s-1$ twice on the two ramps, in total
$s(b-a-s)=s\bigl(\lceil g/2\rceil-a\bigr)$.  Summing over the
summands gives the rank formula.
\end{proof}

For $g=2,3,4,5$ the formulas give
$(\rkD,\degD)(D_g)=(1,1)$, $(2,1)$, $(20,18)$ and $(56,27)$.

\begin{corollary}
\label{thm:superspecial-threefold}
The middle diagonal domino $D_3$ of the superspecial abelian
threefold is the Christoffel domino $U_{1/2}$.
\end{corollary}

\begin{proof}
By Proposition~\ref{prop:middle-domino-slope}, $\muD(D_3)=\tfrac12$.
A graded piece of $\Delta$-rank one is a Breuil--Kisin twist of a
single layer $U_j$ (Lemma~\ref{lem:postnikov}), of integer slope
$j$, positive here by Proposition~\ref{prop:MO-devissage}.  So a
two-step HN filtration of $D_3$ would have degree at least $2$, and
$D_3$ is semistable.  A proper saturated subobject of slope
$\tfrac12$ would have rank one, so $D_3$ is stable, and
Theorem~\ref{thm:christoffel-stability} gives
$D_3\cong U_{1/2}\{a\}$ for some $a\in\mathbb Z$.  The twist shifts
the levels, while the proof of
Proposition~\ref{prop:middle-domino-slope} puts the zero levels of
$S(D_3)$ at $r\in\{1,2\}$, matching those of $S(U_{1/2})$
(Proposition~\ref{prop:one-star-gauge}).  So $a=0$.
\end{proof}

We end by revisiting a theorem of Ekedahl: the slope spectral
sequence of a supersingular abelian fourfold does not degenerate at
$E_2$ \cite{ekedahl1}.  For the superspecial fourfold we recover it
from the structure of $D_4$.

\begin{theorem}
\label{thm:superspecial-fourfold}
Let $k$ be algebraically closed and let $D_4$ be the middle diagonal
domino of the superspecial abelian fourfold.  Then
$D_4\cong U_1\{-1\}^{\oplus16}\oplus Y$, where $Y$ is the nonsplit
extension of $U_{1/3}$ by $U_1\{-1\}$ and
$\operatorname{Ext}^1_\Delta(U_{1/3},U_1\{-1\})=k$.  Its
Harder--Narasimhan filtration has graded pieces
$U_1\{-1\}^{\oplus17}$ and $U_{1/3}$, and its Postnikov layers are
\[
 \Dom^1(D_4)=U_1,\qquad
 \Dom^2(D_4)=U_1^{\oplus16}\oplus U_{0,2},\qquad
 \Dom^3(D_4)=U_{-1},
\]
where $U_{0,2}=\widehat R/\widehat R(F^2,Fd)$ is the nonsplit
extension of $U_0$ by $U_2$ \cite[Example~3.10]{ZhangDominoes}.
\end{theorem}

\begin{proof}
For $g=4$ the summands $\langle y,y'\rangle$ with $s>0$ in the proof
of Proposition~\ref{prop:middle-domino-slope} are sixteen of
$\varphi$-valuations $(1,3)$ and one of valuations $(0,4)$.  On a
summand of
valuations $(1,3)$, where $s=1$, the quotient has every level $k$
with a single zero at $r=2$, which is $S(U_1\{-1\})$ by
Proposition~\ref{prop:one-star-gauge}.  On the summand of valuations
$(0,4)$, where $s=2$, it is the $F$-gauge
\[
 S(Y):=\cdots\fgpair{p}{1}
 \underset{\scriptstyle 4}{W/p^{2}}\fgpair{1}{p}
 \underset{\scriptstyle 3}{W/p}\rightleftarrows
 \underset{\scriptstyle 2}{0}\rightleftarrows
 \underset{\scriptstyle 1}{W/p}\fgpair{p}{1}
 \underset{\scriptstyle 0}{W/p^{2}}\fgpair{1}{p}\cdots,
\]
drawn as in that proof, with the levels marked, the row $y'$ on the
left and the row $y$ on the right.  So
$S(D_4)=S(U_1\{-1\})^{\oplus16}\oplus S(Y)$, and this defines $Y$.

By the description of $u$ and $t$ in that proof, $S(Y)[p]$ has every
level $k$ with a single zero at $r=2$, so $S(Y)[p]\cong S(U_1\{-1\})$,
and $S(Y)/S(Y)[p]$ has zero levels exactly at $r=1,2,3$, so it is
$S(U_{1/3})$ by Proposition~\ref{prop:one-star-gauge}.  Applying
$S^{-1}$ gives $0\to U_1\{-1\}\to Y\to U_{1/3}\to0$, which is
nonsplit, since a direct sum would be killed by $p$ while
$S(Y)^0=W/p^2$.

The rightward word of $U_{1/3}$ is $100$
(Lemma~\ref{lem:normalization}), so
Corollary~\ref{cor:postnikov-layers} gives $H^{0}(U_{1/3})=U_1$,
$H^{-1}(U_{1/3})=U_0(-1)$ and $H^{-2}(U_{1/3})=U_0(-2)$, and the
long exact sequence of ordinary cohomology reads
\[
 0\to H^{-2}(Y)\to U_0(-2)\xrightarrow{\ \partial\ }U_1(-1)\to
 H^{-1}(Y)\to U_0(-1)\to0,
 \qquad H^{0}(Y)=U_1 .
\]
The rightward word $100$ gives
$\tau_{\ge-1}U_{1/3}\cong U_{1/2}$.
For $S(U_{1/2})$ and $S(U_1\{-1\})$, the underlying gauges are
$K_u(3)\oplus K_t(0)$ and $K_u(3)\oplus K_t(1)$.
By \eqref{eq:interval-ext}, the gluing map in
\eqref{eq:gluing-fibre} is surjective on $\operatorname{Hom}$
($k^2\to k$) and an isomorphism on $\operatorname{Ext}^1$
($k\to k$, by restriction at $-\infty$).
Thus $\operatorname{Ext}^1_\Delta(U_{1/2},U_1\{-1\})=0$.
Hence, by the truncation triangle,
the class of $Y$ in
$\operatorname{Ext}^1_\Delta(U_{1/3},U_1\{-1\})
=\operatorname{Hom}(U_{1/3},U_1\{-1\}[1])$ is determined by its
restriction $\partial$ to $\tau_{\le-2}U_{1/3}=U_0(-2)[2]$, an
element of $\operatorname{Hom}_R(U_0(-2),U_1(-1))=k$, a map
there being a $k$-linear map $U_0^0/VU_0^0\to U_1^{1}[F]$ between
one-dimensional $k$-vector spaces (Definition~\ref{def:domino}).  So $\partial\ne0$ and
$\operatorname{Ext}^1_\Delta(U_{1/3},U_1\{-1\})=k$.  Hence
$H^{-2}(Y)=[VU_0^{0}\to U_0^{1}](-2)=U_{-1}(-2)$, and $H^{-1}(Y)$ is
an extension of $U_0(-1)$ by
$[U_1^{0}\to U_1^{1}/U_1^1[F]](-1)=U_2(-1)$.  The summand $U_1\{-1\}$
has the single layer $\Dom^2(U_1\{-1\})=U_1$, so the layers of
$D_4$ are $\Dom^1(D_4)=H^0(Y)=U_1$,
$\Dom^2(D_4)=U_1^{\oplus16}\oplus H^{-1}(Y)(1)$ and
$\Dom^3(D_4)=H^{-2}(Y)(2)=U_{-1}$.

Since $\operatorname{Ext}^1_R(U_0,U_2)=k$
\cite[Corollary~I.3.7]{IR83}, \cite[Corollary~III.1.5.4(iii)]{ekedahl2},
with all nonzero classes giving isomorphic extensions, it remains
to see that $p\ne0$ on $H^{-1}(Y)$.  Multiplication by $p$ on $Y$
kills $U_1\{-1\}$ and $U_{1/3}$, so it factors as
$Y\twoheadrightarrow U_{1/3}\xrightarrow{\psi}U_1\{-1\}\hookrightarrow Y$
with $\psi\ne0$, as $pS(Y)\ne0$.  On $H^{-1}$ it is the composite
\[
 H^{-1}(Y)\twoheadrightarrow U_0(-1)\xrightarrow{\ H^{-1}(\psi)\ }
 U_1(-1)\longrightarrow H^{-1}(Y),
\]
whose last map has kernel $\partial(U_0(-2))=U_1^1[F]$.
Since $\tau_{\le-2}U_{1/3}$ lies in $D^{\le-2}$, restriction
identifies $\operatorname{Hom}(U_{1/3},U_1(-1)[1])$ with
$\operatorname{Hom}(\tau_{\ge-1}U_{1/3},U_1(-1)[1])$.  The triangle
$U_0(-1)[1]\to\tau_{\ge-1}U_{1/3}\to U_1$ gives an exact sequence
\[
 \operatorname{Ext}^1_R(U_1,U_1(-1))\to
 \operatorname{Hom}(\tau_{\ge-1}U_{1/3},U_1(-1)[1])\xrightarrow{\ H^{-1}\ }
 \operatorname{Hom}_R(U_0(-1),U_1(-1)),
\]
whose first term is $\operatorname{Hom}(U_1\{0\},U_1\{-1\})=0$
by the table~\eqref{eq:ext-table}.  So $H^{-1}(\psi)\ne0$, and its image is not inside $U_1^1[F]$,
since an $R$-linear map $U_0\to U_1$ with image in $U_1^1[F]$ kills
$U_0^0$ and $F(U_0^1)=U_0^1$.  Thus $p\ne0$ on $H^{-1}(Y)$, and
$H^{-1}(Y)=U_{0,2}(-1)$.

Finally $U_1\{-1\}^{\oplus17}\subseteq D_4$ is semistable of slope
$1$ with quotient $U_{1/3}$ of slope $\tfrac13$, which is the
Harder--Narasimhan filtration.
\end{proof}

\begin{corollary}
\label{cor:no-e2-degeneration}
The slope spectral sequence of the superspecial abelian fourfold
$X=E^4$ does not degenerate at $E_2$.
\end{corollary}

\begin{proof}
Put $C:=R\Gamma(X,W\Omega_X^\bullet)$ and $M:=\widetilde H^4(C)$.
By \cite[Theorem~IV.1.2(iv)]{ekedahl3}, $M[-4]$ is a direct summand
of $C$. Since $H^4_{\mathrm{crys}}(X/W)[1/p]$ is isoclinic of slope $2$,
Proposition~\ref{prop:crystal-devissage} gives $HW(M)=L\{-2\}$
for a finite free Dieudonn\'e module $L$.
Taking ordinary cohomology of the d\'evissage in
Proposition~\ref{prop:MO-devissage} and using
Theorem~\ref{thm:superspecial-fourfold} gives
\[
 0\longrightarrow L(-2)\longrightarrow H^{-2}(M)
 \longrightarrow U_{-1}(-2)\longrightarrow0.
\]
Since $L(-2)$ is concentrated in internal degree $2$, we obtain
$E_2^{3,2}(M[-4])\cong H^1(\mathbf s(U_{-1}))=k$
by Example~\ref{ex:Uj-s-torsion}.
By Lemma~\ref{lem:ekedahl-reconstruction},
$\mathbf s(M[-4])$ is concentrated in degree $4$, so
$E_\infty^{3,2}(M[-4])=0$.  Thus the slope spectral sequence of
$M[-4]$ does not degenerate at $E_2$, and neither does that of $C$,
since $M[-4]$ is a direct summand of $C$.
\end{proof}

\clearpage
\appendix

\section{Stable $F$-gauges of small \texorpdfstring{$\rkD$}{rank}}
\label{app:table}

The table lists the $F$-gauges of the Christoffel dominoes $U_{d/q}$ of
slope $d/q$ for $2\le q\le5$, the stable objects of
Theorem~\ref{thm:christoffel-stability}, drawn by the recipe of
Example~\ref{ex:sorted-path}.  Each dot is a copy of $k$.  A left arrow
is $t$, a right arrow is $u$, and a diagonal arrow is the boundary
component $e_a\mapsto e_{a-1}$.  Both tails are invertible beyond the
displayed levels, and the gluing is the identity on the basis.

\begin{longtable}{ccccc}
  $\muD$ & $\muFG$ & $\epsilon(d/q)$ & level & $S(U_{d/q})$ \\[0.3em]
  \hline
  \noalign{\vskip 0.6em}
  \endhead
  $\tfrac{1}{2}$ & $-2$ & $01$ & $[0,3]$ &
  \begin{tikzcd}[column sep=1.05em, row sep=0.55em, cells={nodes={inner sep=1.4pt}}, ampersand replacement=\&]
   \cdots \& \bullet\arrow[l] \& 0 \& 0 \& \bullet\arrow[r] \& \cdots
  \end{tikzcd} \\
  \noalign{\vskip 0.8em}
  \noalign{\global\arrayrulewidth=1pt}\hline\noalign{\global\arrayrulewidth=0.4pt}
  \noalign{\vskip 0.8em}
  $\tfrac{1}{3}$ & $-3$ & $001$ & $[0,4]$ &
  \begin{tikzcd}[column sep=1.05em, row sep=0.55em, cells={nodes={inner sep=1.4pt}}, ampersand replacement=\&]
   \cdots \& \bullet\arrow[l] \& 0 \& 0 \& 0 \& \bullet\arrow[r] \& \cdots
  \end{tikzcd} \\
  \noalign{\vskip 0.6em}
  \hline
  \noalign{\vskip 0.6em}
  $\tfrac{2}{3}$ & $-\tfrac{3}{2}$ & $011$ & $[0,4]$ &
  \begin{tikzcd}[column sep=1.05em, row sep=0.55em, cells={nodes={inner sep=1.4pt}}, ampersand replacement=\&]
   \cdots \& \bullet\arrow[l] \& 0 \& 0 \& \bullet\arrow[r] \& \bullet\arrow[r] \& \cdots \\
   \cdots \& \bullet\arrow[l] \& \bullet\arrow[l] \& \bullet\arrow[l]\arrow[ur] \& 0 \& \bullet\arrow[r] \& \cdots
  \end{tikzcd} \\
  \noalign{\vskip 0.8em}
  \noalign{\global\arrayrulewidth=1pt}\hline\noalign{\global\arrayrulewidth=0.4pt}
  \noalign{\vskip 0.8em}
  $\tfrac{1}{4}$ & $-4$ & $0001$ & $[0,5]$ &
  \begin{tikzcd}[column sep=1.05em, row sep=0.55em, cells={nodes={inner sep=1.4pt}}, ampersand replacement=\&]
   \cdots \& \bullet\arrow[l] \& 0 \& 0 \& 0 \& 0 \& \bullet\arrow[r] \& \cdots
  \end{tikzcd} \\
  \noalign{\vskip 0.6em}
  \hline
  \noalign{\vskip 0.6em}
  $\tfrac{3}{4}$ & $-\tfrac{4}{3}$ & $0111$ & $[0,5]$ &
  \begin{tikzcd}[column sep=1.05em, row sep=0.55em, cells={nodes={inner sep=1.4pt}}, ampersand replacement=\&]
   \cdots \& \bullet\arrow[l] \& 0 \& 0 \& \bullet\arrow[r] \& \bullet\arrow[r] \& \bullet\arrow[r] \& \cdots \\
   \cdots \& \bullet\arrow[l] \& \bullet\arrow[l] \& \bullet\arrow[l]\arrow[ur] \& 0 \& \bullet\arrow[r] \& \bullet\arrow[r] \& \cdots \\
   \cdots \& \bullet\arrow[l] \& \bullet\arrow[l] \& \bullet\arrow[l] \& \bullet\arrow[l]\arrow[ur] \& 0 \& \bullet\arrow[r] \& \cdots
  \end{tikzcd} \\
  \noalign{\vskip 0.8em}
  \noalign{\global\arrayrulewidth=1pt}\hline\noalign{\global\arrayrulewidth=0.4pt}
  \noalign{\vskip 0.8em}
  $\tfrac{1}{5}$ & $-5$ & $00001$ & $[0,6]$ &
  \begin{tikzcd}[column sep=1.05em, row sep=0.55em, cells={nodes={inner sep=1.4pt}}, ampersand replacement=\&]
   \cdots \& \bullet\arrow[l] \& 0 \& 0 \& 0 \& 0 \& 0 \& \bullet\arrow[r] \& \cdots
  \end{tikzcd} \\
  \noalign{\vskip 0.6em}
  \hline
  \noalign{\vskip 0.6em}
  $\tfrac{2}{5}$ & $-\tfrac{5}{2}$ & $00101$ & $[0,6]$ &
  \begin{tikzcd}[column sep=1.05em, row sep=0.55em, cells={nodes={inner sep=1.4pt}}, ampersand replacement=\&]
   \cdots \& \bullet\arrow[l] \& 0 \& 0 \& 0 \& \bullet\arrow[r] \& \bullet\arrow[r] \& \bullet\arrow[r] \& \cdots \\
   \cdots \& \bullet\arrow[l] \& \bullet\arrow[l] \& \bullet\arrow[l] \& \bullet\arrow[l]\arrow[ur] \& 0 \& 0 \& \bullet\arrow[r] \& \cdots
  \end{tikzcd} \\
  \noalign{\vskip 0.6em}
  \hline
  \noalign{\vskip 0.6em}
  $\tfrac{3}{5}$ & $-\tfrac{5}{3}$ & $01011$ & $[0,6]$ &
  \begin{tikzcd}[column sep=1.05em, row sep=0.55em, cells={nodes={inner sep=1.4pt}}, ampersand replacement=\&]
   \cdots \& \bullet\arrow[l] \& 0 \& 0 \& \bullet\arrow[r] \& \bullet\arrow[r] \& \bullet\arrow[r] \& \bullet\arrow[r] \& \cdots \\
   \cdots \& \bullet\arrow[l] \& \bullet\arrow[l] \& \bullet\arrow[l]\arrow[ur] \& 0 \& 0 \& \bullet\arrow[r] \& \bullet\arrow[r] \& \cdots \\
   \cdots \& \bullet\arrow[l] \& \bullet\arrow[l] \& \bullet\arrow[l] \& \bullet\arrow[l] \& \bullet\arrow[l]\arrow[ur] \& 0 \& \bullet\arrow[r] \& \cdots
  \end{tikzcd} \\
  \noalign{\vskip 0.6em}
  \hline
  \noalign{\vskip 0.6em}
  $\tfrac{4}{5}$ & $-\tfrac{5}{4}$ & $01111$ & $[0,6]$ &
  \begin{tikzcd}[column sep=1.05em, row sep=0.55em, cells={nodes={inner sep=1.4pt}}, ampersand replacement=\&]
   \cdots \& \bullet\arrow[l] \& 0 \& 0 \& \bullet\arrow[r] \& \bullet\arrow[r] \& \bullet\arrow[r] \& \bullet\arrow[r] \& \cdots \\
   \cdots \& \bullet\arrow[l] \& \bullet\arrow[l] \& \bullet\arrow[l]\arrow[ur] \& 0 \& \bullet\arrow[r] \& \bullet\arrow[r] \& \bullet\arrow[r] \& \cdots \\
   \cdots \& \bullet\arrow[l] \& \bullet\arrow[l] \& \bullet\arrow[l] \& \bullet\arrow[l]\arrow[ur] \& 0 \& \bullet\arrow[r] \& \bullet\arrow[r] \& \cdots \\
   \cdots \& \bullet\arrow[l] \& \bullet\arrow[l] \& \bullet\arrow[l] \& \bullet\arrow[l] \& \bullet\arrow[l]\arrow[ur] \& 0 \& \bullet\arrow[r] \& \cdots
  \end{tikzcd} \\
  \noalign{\vskip 0.3em}
\end{longtable}

\bibliographystyle{alpha}
\bibliography{refs}

@article{ekedahl1,
  author  = {Torsten Ekedahl},
  title   = {On the multiplicative properties of the de {R}ham--{W}itt complex. {I}},
  journal = {Arkiv f{\"o}r Matematik},
  volume  = {22},
  number  = {1--2},
  pages   = {185--239},
  year    = {1984}
}

@article{ekedahl2,
  author  = {Torsten Ekedahl},
  title   = {On the multiplicative properties of the de {R}ham--{W}itt complex. {II}},
  journal = {Arkiv f{\"o}r Matematik},
  volume  = {23},
  number  = {1--2},
  pages   = {53--102},
  year    = {1985}
}

@book{ekedahl3,
  author    = {Torsten Ekedahl},
  title     = {Diagonal complexes and {$F$}-gauge structures},
  publisher = {Hermann},
  address   = {Paris},
  year      = {1986}
}

@article{Illusie79,
  author  = {Luc Illusie},
  title   = {Complexe de de {R}ham--{W}itt et cohomologie cristalline},
  journal = {Annales scientifiques de l'{\'E}cole Normale Sup{\'e}rieure, S{\'e}r. 4},
  volume  = {12},
  number  = {4},
  pages   = {501--661},
  year    = {1979},
  doi     = {10.24033/asens.1374},
  url     = {https://www.numdam.org/articles/10.24033/asens.1374/}
}

@article{IR83,
  author  = {Luc Illusie and Michel Raynaud},
  title   = {Les suites spectrales associ\'ees au complexe de de {R}ham--{Witt}},
  journal = {Publications Math\'ematiques de l'IH\'ES},
  volume  = {57},
  pages   = {73--212},
  year    = {1983},
  publisher = {Institut des Hautes \'Etudes Scientifiques},
  doi     = {10.1007/BF02698774},
  url     = {https://www.numdam.org/item/PMIHES_1983__57__73_0/}
}

@incollection{Illusie83,
  author    = {Luc Illusie},
  title     = {Finiteness, duality, and {K}{\"u}nneth theorems in the cohomology of the de {R}ham--{W}itt complex},
  booktitle = {Algebraic Geometry},
  series    = {Lecture Notes in Mathematics},
  volume    = {1016},
  pages     = {20--72},
  year      = {1983},
  publisher = {Springer},
  address   = {Berlin, Heidelberg},
  doi       = {10.1007/BFb0099957}
}

@article{Manin63,
  author  = {Manin, Yu. I.},
  title   = {The theory of commutative formal groups over fields of finite characteristic},
  journal = {Russian Mathematical Surveys},
  volume  = {18},
  number  = {6},
  pages   = {1--83},
  year    = {1963},
  doi     = {10.1070/RM1963v018n06ABEH001143}
}

@misc{BhattFgauges,
  author       = {Bhatt, Bhargav},
  title        = {Prismatic {$F$}-gauges},
  year         = {2022},
  howpublished = {Lecture notes for {MAT} 549, Princeton University},
  note         = {Available at \url{https://www.math.ias.edu/~bhatt/teaching/mat549f22/lectures.pdf}}
}

@article{Bridgeland07,
  author  = {Bridgeland, Tom},
  title   = {Stability conditions on triangulated categories},
  journal = {Annals of Mathematics},
  series  = {2},
  volume  = {166},
  number  = {2},
  year    = {2007},
  pages   = {317--345}
}

@book{HRS96,
  author    = {Happel, Dieter and Reiten, Idun and Smal\o, Sverre O.},
  title     = {Tilting in abelian categories and quasitilted algebras},
  series    = {Memoirs of the American Mathematical Society},
  volume    = {120},
  note      = {No. 575},
  publisher = {American Mathematical Society},
  address   = {Providence, RI},
  year      = {1996}
}

@book{BLRS09,
  author    = {Berstel, Jean and Lauve, Aaron and Reutenauer, Christophe and Saliola, Franco V.},
  title     = {Combinatorics on Words: Christoffel Words and Repetitions in Words},
  series    = {CRM Monograph Series},
  volume    = {27},
  publisher = {American Mathematical Society},
  address   = {Providence, RI},
  year      = {2009}
}

@article{HarderNarasimhan1975,
  author  = {Harder, G. and Narasimhan, M. S.},
  title   = {On the cohomology groups of moduli spaces of vector bundles on curves},
  journal = {Mathematische Annalen},
  volume  = {212},
  year    = {1975},
  pages   = {215--248},
  doi     = {10.1007/BF01357141}
}

@misc{ZhangDominoes,
  author       = {Zhang, Yuanning},
  title        = {Higher dimensional dominoes in de {R}ham--{W}itt cohomology},
  year         = {2026},
  eprint       = {2607.26323},
  archivePrefix = {arXiv},
  howpublished = {Preprint, arXiv:2607.26323}
}

@article{Bondarenko92,
  author  = {Bondarenko, Vyacheslav M.},
  title   = {Representations of bundles of semi-chains and their
             applications},
  journal = {St. Petersburg Mathematical Journal},
  volume  = {3},
  pages   = {973--996},
  year    = {1992}
}

@article{BurbanDrozd,
  author  = {Burban, Igor and Drozd, Yuriy},
  title   = {Derived categories of nodal algebras},
  journal = {Journal of Algebra},
  volume  = {272},
  number  = {1},
  pages   = {46--94},
  year    = {2004}
}

@article{BTCB,
  author  = {Bennett-Tennenhaus, Raphael and Crawley-Boevey, William},
  title   = {Semilinear clannish algebras},
  journal = {Proceedings of the London Mathematical Society},
  volume  = {129},
  number  = {4},
  pages   = {e12637},
  year    = {2024}
}

@article{FontaineJannsen,
  author  = {Fontaine, Jean-Marc and Jannsen, Uwe},
  title   = {Frobenius gauges and a new theory of $p$-torsion sheaves in
             characteristic $p$},
  journal = {Documenta Mathematica},
  volume  = {26},
  pages   = {65--101},
  year    = {2021}
}

@article{Fargues10,
  author  = {Fargues, Laurent},
  title   = {La filtration de {H}arder--{N}arasimhan des sch\'emas en
             groupes finis et plats},
  journal = {Journal f\"ur die reine und angewandte Mathematik},
  volume  = {645},
  pages   = {1--39},
  year    = {2010}
}

@book{FarguesFontaine18,
  author    = {Fargues, Laurent and Fontaine, Jean-Marc},
  title     = {Courbes et fibr\'es vectoriels en th\'eorie de {H}odge
               $p$-adique},
  series    = {Ast\'erisque},
  volume    = {406},
  publisher = {Soci\'et\'e Math\'ematique de France},
  year      = {2018}
}

@misc{CarmeliFeng25,
  author        = {Carmeli, Shachar and Feng, Tony},
  title         = {Prismatic {S}teenrod operations and arithmetic duality
                   on {B}rauer groups},
  year          = {2025},
  eprint        = {2507.13471},
  archivePrefix = {arXiv},
  howpublished  = {Preprint, arXiv:2507.13471}
}

@article{Artin74,
  author  = {Michael Artin},
  title   = {Supersingular {K3} surfaces},
  journal = {Annales scientifiques de l'{\'E}cole Normale Sup{\'e}rieure, S{\'e}r. 4},
  volume  = {7},
  number  = {4},
  pages   = {543--567},
  year    = {1974},
  doi     = {10.24033/asens.1279},
  url     = {https://www.numdam.org/articles/10.24033/asens.1279/}
}

@incollection{Ogus79,
  author    = {Arthur Ogus},
  title     = {Supersingular {K3} crystals},
  booktitle = {Journ\'ees de G\'eom\'etrie Alg\'ebrique de Rennes (Rennes, 1978), Vol. II},
  series    = {Ast\'erisque},
  volume    = {64},
  pages     = {3--86},
  year      = {1979},
  publisher = {Soci\'et\'e Math\'ematique de France}
}

@article{LevinWangErickson20,
  author  = {Levin, Brandon and Wang-Erickson, Carl},
  title   = {A {H}arder--{N}arasimhan theory for {K}isin modules},
  journal = {Algebraic Geometry},
  volume  = {7},
  number  = {6},
  pages   = {645--695},
  year    = {2020}
}

\end{document}